%% file: Adaptive_Basis_Construction.tex
\documentclass[AMS,STIX1COL]{WileyNJD-v2}
\articletype{Research Article}%

\received{}
\revised{}
\accepted{}
\usepackage{mathtools}
\usepackage{autobreak}
\usepackage{caption}
\usepackage{subcaption}
\usepackage{tikz}
\usepackage{pgfplots}
\usepackage[capitalise]{cleveref}
\crefformat{equation}{(#2#1#3)}
\providecolor{added}{rgb}{0,0,1}
\providecolor{deleted}{rgb}{1,0,0}

\begin{document}

\title{Adaptive Basis Construction and Improved Error Estimation for Parametric Nonlinear Dynamical Systems}

\author{Sridhar Chellappa}

\author{Lihong Feng}

\author{Peter Benner}

\authormark{Chellappa, Feng, Benner : Adaptive Basis Construction and Improved Error Estimation}

\address{\orgname{Max Planck Institute for Dynamics of Complex Technical Systems}, \orgaddress{1~Sandtorstra\ss e, 39106 Magdeburg, \country{Germany}}}

\corres{Sridhar Chellappa. \email{chellappa@mpi-magdeburg.mpg.de}}


\abstract[Summary]{An adaptive scheme to generate reduced-order models for parametric nonlinear dynamical systems is proposed. It aims to automatize the POD-Greedy algorithm combined with empirical interpolation. At each iteration, it is able to adaptively determine the number of the reduced basis vectors and the number of the interpolation basis vectors for basis construction. The proposed technique is able to derive a suitable match between the reduced basis and the interpolation basis vectors, making the generation of a stable, compact and reliable reduced-order model possible. This is achieved by adaptively adding new basis vectors or removing unnecessary ones, at each iteration of the greedy algorithm.  
An efficient output error indicator plays a key role in the adaptive scheme. We also propose an improved output error indicator based on previous work. Upon convergence of the POD-Greedy algorithm, the new error indicator is shown to be sharper than the existing ones, implicating that a more reliable reduced-order model can be constructed. The proposed method is tested on several nonlinear dynamical systems, namely,  the viscous Burgers' equation and two other models from chemical engineering.}



\maketitle

\section{Introduction}\label{sec1}
Large-scale mathematical models have become common in detailed modelling of complex physical and chemical processes. By large-scale, we refer to a model with a large number of degrees of freedom. Very often, these models need to be evaluated repeatedly, for different sets of parameters. To avoid the huge computational burden, model order reduction (MOR) typically seeks a small-scale system, with substantially fewer (typically, at least by $1-2$ orders of magnitude) degrees of freedom, that faithfully approximates the original system with parameter variations. The original large-scale system is referred to as the full-order model (FOM) while the small-scale system is called the reduced-order model (ROM) in the following.
   
MOR for linear systems has been under investigation for several decades and is well-established~\cite{morBauBF14}. However, methods for nonlinear systems are still under active research. Existing MOR methods for linear systems could be extended to weakly-nonlinear systems or systems with structured non-linearities~\cite{morBenB12b,morBenGG18}. For general strong nonlinear systems, the snapshot based methods, e.g., Proper Orthogonal Decomposition (POD), Reduced Basis Method (RBM) are chosen most often. 
The use of POD/RBM is often accompanied by an interpolatory strategy for efficient evaluation of the nonlinear function of the ROM. 
The idea is instead of evaluating the vector of nonlinearities with the full dimension,  
only several elements in the vector are evaluated. The interpolation indices decide which elements should be evaluated.
Several interpolation methods have been proposed, e.g., the Empirical Interpolation Method (EIM)~\cite{morBarMNetal04, morGreMetal07}, the Discrete Empirical Interpolation Method (DEIM)~\cite{morChaS10}. In addition, there exist gappy-POD~\cite{morAmsZCF15, morAmsZF12, morCaretal11,morCarBF11,Eve95} and Missing Point Estimation~\cite{morAstetal08}.
POD with DEIM, i.e. POD-DEIM, was proposed for MOR of non-parametric nonlinear systems~\cite{morGreMetal07, morChaS10}, while RBM combined with EIM (RBM-EIM for short)~\cite{morBarMNetal04, morGreMetal07} is often applied to parametric nonlinear systems.  
 
A standard implementation of POD-DEIM is to separately generate the reduced bases (RB) for the state vector and the interpolation bases for the nonlinear vector using POD and DEIM, respectively. This technique nevertheless does not guarantee that the ROM and the interpolation bases are as small as possible. In the worst case, the ROM might be unstable. To avoid these issues, the authors in~\cite{morFenMB17} introduced a method of adaptively constructing the RB and the DEIM bases. The adaptivity is guided by an \textit{a posteriori} output error indicator for the ROM. Finally, a compact and reliable ROM is obtained. The algorithm in ~\cite{morFenMB17} has several drawbacks. Firstly, it is only applicable to non-parametric systems. Secondly, the adaptive scheme is only {\it one-way}. This means, the interpolation bases can only be extended but cannot be shrunk when necessary. Therefore, the number of the initial basis vectors must be small enough, which could cause more iterations until convergence. Thirdly, the error indicator used in \cite{morFenMB17} is not ideally sharp and could be further improved. 

For parametric systems, direct implementation of RBM-EIM by separately implementing EIM could give rise to similar issues as for non-parametric systems caused by POD-DEIM. In addition, the standard EIM needs FOM simulations at all samples in a training set, which is time consuming. This issue is also pointed out in~\cite{morBenEEetal18, morDavP15}. In~\cite{morDroHO12}, RBM-EIM is implemented such that both the reduced basis and the interpolation basis are updated simultaneously. The update was done by trivially adding a single new member to each of the bases at each iteration step of the POD-Greedy algorithm~\cite{morHaaO08} designed for parametric nonlinear systems. Moreover,  in \cite{morDroHO12}, initial RB and EIM basis vectors must be computed by simulating the FOM several times over a coarse training set. In contrast, the adaptive approach proposed in this paper builds the ROM by starting with a single FOM simulation, which is more effective in the sense of adaptivity.

Recently, adaptive schemes have been proposed in~\cite{morBenEEetal18,morDavP15} for RB and EIM bases construction. 
The authors in~\cite{morDavP15} propose the Simultaneous EIM-RB (SER) method of simultaneously enriching the RB and EIM bases for nonlinear but stationary problems. The goal is to enrich the EIM and RB basis alternately, avoiding the computation of expensive FOM
solutions for all the parameters in a given training set. In the first step, a FOM solution at a randomly chosen parameter is obtained. Based on this, the EIM basis and interpolation point are evaluated. Further, the first RB basis is built by orthogonalizing
the available snapshot from the FOM solution. In the subsequent steps, the EIM and RB basis are enriched alternately, relying only on the approximate solutions computed from the ROM simulation. In essence, the approach requires only an initial FOM simulation at a single parameter.
Since only the ROM is simulated for EIM and RB updates, the snapshots used for EI and RB construction are approximate snapshots. Thus, SER can be considered as an approximate RBM-EIM method.

A progressive EIM (PREIM) method for nonlinear, dynamical (time-dependent) systems is proposed in~\cite{morBenEEetal18}. A special case of the method is a natural extension of the SER method in~\cite{morDavP15}. The PREIM method evaluates the nonlinear function not only at the approximate FOM solution, but also at the high-fidelity FOM solution, whenever available. More precisely, if a new sample of the parameter is selected for RB enrichment, the high-fidelity solution at this sample and the corresponding many time instances need to be computed to enrich the RB bases. The nonlinear function evaluated at the high-fidelity solution of this sample can be readily obtained. For other samples in the training set, the nonlinear function is evaluated at the approximate FOM solutions computed from the ROM. Therefore, this method can be referred to as semi-approximate RBM-EIM. 

In this work, we propose a new adaptive scheme of RB-(D)EIM basis enrichment for parametric nonlinear systems. As compared with the adaptive POD-DEIM in~\cite{morFenMB17}, 

\begin{itemize}

\item we have extended the technique to parametric nonlinear systems. The extension is nontrivial, since our adaptive scheme is constructed based on POD-Greedy~\cite{morHaaO08}, tailored for parametric systems. Unlike the POD-DEIM algorithm in \cite{morFenMB17} for non-parametric systems, or the standard POD-greedy algorithm for parametric systems where (D)EIM needs to be pre-implemented by simulating the FOM at all the parameter samples in the training set before the main algorithm starts, we consider a simultaneous enrichment of RB and (D)EIM in our proposed adaptive POD-Greedy-(D)EIM algorithm starting with a single FOM simulation at the initial parameter.

\item We have made the adaptivity more flexible. The RB and (D)EIM bases enrichment is now a {\it two-way} technique. i.e., it is not only for adaptive basis extension but also for adaptive basis shrinking according to a user-defined error tolerance of the ROM, and the user-given initial basis dimensions.   

\item An improved and sharper output error indicator is derived for both non-parametric and parametric systems, which can be further used in the proposed adaptive schemes. Radial basis function interpolation is applied to compute the parameter-dependent inf-sup constant cheaply and fast, so that singular value decompositions of a full dimensional matrix at all the samples of the parameter are avoided. In \cite{morHesGB15} the authors propose a Kriging interpolation based method to estimate the inf-sup constant and also remark on the connections between Kriging and radial basis interpolation. In this work, we apply the more straightforward method of radial basis interpolation as proposed in \cite{morManN15}, where the interpolation points are adaptively constructed. Adaptivity is nevertheless not considered in \cite{morHesGB15}.

\end{itemize}

As compared with SER, PREIM in~\cite{ morBenEEetal18,morDavP15}, 

\begin{itemize}

\item we do not use the approximate state to evaluate the nonlinear function during the RB and EIM construction so that no extra errors are introduced. 
For DEIM/EIM, we only use the {\it high-fidelity} solutions that are needed for RB enrichment and are computed by FOM simulation at the parameter samples selected iteratively by the POD-Greedy algorithm. The nonlinear function is evaluated only at those high-fidelity FOM solutions that are available {\it for free}. 

\item For both methods, SER and PREIM, a new sample is selected for basis enrichment according to the approximation quality of EIM rather than the ROM quality. Furthermore, the selection process might be time consuming for dynamical systems, since it is done at every parameter sample in the training set and at each time instance corresponding to each sample. Our adaptive scheme is based on an efficient output error indicator for the ROM. At each iteration, a new sample is selected by considering both the EIM error and the RB error. The parameter selection follows POD-Greedy, but the adaptive bases enrichment is based on the separate contributions of the RB and EIM errors to the whole ROM error, thanks to the error indicator. 

\item At each iteration of either SER or PREIM, only one candidate vector is computed for EIM basis enrichment. Our proposed adaptive scheme makes adaptive EIM construction possible, meaning the number of the new EI basis vectors for basis extension could vary at each iteration. Moreover, the dimension of both RB and EI basis space can also be adaptively shrunk according to the error indicator.

\item Beyond those aspects, we have made the RB enrichment adaptive, which is not considered in either SER, PREIM or standard POD-Greedy. This means the number of POD modes to be added to the RB space is adaptively varied at each POD-Greedy iteration.
\end{itemize}
Note that in the process of constructing the ROM, many adaptive techniques are proposed with adaptivity emphases or aspects of adaptivity differing from our proposed work. For example, in \cite{morTaineA14} an adaptive POD-Greedy algorithm is proposed, but with a different sense of adaptivity. There, the algorithm focuses on adaptively enriching the training set. In contrast to {\it offline} adaptive construction of the reduced basis as considered in our work, methods in \cite{morCar15,morEttC19,morPeh18,morPehW15,morZimPW18} focus on {\it online} adaptively updating the RB or the ((D)EIM) interpolation basis. \cite{morYanP19} treats simultaneous training of RB and quadrature points for quadrature computation of the integral occurring in, e.g., finite element descretization.  The approximate quadrature computation is a different way of reducing the nonlinear complexity of EIM or DEIM than considered in our work, though the simultaneous training technique is analogous to the SER method in \cite{morDavP15}. As a result, what is to be simultaneously trained is also different. \cite{morChoet18} replaces the DEIM basis with a different basis, the motivation is to avoid a second SVD associated with the DEIM basis, so as to speed up the offline computation. It is different in the aspect of the simultaneous training considered in our work. We admit that the short review here is by no means exhaustive, and all the adaptive techniques proposed based on interpolatory MOR are not reviewed, though numerous papers exist in the literature, see, e.g., \cite{morGugAB08,morFenB19,morFenKB15}.

\paragraph*{Summary of contributions:}
\begin{itemize}
			
\item In Section 3, we propose an improved primal-dual based output error indicator. The error indicator can be seen as being composed of two parts. The first part is the product of the norms of two residuals: the dual residual and the primal residual. The second part is associated with the approximate state of the dual system. We show that the norm of the dual residual (residual of the dual system) can be further reduced by introducing more efficient solvers for the dual system. The second part of the error indicator can also be reduced by introducing a modified output. Upon convergence of the POD-Greedy algorithm, the new error indicator is shown to be much sharper than the existing ones. The indicator is derived based on a semi-implicit time discretization scheme, with explicit time discretization in the nonlinear part. Extensions to general implicit discretization schemes are possible and should be considered as future work.
				
\item In Section 4, an adaptive process for RB-(D)EIM basis generation is proposed, which we call adaptive POD-Greedy-(D)EIM. Since we improved the existing adaptive POD-DEIM using the improved error indicator and the proposed two-way approach, we also show the improved adaptive POD-DEIM in this section. 

\item To efficiently compute the inf-sup constant for error estimation, we apply the radial basis interpolation approach~\cite{morManN15} to approximate the inf-sup constant for parametric systems. It shows 
good accuracy and speed-up.

\item The proposed ideas are tested on several examples in Section 5 including two examples from chemical engineering - a fluidized bed crystallizer model and a batch chromatographic model. 
\end{itemize}

The remaining part of the paper is organized as follows. In Section 2, the standard algorithms, on which the proposed adaptive algorithm is built, are reviewed: POD, POD-Greedy, EIM and DEIM. 
Since the adaptive algorithm is based on an efficient error indicator, we introduce an improved output error indicator in Section 3. The adaptive POD-Greedy-(D)EIM is proposed in Section 4. Simulation results are presented in Section 5, and conclusions are given in Section 6.

\section{Preliminaries}
\label{sec:prelim}
In this section, we review existing algorithms for model order reduction of nonlinear parametric systems, which are the building blocks of the proposed 
adaptive algorithm. 
Consider a parametric, nonlinear dynamical system of the following form,
\begin{equation}
\begin{aligned}
	E(\mu) \dot{x}(t , \mu) &= A(\mu) x(t , \mu) + f(x(t , \mu),\mu) + B(\mu) u(t),\\
	y(t , \mu) &= C(\mu) x(t , \mu),
\end{aligned}
\label{fom}
\end{equation}
where \begin{itemize}
	\item $\mu \in \mathbb{R}^{d}$ is a vector of parameters in a parameter domain $\mathscr{P} \subset \mathbb{R}^{d}$,
	\item $x(t ,\mu), f(x(t , \mu),\mu) \in \mathbb{R}^{N}$ are the state and the state dependent nonlinear vectors respectively,
	\item $u(t) \in \mathbb{R}^{N_{I}}$ is the input signal,
	\item $y(t, \mu) \in \mathbb{R}^{N_{O}}$ is the output/quantity of interest,
	\item $E(\mu), A(\mu) \in \mathbb{R}^{N \times N}$ are the system matrices,
	\item $B(\mu) \in \mathbb{R}^{N \times N_{I}}, C(\mu) \in \mathbb{R}^{N_{O} \times N}$ are the input and output matrices, respectively.
\end{itemize}
Such systems typically arise from the spatial discretization of parametric partial differential equations (PDEs) through finite difference, finite element or finite volume methods. Since $N$ is typically large, we wish to find a system of reduced order $r \ll N$, that accurately approximates the original system solution. 
\subsection{Parametric Model Order Reduction via projection}
Assume the solution to \cref{fom} lies in a low dimensional subspace, then it is possible to construct a matrix $V \in \mathbb R^{N\times r}$ with low rank ($r \ll N$), whose columns constitute an orthogonal basis that spans this subspace, such that the solution can be approximated by the basis vectors. Then a ROM can be given via Petrov-Galerkin projection,
\begin{equation}
\begin{aligned}
	E_{r}(\mu)\dot{x}_{r}(t , \mu) &= A_{r}(\mu) x_{r}(t , \mu) + f_{r}(Vx_{r}(t , \mu),\mu) + B_{r}(\mu) u(t), \\
	y_{r}(t , \mu) &= C_{r}(\mu) x_{r}(t , \mu),
\end{aligned}
\label{rom}
\end{equation}
where \begin{itemize}
	\item $W  \in \mathbb R^{N\times r}$ is a matrix, whose columns span a test subspace,
	\item $x_{r}(t, \mu) \in \mathbb{R}^{r}$, is the reduced state vector,
	\item $f_{r}(Vx_{r}(t , \mu),\mu) \coloneqq W^{T} f(Vx_{r}(t , \mu),\mu) \in \mathbb{R}^{r}$, is the reduced nonlinear vector, 
	\item $A_{r}(\mu) \coloneqq W^{T} A(\mu) V \in \mathbb{R}^{r \times r}$, $E_{r}(\mu) \coloneqq W^{T} E(\mu) V \in \mathbb{R}^{r \times r}$ are the reduced system matrices,
	\item $B_{r}(\mu) \coloneqq W^{T} B(\mu) \in \mathbb{R}^{r \times N_{I}}, C_{r}(\mu) \coloneqq C(\mu) V \in \mathbb{R}^{N_{O} \times r}$ are the reduced input and output matrices respectively.
\end{itemize}
For parametric nonlinear dynamical systems, snapshot based methods are widely used to construct the basis vectors in $V$ and Galerkin projection is often chosen to construct the ROM, i.e. $W=V$. POD is either used as an independent method or as an intermediate step in POD-Greedy~\cite{morHaaO08} for ROM construction. Throughout the paper, $\|\cdot\|$ refers to the vector $2$-norm or the matrix spectral norm.

\subsection{POD and POD-Greedy}
Proper orthogonal decomposition (POD) is a procedure of constructing an optimal orthonormal basis that is used to approximate a given dataset. Consider a matrix $X \in \mathbb{R}^{N \times K}$, consisting of the given data. Let rank($X$) = $r_{X}$. For any $\ell \leq r_{X}$, POD gives rise to an orthonormal basis $\{u_{i}\}_{i=1}^{\ell}$ satisfying the following optimality criterion,
\begin{equation}
	\begin{aligned}
		\{u_{i}\}_{i=1}^{\ell} = \arg \min_{\{\tilde{u}_{i}\}_{i=1}^{\ell}} \sum_{j=1}^{K} \bigg\| x_{j} - \sum_{k=1}^{l} \langle x_{j}, \tilde u_{k} \rangle \tilde u_{k}  \bigg\| ^{2},\\
		\text{s.t}, \langle \tilde u_{i}, \tilde u_{j} \rangle = \delta_{ij}, 1 \leq i,j \leq \ell.
	\end{aligned}
	\label{pod_opt}
\end{equation}
The basis vectors $\{u_{i}\}_{i=1}^{\ell}$ are called the POD basis, and they are obtained through the Singular Value Decomposition (SVD) of $X$. When POD is used to compute the projection matrix $V$ for MOR of a dynamical system, the system is first simulated to obtain solutions $x(t_{1}),\,x(t_{2}),\,\ldots,\,x(t_{K})$ at (selected) time instances $t_{1},\,t_{2},\,\ldots,\,t_{K}$, which are called {\it snapshots}, then the POD basis of the snapshot matrix $X \coloneqq [x(t_{1}), x(t_{2}), \ldots, x(t_{K})]$ defines the projection matrix $V$. \cref{alg_pod} demonstrates the construction of $V$ through POD.
\renewcommand{\algorithmicrequire}{\textbf{Input:}}
\renewcommand{\algorithmicensure}{\textbf{Output:}}	
\begin{algorithm}[t!]
    \begin{algorithmic}[1]
		\caption{Proper orthogonal decomposition (POD)}
		\label{alg_pod}
       \Require Snapshots $X = \big[ x( t_{1}),  x( t_{2}), \ldots, x( t_{K}) \big] $, tolerance $\epsilon_{\text{POD}}$ (a heuristically chosen small value).
			
       \Ensure POD basis matrix $V$.
			
			 \State Perform $X \xrightarrow{\text{SVD}} U \Sigma W ^{T}$, $\Sigma \coloneqq \begin{bmatrix}
			 D&0\\ 0 & 0\\
			 \end{bmatrix}$, $D \coloneqq diag(\sigma_{1}, \sigma_{2}, \ldots, \sigma_{r_{X}})$.
			
			 \State Find $r$, s.t., $  \sum_{i=r + 1 }^{r_{X}}\sigma_{i} \big/ \sum_{i = 1}^{r_{X}} \sigma_{i} < \epsilon_{\text{POD}}$, $r_{X}$ is the number of non-zero singular values from $\Sigma$, $V = U(: \, , 1: r)$.
    \end{algorithmic}
\end{algorithm}
\renewcommand{\algorithmicrequire}{\textbf{Input:}}
\renewcommand{\algorithmicensure}{\textbf{Output:}}	
\begin{algorithm}[b!]
	\begin{algorithmic}[1]
		\caption{Standard POD-Greedy \cite{morHaaO08}}
		\label{alg_podgreedy}
		\Require Parameter training set $\Xi \subset \mathscr{P}$, tolerance \texttt{tol}.
		\Ensure Basis matrix $V$.			
		\State Initialize. $V = [\,]$, $\mu^{*} \in \Xi$.
		\While $\Delta(\mu^{*})$ $>$ \texttt{tol}
		\State Simulate FOM at $\mu^{*}$ and obtain snapshots, $X =[x(t_{1}, \mu^{*}), x(t_{2}, \mu^{*}), \ldots,  x(t_{K}, \mu^{*})]$.
		\State Update the projection matrix $\bar{X} \coloneqq$ \big($X -\text{Proj}_{\mathcal{V}}(X)$\big) $\xrightarrow{\text{SVD}}$ $U \Sigma W^{T}$, $\mathcal{V}$ is the subspace spanned by $V$. 
		\State $V \leftarrow{} \texttt{orth} \{V, U(: \, , 1 )\}$. 
		\State $\mu^{*} \coloneqq$ arg $\max\limits_{\mu \in \Xi} \Delta(\mu)$.
		\EndWhile
	\end{algorithmic}
\end{algorithm}

For parametric systems, one needs a suitable method to capture the variation in the solution manifold due to the parameter variations. For this purpose, we adopt the POD-Greedy approach~\cite{morHaaO08}. It relies on POD in time domain and a greedy selection in the parameter domain; it is a standard implementation of the RBM for parametric dynamical systems. The details are given in \cref{alg_podgreedy}.
A crucial ingredient of the algorithm is the availability of a cheap and sharp error estimator ${\Delta}$. At each iteration, the error estimator must be evaluated $n_{\text{train}}$ times, where $n_{\text{train}}$ is the cardinality of the parameter training set $\Xi$ \cite{morQuaMN16}.  Therefore, it is important that the error estimator is computed in a rapid and reliable manner.
\subsection{EIM and DEIM}
From \cref{rom}, we observe that the complexity of evaluating the nonlinear term $f_{r}$, still depends on the FOM, since we need to first evaluate it at $V x_{r} \in \mathbb{R}^{N}$ for a given parameter $\mu$. Some interpolation techniques are proposed to reduce this complexity. 
\paragraph*{Empirical interpolation.}
The Empirical Interpolation Method was introduced in the context of the Reduced Basis Method, in order to reduce the complexity in evaluating nonaffine parameter dependence or nonlinear dependance in the ROM \cite{morBarMNetal04,morGreMetal07}. Both the interpolation indices and the interpolation basis vectors are chosen in a greedy manner.
\paragraph*{Discrete empirical interpolation.}
The discrete variant of EIM, i.e. DEIM was introduced in \cite{morChaS10}.
The main differences of DEIM from EIM include two aspects. One is the interpolation basis construction: DEIM uses the {\it pre-computed POD basis} of the nonlinear snapshot matrix via SVD, as the interpolation basis, whereas EIM constructs the interpolation basis iteratively through a greedy algorithm and from the nonlinear snapshots. The other aspect is the interpolation indices selection: DEIM selects the interpolation indices by looking at the distance between the current {\it interpolation basis} vector and its approximation obtained  via interpolation using the previous interpolation basis vectors. However, EIM chooses the interpolation indices based on the distance between the current {\it nonlinear snapshot} vector and its approximation obtained via interpolation using the previous interpolation basis vectors.
\newline \indent
Both EIM and DEIM consider the following approximation for the nonlinear term,
\begin{equation}
	f(x(t , \mu),\mu) \approx U_{f} c(t,\mu),
\label{nonlin}
\end{equation}
where $U_{f}$ is the matrix of interpolation vectors and $c(t,\mu)$ is the vector of the time-parameter dependent coefficients.
The interpolation is achieved by choosing a few interpolation indices where the approximation matches the original function, i.e.
\begin{equation}
	P^{T} f(x(t , \mu),\mu) = P^{T} U_{f} c(t,\mu),
\label{nonlin_interp}
\end{equation}
where $P = [e_{\wp_{1}}, e_{\wp_{2}}, \ldots,  e_{\wp_{\ell}}]$ is a column permutation of the identity matrix, such that the $i^{\text{th}}$ column $e_{\wp_{i}}$, is all zeros except for the $\wp_{i}$-th row where the value is 1. It stores
all the interpolation indices $\wp_{1}, \ldots, \wp_{\ell}$.
We provide both algorithms as \cref{alg_eim} and~\cref{alg_deim} for the sake of completeness.
\renewcommand{\algorithmicrequire}{\textbf{Input:}}
\renewcommand{\algorithmicensure}{\textbf{Output:}}	
\begin{algorithm}[t!]
    \begin{algorithmic}[1]
		\caption{Empirical interpolation method (EIM)\cite{morBarMNetal04}}
		\label{alg_eim}
       \Require Snapshots of the nonlinear vector at a set of parameter samples,
              
       \noindent $F = \big[ f( x(t_{1}, \mu_{1}) \, , \mu_{1}), \ldots, f( x(t_{K}, \mu_{1}) \, , \mu_{1}), \ldots, f( x(t_{1}, \mu_{n_{\text{s}}}) \, , \mu_{n_{\text{s}}}), \ldots, f( x(t_{K}, \mu_{n_{\text{s}}}) \, , \mu_{n_{\text{s}}}) \big] \coloneqq [f_1,\ldots, f_{Kn_s}] \in \mathbb R^{N\times K\cdot n_s} $, where $n_{\text{s}}$ is the total number of parameter samples, maximal iteration steps \texttt{max\_iter}, tolerance $\epsilon_{\text{EI}}$.
			
       \Ensure EIM basis $U_{f} = \big[\zeta_{1}, \zeta_{2}, \ldots, \zeta_{\ell}  \big]$, index matrix $P = [e_{\wp_{1}},  e_{\wp_{2}}, \ldots, e_{\wp_{\ell}}]$.
			
       \State Initialize $U_{f} = [\,]$,  $P = [\,]$, $m = 1$.
			
			 \State Select the snapshot that maximizes the norm. $\eta_{1} = \arg \max \limits_{\substack{f_{i} \\ 1 \leq i \leq K\cdot n_s}} \| f_{i} \| $,
			 the interpolation index is given as the position of the row element of $\eta_{1}$ with maximal magnitude. $\big[ \sim, \wp_{1} \big] = \texttt{max}(| \eta_{1} |)$, where $\eta_{1} = [\eta_{11}, \eta_{12}, \ldots, \eta_{1N}]^{T}$. Here, \texttt{max}() refers to the MATLAB\textsuperscript{$\circledR$} function.
			 \State The first interpolation basis vector, $\zeta_{1} = \eta_{1}/\eta_{1, \wp_{1}}$.
			 
			 \State  Update the interpolation matrix $U_{f} \leftarrow{} \big[U_{f},   \zeta_{1}\big]$ , $P \leftarrow{} [P,  e_{\wp_{1}}]$.
			
			 \While{$m \leq$ $\texttt{max\_iter}$}
							
					\State $m = m + 1$.
					
					\State Form the $m^{\text{th}}$ EIM interpolation for each snapshot vector in $F$: $\mathcal{I}_{m}\big[f_i\big] = U_{f} \big( P^{T} U_{f} \big)^{-1} P^{T} f_i$, $i = 1, \ldots, K \cdot n_{s}$.
					\State Find $f^{*} =  \arg \max \limits_{\substack{f_{i} \\ 1 \leq i \leq K\cdot n_s}} \Big \| f_{i} - \mathcal{I}_{m}[f_{i}]  \Big\|$. Determine residual $\eta_{m} = f^{*} - \mathcal{I}_{m}[f^{*}]$.
								
					\If{$\| \eta_{m} \| <$  $\epsilon_{\text{EI}}$}
						\State $m = m - 1$.
						
						\State stop.
					\EndIf
									
					\State $\big[ \sim, \wp_{m} \big] =\texttt{max}(| \eta_{m} |)$. Set $\zeta_{m} = \eta_{m}/\eta_{m, \wp_{m}}$				
					
					\State $U_{f} \leftarrow{} \big[U_{f},   \zeta_{m}\big]$, $P \leftarrow{} \big[ P,  e_{\wp_{m}}\big]$.
			 \EndWhile
    \end{algorithmic}
\end{algorithm}
\renewcommand{\algorithmicrequire}{\textbf{Input:}}
\renewcommand{\algorithmicensure}{\textbf{Output:}}	
\begin{algorithm}[h!]
    \begin{algorithmic}[1]
		\caption{Discrete empirical interpolation method (DEIM)\cite{morChaS10}}
		\label{alg_deim}
       \Require Snapshots of the nonlinear vector at a set of parameter samples,
       
       \noindent $F = \big[ f( x(t_{1}, \mu_{1}) \, , \mu_{1}), \ldots, f( x(t_{K}, \mu_{1}) \, , \mu_{1}), \ldots, f( x(t_{1}, \mu_{n_{\text{s}}}) \, , \mu_{n_{\text{s}}}), \ldots, f( x(t_{K}, \mu_{n_{\text{s}}}) \, , \mu_{n_{\text{s}}}) \big]$, where $n_{\text{s}}$ is as defined in \cref{alg_eim}, $\epsilon_{\text{POD}}$.
			
       \Ensure DEIM basis $U_{f}$, index matrix $P = [e_{\wp_{1}}, e_{\wp_{2}}, \ldots, e_{\wp_{\ell}}]$.
			
       \State Initialize $U_{f} = [\,]$,  $P = [\,]$.
			
			 \State Perform $F \xrightarrow{\text{SVD}} U \Sigma W ^{T}$, where $U \coloneqq [u_{1}^{f}, u_{2}^{f}, \ldots, u_{r_{F}}^{f}]$, $\Sigma \coloneqq \begin{bmatrix}
			 D_{f}&0\\ 0 & 0\\
			 \end{bmatrix}$, $D_{f} \coloneqq diag(\sigma_{1}^{f}, \sigma_{2}^{f}, \ldots, \sigma_{r_{F}}^{f})$.	
			 \State Find $\ell$, s.t., $  \sum_{i=\ell + 1 }^{r_{F}}\sigma_{i}^{f} \big/ \sum_{i = 1}^{r_{F}} \sigma_{i} ^{f} < \epsilon_{\text{POD}}$, $r_{F}$ is the number of non-zero singular values in $\Sigma$. 
			
			 \State Select the first interpolation index as the position of the row element with maximal magnitude in the first column of $U$: $\wp_{1} = \text{arg} \max \limits_{j \in \{1, 2, \ldots, N \}} | u_{1j}^{f} |$, where $u_{1}^{f} = [u_{11}^{f}, u_{12}^{f}, \ldots, u_{1N}^{f}]^{T}$.
			 \State $U_{f} \leftarrow{} u_{1}^{f}$, $P \leftarrow{} [e_{\wp_{1}}]$.
			
			 \For{$i = 2 $ to $\ell$}
						\State Solve $(P^{T}U_{f}) c = P^{T} u_{i}^{f}$, for $c$.
						\State Form the residual, $r_{i} = u_{i}^{f} - U_{f}c$. 
						\State $\wp_{i} = \text{arg} \max \limits_{j \in \{1, 2, \ldots, N \}} | r_{ij} |$. Here, $r_{i} = [r_{i1}, r_{i2}, \ldots, r_{iN}]^{T}$.
						\State  $U_{f} \leftarrow{} [U_{f}, u_{i}^{f}]$, $P \leftarrow{} [P, e_{\wp_{i}}]$.
			 \EndFor
    \end{algorithmic}
\end{algorithm}
\paragraph*{ROM after interpolation.}
Using either EIM or DEIM, the ROM in \cref{rom} can now be evaluated as,
\begin{equation}
\begin{aligned}
	E_{r}(\mu)\dot{x}_{r}(t , \mu) &= A_{r}(\mu) x_{r}(t , \mu) + V^{T} U_{f} \big( P^{T} U_{f}\big)^{-1} P^{T} f(V x_{r}(t , \mu),\mu) + B_{r}(\mu) u(t), \\
	y_{r}(t , \mu) &= C_{r}(\mu) x_{r}(t , \mu).
\end{aligned}
\label{rom_deim}
\end{equation}
\begin{remark}
The term $V^{T} U_{f} \big( P^{T} U_{f}\big)^{-1}$ in \cref{rom_deim} can be precomputed. In evaluating $P^{T} f(V x_{r}(t, \mu), \mu)$, only a few terms (say, $\ell_{\text{EI}} \ll N$ terms) of the nonlinear vector $f(V x_{r}(t, \mu), \mu)$ needs to be evaluated, thereby removing the bottleneck in computing the nonlinear term of the ROM.
\end{remark}
\begin{remark}
	\cref{alg_podgreedy} can be combined with either EIM or DEIM for MOR of nonlinear systems. The interpolation bases are precomputed before starting the greedy loop. \cref{alg_podgreedy_ei} illustrates the Standard POD-Greedy-(D)EIM algorithm. Clearly, in Step 1, the nonlinearity needs to be evaluated at all samples of $\mu \in \Xi$. This involves computing the full order solutions for all parameters in the training set.
\end{remark}
\renewcommand{\algorithmicrequire}{\textbf{Input:}}
\renewcommand{\algorithmicensure}{\textbf{Output:}}	
\begin{algorithm}[b!]
	\begin{algorithmic}[1]
		\caption{Standard POD-Greedy-(D)EIM \cite{morQuaMN16}}
		\label{alg_podgreedy_ei}
		\Require Parameter training set $\Xi \subset \mathscr{P}$, tolerance \texttt{tol}, maximal number of iterations $\texttt{max\_iter}$, snapshot matrix of the nonlinear vector,
		
		\noindent $F = \big[ f( x(t_{1}, \mu_{1}) \, , \mu_{1}), \ldots, f( x(t_{K}, \mu_{1}) \, , \mu_{1}), \ldots, f( x(t_{1}, \mu_{n_{\text{train}}}) \, , \mu_{n_{\text{train}}}), \ldots, f( x(t_{K}, \mu_{n_{\text{train}}}) \, , \mu_{n_{\text{train}}}) \big] $, recall that $n_{\text{train}}$ is the cardinality of the training set $\Xi$.
		\Ensure Basis matrix $V$.			
		\State (D)EIM interpolation basis calculated using \cref{alg_eim} or \cref{alg_deim}.
		\State Call \cref{alg_podgreedy}, where instead of the ROM in \cref{rom}, the ROM in \cref{rom_deim} is simulated at each iteration.
	\end{algorithmic}
\end{algorithm}

\section{Improved \textit{a posteriori} Output Error Estimation}
\label{sec:aposterr}
Error estimation plays a crucial role in both standard POD-Greedy and the proposed adaptive algorithms. Error estimators for RBM were mostly proposed based on the weak form of the PDE arising from the finite element discretization \cite{morHaaO08, morVeretal03, morGre05}. In~\cite{morZhaFLetal15}, an efficient output error estimator was proposed in the discretized vector space, which makes the error estimator straightforwardly applicable to the already discretized systems. There, the authors propose an~\textit{a posteriori} output error estimator for the ROM in \cref{rom_deim}. It avoids the accumulation of the residual over time, a phenomenon often seen in other error estimation approaches \cite{morHaaO08, morGre05, morHaaO11}, and therefore is often much sharper than the other error estimators. Moreover, the error estimator is applicable to nonlinear dynamical systems. However, we observed that it is still possible to further improve the sharpness and computational efficiency of the error estimator from
different aspects. In the following we briefly review the output error estimation in~\cite{morZhaFLetal15} and propose an improved output error indicator, as well as a more efficient way of computing the error indicator. For the sake of concise notation, we do not explicitly show the parameter dependance of the system matrices $(E, A, B, C)$ and vectors $ x(t, \mu),\, f(x(t, \mu), \mu),\,y(t, \mu)$. The same shall also be  followed by the corresponding dual system matrices and vectors that will be introduced below.
\subsection{Output error estimator from ~\cite{morZhaFLetal15}}
In this subsection, we briefly review the output error estimator in~\cite{morZhaFLetal15}. As in~\cite{morZhaFLetal15}, we consider systems with single output, i.e. $C \in \mathbb{R}^{1 \times N}$ and $y$ is a scalar in \cref{fom}. We will address the error estimation for multiple outputs in \cref{rem:multop}. 
Consider a semi-implicit scheme for the time integration of \cref{fom},
\begin{equation*}
\begin{aligned}
\tilde{E}^k x^{k+1} &= \tilde A^kx^k+ \Delta t_k f(x^{k}) + \Delta t_k B^k u^{k}, \\
y^{k+1} &= C x^{k+1}.
\end{aligned}
\label{disc_primalfom}
\end{equation*}
It is called the primal system. For a sharp estimation of the output error, the following dual system is needed,
\begin{equation*}
\begin{aligned}
(\tilde{E}^k)^{T} x_{\text{du}}^{k+1} &= -C^{T}.
\end{aligned}
\label{disc_dualfom}
\end{equation*}
In general, the system matrices may be time-dependent if the time step $\Delta t_k$ changes over time, and therefore they are associated with the superscript $k$. If we consider constant time steps for simplicity, i.e. $\Delta t_k =\Delta t$, then the superscript $k$ can be removed, and the primal and dual systems can be simplified to
\begin{equation}
\begin{aligned}
\tilde{E} x^{k+1} &= \tilde A+ \Delta t  f(x^{k}) + \Delta t B u^{k}, \\
y^{k+1} &= C x^{k+1}.
\end{aligned}
\label{disc_primalfom}
\end{equation}
and
\begin{equation}
\begin{aligned}
\tilde{E}^{T} x_{\text{du}}&= -C^{T},
\end{aligned}
\label{disc_dualfom}
\end{equation}
respectively.
Note that the dual system becomes a steady system in the simplified case.
For clarity we use the simplified case to describe the error estimator, though it is well defined for the general case, too.

Applying the same time integration scheme to the ROM \cref{rom} results in,
\begin{equation}
\begin{aligned}
E_{r}(\mu) x_{r}^{k+1} &= A_{r}(\mu) x_{r}^{k} + \Delta t f_r(V x_{r}^{k}) + \Delta t B_{r} u^{k}, \\
y_{r}^{k+1} &= C_{r} x_{r}^{k+1}.
\end{aligned}
\label{disc_primalrom}
\end{equation}
It is clear that the time-discrete ROM in \cref{disc_primalrom} is exactly the ROM of the primal system. In~\cite{morZhaFLetal15}, a ROM of the dual system is obtained by Galerkin projection using 
the same projection matrix $V$ as for \cref{rom}, i.e.
\begin{equation}
\begin{aligned}
(V^T\tilde{E}V)^{T} x^{\text{du}}_r&= -V^TC^{T}.
\end{aligned}
\label{disc_dualrom}
\end{equation}
The approximate solutions, $\hat{x}^{k+1} \coloneqq V x_{r}^{k+1}$ to the primal system and $\hat{x}_{\text{du}} \coloneqq V x^{du}_r$ to the dual system,
introduce their residuals, respectively,
\begin{equation}
\begin{aligned}
r_{\text{pr}}^{k+1} &= \tilde{A} \hat{x}^{k} + \Delta t f(\hat{x}^{k}) + \Delta t B u^{k} - \tilde{E} \hat{x}^{k+1},\\
r_{\text{du}} &= -C^{T} - \tilde{E}^{T} \hat{x}_{\text{du}}.
\end{aligned}
\label{residues}
\end{equation}
Using \cref{disc_primalfom,,disc_dualfom,,disc_primalrom} the error in the output can be shown to be bounded as
\begin{equation}
\begin{aligned}
| y^{k+1} - y_{r}^{k+1} | &\leq \phi^{k+1} \| r_{\text{pr}}^{k+1} \|.
\end{aligned}
\label{err_est}
\end{equation}
Here, $\phi^{k+1} \coloneqq \rho^{k+1} \, (\| \tilde{E}^{-1} \| \| r_{\text{du}} \| + \| \hat{x}_{\text{du}} \|)$.
The term 
\begin{equation}
\label{eq:rho}
\rho^{k+1} \coloneqq \dfrac{ \| \tilde{r}^{k+1}_{\text{pr}} \| }{\| r^{k+1}_{\text{pr}} \|},
\end{equation}
where
\begin{equation}
\begin{aligned}
\tilde{r}_{\text{pr}}^{k+1} &= \tilde{A} x^{k} + \Delta t f(x^{k}) + \Delta t B u^{k} - \tilde{E} \hat{x}^{k+1},\\
&=\tilde E (x^{k+1}-\hat{x}^{k+1}).
\end{aligned}
\label{auxresidual}
\end{equation}
is an auxiliary residual obtained by replacing $\hat{x}^{k+1}$ in the ``right-hand-side'' part ($\tilde{A} \hat{x}^{k} + \Delta t f(\hat{x}^k) + \Delta t B u^{k}$)  of $r^{k+1}_{\text{pr}}$ (see \cref{residues}) with the true solution $x^{k+1}$. It leads to a relation to the state error $x^{k+1}-\hat{x}^{k+1}$.

\begin{remark}
	\label{rem:rho}
	Note that the error bound only depends on the residuals at the current time step $t_{k+1}$, and avoids the error accumulation over time. The detailed explanation and proof of the error estimator can be found in \cite{morZhaFLetal15}. It is proved in~\cite{morZhaFLetal15} that, under mild assumptions, $\rho^{k+1}$ is lower and upper bounded. Upon convergence of the POD-Greedy algorithm,  $\rho^{k+1}$ should tend to be $1$. This conclusion will be demonstrated numerically in the section on simulation results.
	
	From \cref{err_est}, \cref{auxresidual}, we note that, at each time instance, the true solution $x^{k+1}$ is required for computing $ \tilde{r}^{k+1}_{\text{pr}}$ in the expression of  $\rho^{k+1}$. This can be avoided by approximating $\rho^{k+1}$ with a time-averaged value $\bar{\rho}$ obtained as,
	\begin{equation}
	\bar{\rho} = \frac{1}{K}\sum_{i=1}^{K} \rho^{i},
	\label{rho_est}
	\end{equation}
	where $\rho^{i}$ corresponds to $\tilde{r}^{i}_{\text{pr}}$, which requires the true solution $x^{i}$ at time instance $t_i$. If we take $t_i$ as the time instances at which the snapshots are computed for bases enrichment, then  $x^{i}$ are exactly the snapshots, which are available for free.  
	
	Furthermore, $\bar{\rho}$ is not available for all $\mu$. Therefore, when used inside a greedy algorithm, at each iteration, we simply approximate the value of $\bar{\rho}$ with, \\
	\[ \bar{\rho} \approx \bar{\rho}(\mu^{*}), \]
	where $\mu^{*}$ is the parameter chosen at the current iteration of the POD-Greedy algorithm, so that the snapshots at $\mu^{*}$ are available to compute $\bar{\rho}(\mu^{*})$.
\end{remark}

With the approximations in \cref{rem:rho}, the error bound becomes an error indicator, i.e.
\begin{equation}
\label{eq:error_est}
| y^{k+1} - y_{r}^{k+1} | \lesssim \bar{\Phi} \| r_{\text{pr}}^{k+1} \| =:\bar \Delta,
\end{equation}
where $\bar{\Phi} \coloneqq \bar{\rho}\,(\| \tilde{E}^{-1} \| \| r_{\text{du}} \| + \| \hat{x}_{\text{du}} \|)$.

It can be seen that the error indicator $\bar \Delta$ in \cref{eq:error_est} consists of two parts: 
$$\bar \Delta_1 \coloneqq \bar{\rho} (\| \tilde{E}^{-1} \| \| r_{\text{du}} \| \| r_{\text{pr}}^{k+1} \|)$$  
and $$\bar \Delta_2 \coloneqq \bar{\rho}\| \hat{x}_{\text{du}} \| \|r_{\text{pr}}^{k+1}\|.$$
The decay rate of $\bar \Delta_1$ is determined by the two residuals and the decay speed of $\bar \Delta_2$ depends on the product of the primal residual norm and the norm
of the approximate dual state  $\|\hat{x}_{\text{du}} \|$.  
We aim to improve the efficiency of the error indicator by considering each of them. On the one hand, we seek a corrected output so that 
the second part $\bar \Delta_2$ is modified to a form with faster decay rate. On the other hand, we try to use more suitable methods to obtain smaller $\| r_{\text{du}}\|$ than that in~\cite{morZhaFLetal15}. In this way, both $\bar \Delta_1$ and $\bar \Delta_2$ could decay faster, which results in sharper error estimation.
\subsection{Modified $\bar \Delta_2$ with corrected output}
We consider a correction term to the estimated output quantity given as,
\begin{equation}
\bar{y}_{r}^{k+1} = y_{r}^{k+1} - \big(\hat{x}_{\text{du}} \big)^{T}r_{\text{pr}}^{k+1}.
\label{mod_out}
\end{equation}
Similar techniques of output correction can be found in \cite{morGreP05,PieG00}, which are inherited from error analysis for finite element discretization \cite{morPerP98,morBecR98}. The resulting error estimation based on the correction term in our work differs from those in \cite{morGreP05,PieG00}. As far as we can see, the error bound derived in \cite{PieG00} estimates the numerical discretization error for steady systems, and the error bound in \cite{morGreP05} is only applicable to RBM for linear PDEs.
\begin{theorem}
	Given the discrete FOM in \cref{disc_primalfom} and the discrete ROM  in \cref{disc_primalrom}, assuming $\tilde{E}$ is non-singular at all values of $\mu$, we have the following error bound for the modified output term in \cref{mod_out},
	\begin{equation}
	| y^{k+1} - \bar{y}_{r}^{k+1} | \leq \| \tilde{E}^{-1} \| \| r_{\text{du}} \| \| \tilde{r}_{\text{pr}}^{k+1} \| + \| \hat{x}_{\text{du}} \| \| r_{\text{pr}}^{k+1} - \tilde{r}_{\text{pr}}^{k+1} \|,
	\end{equation}
\end{theorem}
\begin{proof}
	The error in the modified output can be represented as,
	\begin{align}
	y^{k+1} - \bar{y}_{r}^{k+1} &= C (x^{k+1} - \hat{x}^{k+1}) + (\hat{x}_{\text{du}})^{T} r_{\text{pr}}^{k+1}.
	\label{proof: eqn 3}
	\end{align}		
	Multiplying $(x^{k+1} - \hat{x}^{k+1})^{T}$  on both sides of \cref{disc_dualfom} we get,
	\begin{align}
	(x^{k+1} - \hat{x}^{k+1})^{T} \tilde{E}^{T} x_{\text{du}} &= -(x^{k+1} - \hat{x}^{k+1})^{T} C^{T}.
	\label{proof:eqn 1}
	\end{align}
	Transposing the above equation we obtain,
	\begin{align}
	(x_{\text{du}})^{T}\tilde{r}_{\text{pr}}^{k+1} &= -C\big(x^{k+1} - \hat{x}^{k+1}\big),
	\label{proof: eqn 2}
	\end{align}
	where we have made use of \cref{auxresidual}. 
	Next, we simply substitute \cref{proof: eqn 2} into \cref{proof: eqn 3}, followed by addition and subtraction of the term $(\hat{x}_{\text{du}})^{T}\tilde{r}_{\text{pr}}^{k+1}$ to obtain,
	\begin{equation}
	\begin{aligned}
	y^{k+1} - \bar{y}_{r}^{k+1} &= - (x_{\text{du}})^{T}\tilde{r}_{\text{pr}}^{k+1} + (\hat{x}_{\text{du}})^{T}r_{\text{pr}}^{k+1} \\
	&= - (x_{\text{du}})^{T}\tilde{r}_{\text{pr}}^{k+1} + (\hat{x}_{\text{du}})^{T}r_{\text{pr}}^{k+1} + (\hat{x}_{\text{du}})^{T}\tilde{r}_{\text{pr}}^{k+1} - (\hat{x}_{\text{du}})^{T}\tilde{r}_{\text{pr}}^{k+1}\\
	&= -
	(x_{\text{du}} - 	\hat{x}_{\text{du}} )^{T} \tilde{r}_{\text{pr}}^{k+1} + (\hat{x}_{\text{du}})^{T} \big( r_{\text{pr}}^{k+1} -  \tilde{r}_{\text{pr}}^{k+1} \big).	
	\end{aligned}
	\label{proof: eqn 4}
	\end{equation}
	Consider now the dual system residual as given in \cref{residues}. It can be shown that
	\begin{equation}
	\begin{aligned}
	r_{\text{du}} &= \tilde{E}^{T} \big( x_{\text{du}} - \hat{x}_{\text{du}}  \big),\\
	\big( x_{\text{du}} - \hat{x}_{\text{du}}  \big) &= \tilde{E}^{-T} r_{\text{du}}.
	\label{proof: eqn 6}
	\end{aligned}
	\end{equation}
	Inserting \cref{proof: eqn 6} into \cref{proof: eqn 4} yields
	\begin{align}
	y^{k+1} - \bar{y}_{r}^{k+1} &= - (r_{\text{du}})^{T} (\tilde{E})^{-1} \tilde{r}_{\text{pr}}^{k+1} + (\hat{x}_{\text{du}})^{T} \big( r_{\text{pr}}^{k+1} -  \tilde{r}_{\text{pr}}^{k+1} \big).
	\label{proof: eqn 7}
	\end{align}
	From the triangle and Cauchy-Schwarz inequalities we obtain,
	
	\begin{equation}
	\begin{array}{rl}
	| y^{k+1} - \bar{y}_{r}^{k+1} | &=  \| -(r_{\text{du}})^{T} \tilde{E}^{-1} 	 \tilde{r}_{\text{pr}}^{k+1} + (\hat{x}_{\text{du}})^{T} (r_{\text{pr}}^{k+1} - \tilde{r}_{\text{pr}}^{k+1})\|, \\
	&\leq \| \tilde{E}^{-1} \| \| r_{\text{du}} \| \| \tilde{r}_{\text{pr}}^{k+1} \| + \| \hat{x}_{\text{du}} \| \| r_{\text{pr}}^{k+1} - \tilde{r}_{\text{pr}}^{k+1} \|.
	\end{array}
	\label{proof: eqn 9}
	\end{equation}
\end{proof}
In order to remove the quantity $\| \tilde{r}_{\text{pr}}^{k+1}\|$ from the error bound to avoid computing the true solution, we propose the following error indicator. Instead of applying the upper bound $rb_2 \coloneqq \| r_{\text{pr}}^{k+1}\| + \|\tilde{r}_{\text{pr}}^{k+1} \|$ to $nr \coloneqq \|r_{\text{pr}}^{k+1} - \tilde{r}_{\text{pr}}^{k+1}\|$ in \cref{proof: eqn 9},
we propose to use $rb_1 \coloneqq \left | \| r_{\text{pr}}^{k+1}\| - \|\tilde{r}_{\text{pr}}^{k+1} \| \right |$ to approximate $nr$, 
i.e.
$$\|r_{\text{pr}}^{k+1} - \tilde{r}_{\text{pr}}^{k+1}\|\approx \left | \| r_{\text{pr}}^{k+1}\|- \|\tilde{r}_{\text{pr}}^{k+1} \| \right |.$$
This is motivated by the fact that, firstly, applying the upper bound $rb_2$ will result in an upper bound  $(1+\bar \rho) \|\hat x_{\textrm{du}}\|\| r_{\text{pr}}^{k+1}\|$ for $\| \hat{x}_{\text{du}} \| \| r_{\text{pr}}^{k+1} - \tilde{r}_{\text{pr}}^{k+1} \|$ in \cref{proof: eqn 9}, which is even larger than $\bar \Delta_2$ in \cref{eq:error_est}. Here we have used the relation in \cref{eq:rho}, \cref{rho_est}. Secondly, we find that
\begin{equation}
|rb_1-nr| \leq  \left \{
\begin{array}{ll}
2\rho^{k+1} \| r_{\text{pr}}^{k+1}\|, &  \rho^{k+1}>1\\
2 \| r_{\text{pr}}^{k+1}\|,  &    \rho^{k+1} \leq 1,
\end{array}\right.
\end{equation}
whereas,
$$|rb_2-nr| \leq [2+2\rho^{k+1}]\| r_{\text{pr}}^{k+1}\|.$$
This shows that $|rb_1-nr|$ poses a smaller upper bound than $|rb_2-nr|$, implicating that $nr$ could be better approximated by $rb_1$ than by $rb_2$. Therefore, we have the following error estimation
\begin{equation}
\begin{aligned}
| y^{k+1} - \bar{y}_{r}^{k+1} | &\lesssim  \| \tilde{E}^{-1} \| \| r_{\text{du}} \| \| \tilde{r}_{\text{pr}}^{k+1} \| + \| \hat{x}_{\text{du}} \|\left |\| r_{\text{pr}}^{k+1}\| - \|\tilde{r}_{\text{pr}}^{k+1} \| \right|  \\
&=  \rho^{k+1}  \| \tilde{E}^{-1} \| \| r_{\text{du}} \| \| r_{\text{pr}}^{k+1} \| + | 1 - \rho^{k+1} | \| \hat{x}_{\text{du}} \| \| r_{\text{pr}}^{k+1} \|  \\
&= \Big( \rho^{k+1}\| \tilde{E}^{-1} \| \| r_{\text{du}} \| + | 1 - \rho^{k+1} | \| \hat{x}_{\text{du}} \| \Big) \| r_{\text{pr}}^{k+1} \|\\
&\approx \Big( \bar \rho\| \tilde{E}^{-1} \| \| r_{\text{du}} \| + | 1 - \bar \rho | \| \hat{x}_{\text{du}} \| \Big) \| r_{\text{pr}}^{k+1} \|,
\end{aligned}
\end{equation}
where the relation in \cref{eq:rho} is used to remove $\|\tilde r_{\text{pr}}^{k+1} \|$ from the error estimation. Similarly as for the original error indicator, $\rho^{k+1}$ is approximated with the mean value $\bar{\rho}$ in \cref{rho_est}. We then define $\bar{\Psi} \coloneqq \Big( \bar \rho\| \tilde{E}^{-1} \| \| r_{\text{du}} \| + | 1 - \bar \rho | \| \hat{x}_{\text{du}} \| \Big)$ to get the following error indicator:
\begin{equation}
\label{mod_err_est}
| y^{k+1} - \bar{y}_{r}^{k+1} | \lesssim \bar{\Psi} \| r_{\text{pr}}^{k+1} \|.	
\end{equation}
Note that $\bar \Delta_2$ for the error indicator \cref{eq:error_est} now becomes $| 1 - \bar \rho | \| \hat{x}_{\text{du}} \| \| r_{\text{pr}}^{k+1} \|$ in \cref{mod_err_est}. Although $\| \hat{x}_{\text{du}} \| \| r_{\text{pr}}^{k+1} \|$ remains unchanged, the coefficient changes from $\bar\rho $
to $| 1 - \bar \rho |$. As has been analyzed in \cref{rem:rho}, when the POD-Greedy algorithm converges, $\bar \rho $ tends to be 1, so that $| 1 - \bar \rho |$ goes to 0, leading to a bound with faster decay rate. 
\begin{remark}
	For systems with multiple outputs, we define the dual system as
	\begin{equation*}
	\tilde{E}^{T} x_{\text{du}} = -C_{i}^{T},
	\end{equation*}
	where $C_{i}$ represents the $i^{\text{th}}$ row of the output matrix $C \in \mathbb{R}^{N_{O} \times N}$. For error estimation, we consider $| y_{i} - (\bar{y}_{r})_{i} |$ for each element $i$ of $y$. Then we have,
	\[ \|y^{k+1} - \bar{y}_{r}^{k+1}\|_{\infty} = \max \limits_{i \in \{ 1, 2, \ldots, N_{O} \}} |y_{i}^{k+1} - (\bar{y}_{r}^{k+1})_{i}|. \]
	\label{rem:multop}
\end{remark}
\subsection{Improving the decay rate of $\bar \Delta_1$}

Recall that the error indicator consists of two parts $\bar \Delta_1$ and $\bar \Delta_2$, where $\bar \Delta_1=\rho^{k+1}\|\tilde{E}^{-1}\| \| r_{\text{du}}\| \|r_{\text{pr}}^{k+1}\|$. In~\cite{morZhaFLetal15}, $\| r_{\text{du}} \|$ in $\bar \Delta_1$ is obtained by reducing the dual system using the projection matrix $V$ for MOR of the primal system, leading to a slowly decaying $\| r_{\text{du}} \|$. To achieve faster decay, we propose to use more suitable methods to compute the approximate solution
$\hat x_{\text{du}}$ to the dual system, according to different problems. The motivation is based on the observation that $\hat x_{\text{du}}$ is {\it not} necessarily computed by reducing the dual system using MOR.  Any appropriate method which can give a good $\hat x_{\text{du}}$ should be applicable. 

\paragraph*{Parametric dual system.}
When the dual system has parameter dependence, instead of using the primal system projection matrix $V$ to reduce the dual system, we construct the dual reduced basis separately as considered in \cite{morGreP05,morRozHP08}. With the dual reduced basis, the dual system solution can be much better 
approximated than using the primal reduced basis. Consequently, the computed 
$\hat x_{\text{du}}$ will give a residual with smaller norm. Certainly, this will lead to higher computational cost.

\paragraph*{Non-parametric dual system.}
In case the dual system is non-parametric (this can happen when the matrix $\tilde{E}$ is constant), instead of reducing the dual system, we can use Krylov-space methods, e.g. \texttt{GMRES, MINRES} \cite{PaiS75,SaaS86}, to iteratively solve the linear dual system and $\hat x_{\text{du}}$ is just the approximate solution computed by those methods. While this is done only once, the computational cost is similar to computing one snapshot and we do not expect much extra computations. This approach leads to much smaller $\| r_{\text{du}} \|$ than using the primal reduced basis to reduce the dual system.

\subsection{Efficiently computing the inf-sup constant}
In \cref{eq:error_est}, \cref{mod_err_est}, the term $\|  \tilde{E} ^{-1} \|$ needs to be calculated.
In case of matrix spectral norm $\| \cdot \|_{2}$,
\[ \| \tilde{E}^{-1} \|_{2} = \sigma_{\max}(\tilde{E}^{-1}) = \dfrac{1}{\sigma_{\min}(\tilde{E})}. \]
$\sigma_{\min}(\tilde{E})$ is actually the so-called inf-sup constant commonly used in the RBM. 
Since the matrix $\tilde{E}$ is parameter-dependent, it can become expensive to evaluate the smallest singular value for each parameter in the training set $\Xi$, as a large-scale eigenvalue problem needs to be solved for every parameter. Some methods have been proposed to make this computationally efficient. A first attempt was through the Successive Constraints Method (SCM) and its improvement, proposed in \cite{morHuyetal10,morHuyetal07}. However, SCM often suffers from a very slow convergence \cite{morSirK16}. In this work, we make use of the radial basis interpolation method proposed in \cite{morManN15}. This approach avoids the slow convergence rates as seen from SCM, while reducing the computational costs drastically.
\paragraph*{Radial basis interpolation for the inf-sup constant.}
When considering parameter spaces of high dimensions, the radial basis functions are a good candidate for interpolation basis. To achieve the goal of interpolating the smallest singular values of the matrix $\tilde{E}, \, \forall \mu \in \Xi$, we start with an initial coarse training set of points $\varUpsilon \subset \Xi$. The large-scale eigenvalue problem is solved for the parameters in this coarse training set and a coarse radial basis interpolant is formed, following the procedure outlined in \cite{morManN15}. The coarse parameter set $\varUpsilon$ is then enriched in a greedy manner. At each step, a new parameter from $\Xi$ is chosen and added to $\varUpsilon$ in order to improve the radial basis interpolant. The new parameter is chosen as the one that maximizes a pre-defined \emph{criterion function}, $\mathfrak{C}$\ defined over $\Xi$. The criterion function is such that it promotes adding new points in locations with highly varying response and ensures positivity of the interpolant. It was originally proposed in \cite{MacA10} and further adapted in \cite{morManN15} for the purpose of interpolating the smallest singular values. At the end of each iteration, the relative error (in the $\mathcal{L}_{\infty}$ norm) between the current and the previous interpolations for all parameters in $\Xi$ is computed. This defines the termination condition. Such an adaptive procedure offsets the need for performing SVD computations on large-scale matrices at all the parameters in $\Xi$.

\section{Adaptivity}
\label{sec:adaptivity}

In this section, we propose an adaptive scheme for automatically generating the reduced basis and the interpolation basis. The scheme is based on another separation of the error indicator as detailed below. 

\subsection{Error estimation by considering interpolation}

In practice, while calculating the residual of the primal system, we approximate the nonlinear term using (D)EIM to avoid the full dimensional computation. Considering this aspect, the residual of the primal system from \cref{residues} is now expressed as
\begin{equation}
\begin{array}{rl}
r_{\text{pr}}^{k+1} &= \tilde{A} \hat{x}^{k} + \Delta t f(\hat{x}^{k}) + \Delta t \mathcal{I}[f(\hat{x}^{k})] - \Delta t \mathcal{I}[f(\hat{x}^{k})] + \Delta t B u^{k} - \tilde{E} \hat{x}^{k+1}\nonumber\\
&= \underbrace{\Big( \tilde{A} \hat{x}^{k} + \Delta t \mathcal{I}[f(\hat{x}^{k})] + \Delta t B u^{k} - \tilde{E} \hat{x}^{k+1} \Big)}_{\coloneqq r_{\text{pr},\mathcal{I}}^{k+1}} + \underbrace{\Big( (\Delta t f(\hat{x}^{k}) - \Delta t \mathcal{I}[f(\hat{x}^{k})]) \Big)}_{\coloneqq e_{\mathcal{I}}^{k}}. \nonumber
\label{pod_deim_split}
\end{array}
\end{equation}
Here, $\mathcal{I}[f(\cdot)]$, is the interpolation of the function $f(\cdot)$. By considering the separation of $r_{\text{pr}}^{k+1}$ in the above equation, the error indicator in \cref{eq:error_est} or \cref{mod_err_est} can be split into two contributions - one from approximating the state by reduced basis vectors (RB error) and the other from approximating the nonlinear term by interpolation ((D)EIM error), i.e.
\begin{equation}
\begin{aligned}
| y^{k+1} - y_{r}^{k+1} | &\lesssim \underbrace{\bar{\Phi} \| r_{\text{pr},\mathcal{I}}^{k+1} \| }_{\coloneqq \bar{\Delta}_{\text{RB}}^{k+1}}+\underbrace{\bar{\Phi} \| \Delta_{\mathcal{I}}^{k} \|}_{\coloneqq \bar{\Delta}_{\mathcal{I}}^{k}}, 
\end{aligned}
\label{error_split_phi}
\end{equation}
or
\begin{equation}
\begin{aligned}
| y^{k+1} - \bar {y}_{r}^{k+1} | &\lesssim \underbrace{\bar{\Psi} \| r_{\text{pr},\mathcal{I}}^{k+1} \|}_{\coloneqq \bar{\Delta}_{\text{RB}}^{k+1}} + \underbrace{\bar{\Psi} \| \Delta_{\mathcal{I}}^{k}\|}_{\coloneqq \bar{\Delta}_{\mathcal{I}} ^{k}}, \end{aligned}
\label{error_split_psi}
\end{equation}
where $\Delta_{\mathcal{I}}^{k}$ is an error indicator for the interpolation error $e_{\mathcal{I}}^{k}$, obtained based on a higher order (D)EIM approximation proposed in \cite{morBarMNetal04,morGreMetal07,morDroHO12,morWirSH14}. In particular, we have,
$\Delta_{\mathcal{I}}^{k} := \Pi^{\ell^{'}} \big( I - \Pi^{\ell} \big) f(\hat{x}^{k}).$ Here, $\Pi^{\ell^{'}} := (I - \Pi^{\ell}) U^{\ell^{'}} \bigg( (P^{\ell^{'}})^{T} \big( I - \Pi^{\ell} \big) U^{\ell^{'}} \bigg)^{-1} (P^{\ell^{'}})^{T} $ and $  \Pi^{\ell} := U^{\ell} \bigg( (P^{\ell})^{T} U^{\ell} \bigg)^{-1} (P^{\ell})^{T} $. The pair $(U^{\ell}, P^{\ell})$ corresponds to the (D)EIM basis and interpolation indices for an order-$\ell$ approximation of $f(\hat{x}^{k})$, whereas the pair $(U^{\ell^{'}}, P^{\ell^{'}})$ corresponds to those for an order-$\ell^{'}$ approximation, where $\ell^{'} > \ell$. For a detailed treatment, we refer to \cite{morWirSH14}.
\newline

For convenience of explanation, we use the same names $\bar{\Delta}_{\text{RB}}^{k+1}$, $\bar{\Delta}_{\mathcal{I}} ^{k}$ for the two parts of either the original error indicator or the modified one, though we will mainly focus on the modified error indicator in the later analysis.

\paragraph*{Mean error estimate.}

Note that the above error indicator measures the output error at each time step, which corresponds to many values. For the proposed adaptive scheme, we need a single value to measure the actual error.
Therefore, we define the mean error as below,
\begin{equation}
\begin{aligned}
\dfrac{1}{K} \sum_{i=1}^{K} \| y^{i} - \bar{y}_{r}^{i} \| &\leq \dfrac{1}{K} \sum_{i=1}^{K} \bigg(\,\bar{\Delta}_{\text{RB}}^{i} + \bar{\Delta}_{\mathcal{I}}^{i} \, \bigg) =  \bar{\Delta}_{\text{RB}} + \bar{\Delta}_{\mathcal{I}} =: \bar{\Delta}.
\end{aligned}
\label{error_mean}
\end{equation}
Here, $ i = 1, 2, \ldots, K$ indicates the time instances $t_{i}$ where the snapshots are taken, $\bar{\Delta}_{\text{RB}} =  \dfrac{1}{K} \sum_{i=1}^{K} \bar{\Delta}_{\text{RB}}^{i}$ and $\bar{\Delta}_{\mathcal{I}} =  \dfrac{1}{K} \sum_{i=1}^{K} \bar{\Delta}_{\mathcal{I}}^{i}$. With the mean output error indicator, we are ready to propose the adaptive scheme as below.

\subsection{Adaptively increasing/reducing the number of RB and EI basis vectors}

Consider a user defined tolerance $\texttt{tol}$ for the ROM. At every iteration, we check the relative changes of error indicators $\bar{\Delta}_{\text{RB}}$ and $\bar{\Delta}_{\mathcal{I}}$ in \cref{error_mean} w.r.t the error tolerance $\texttt{tol}$, respectively, i.e. 
\begin{equation}
\label{eq:ratios}
\frac{\bar{\Delta}_{\text{RB}}}{\texttt{tol}} \quad \text{and} 
\quad  \frac{\bar{\Delta}_{\mathcal{I}}}{\texttt{tol}}.
\end{equation}
It is easy to see that a value in \cref{eq:ratios} being bigger than 1 implicates the current basis is not accurate enough, and needs to be extended; otherwise, a value being smaller than 1 means that the current bases are accurate enough, and no new basis vectors need to be computed. If possible, some old basis vectors could be removed to make the reduced basis/interpolation basis space as compact as possible. Note that the actual values in \cref{eq:ratios} could vary from, e.g., $10^{-10}$ to $10^{10}$. In order to decide how many basis vectors need to be added/removed, we use the function $log$, which maps the actual values in \cref{eq:ratios} to some moderate values, typically falling into a subinterval of $[-10,10]$.
We then use the rounded $log$-mapped values as the indicators for basis enriching or shrinking. Let $\ell_{\text{RB}}, \ell_{\text{EI}}$ be the number of basis vectors in the RB, (D)EIM projection matrices at the current iteration, respectively.

For the ROM considered in \cref{rom} we update the RB, (D)EIM projection matrices at each iteration of the greedy algorithm based on the following rule,
\begin{equation}
\begin{aligned}
\Delta \ell_{\text{RB}} &\coloneqq \pm 1 + \Bigl \lfloor \log \big ( \frac{\bar{\Delta}_{\text{RB}}}{\textnormal{\texttt{tol}}} \big) \Bigr \rfloor,\\
\Delta \ell_{\text{EI}} &\coloneqq \pm 1 + \Bigl \lfloor \log \big ( \frac{\bar{\Delta}_{\mathcal{I}}}{\textnormal{\texttt{tol}}} \big) \Bigr \rfloor.\\
\end{aligned}
\label{adapt_eqn_first}
\end{equation}
Here, $\bar{\Delta}_{\text{RB}}$, $\bar{\Delta}_{\mathcal{I}}$ are defined as in \cref{error_mean}.
Based on the above update, the number of basis vectors for RB, (D)EIM basis enriching/shrinking is given as
\begin{equation}
\begin{aligned}
\ell_{\text{RB}} &= \ell_{\text{RB}} + \Delta \ell_{\text{RB}} ,\\ \ell_{\text{EI}} &= \ell_{\text{EI}} + \Delta \ell_{\text{EI}}.
\end{aligned}
\label{adapt_eqn_second}
\end{equation}
%
\begin{remark}
	Recall that a one-way definition of $\Delta \ell_{\text{RB}}, \Delta \ell_{\text{EI}}$ was proposed in~\cite{morFenMB17}, where basis shrinking was not considered. Finally, the adaptive algorithm must start from a small number of basis vectors to be able to continue. Whereas, both basis extension and shrinking are covered by \cref{adapt_eqn_first}. The value $\pm 1$ tries to ensure at least one basis vector is added (+1)/removed (-1), in case the rounded value  
	$p = \Bigl \lfloor \log \big ( \frac{\bar{\Delta}_{\text{RB}}}{\textnormal{\texttt{tol}}} \big) \Bigr \rfloor$ or  $d = \Bigl \lfloor \log \big ( \frac{\bar{\Delta}_{\mathcal{I}}}{\textnormal{\texttt{tol}}} \big) \Bigr \rfloor$ is zero, but the values $ \log \big ( \frac{\bar{\Delta}_{\text{RB}}}{\textnormal{\texttt{tol}}} \big)$ or  $\log \big ( \frac{\bar{\Delta}_{\mathcal{I}}}{\textnormal{\texttt{tol}}} \big)$ are nonzero (either slightly smaller than zero or slightly larger than zero).
	\label{rem:nonadapt_trivial_update}
\end{remark}
\paragraph{Dealing with stagnation in adaptivity.}
The goal for adaptivity is to add/delete basis vectors to/from the current RB, EI bases so that the ROM meets the given tolerance while being kept as compact as possible. It is observed that in some cases, the convergence of the adaptive algorithm becomes slow when the estimated error is close to the tolerance. In such situations, the number of additional basis vectors to be added/deleted is usually one, resulting in a slow convergence. A second issue is that the error indicator could keep oscillating (below and above the error tolerance) upon basis enriching/shrinking. To avoid the above phenomena, 
we propose to define a \emph{zone-of-acceptance} (\texttt{zoa}) for the output error.
In particular, we set a new value $\epsilon^* < \texttt{tol}$. $\epsilon^*$ and $\texttt{tol}$ then define a \emph{zone-of-acceptance}: $[\epsilon^*, \texttt{tol}]$. 
Whenever the estimated error falls into $[\epsilon^*, \texttt{tol}]$, the algorithm will terminate. We typically take $\epsilon^*=0.1\,\texttt{tol}$. 
\paragraph*{(D)EIM plateaus.}
Previous works dealing with the simultaneous enrichment of RB and (D)EIM bases have noted the issue of (D)EIM plateaus. Actually, two different notions of plateauing have been observed and presented in \cite{morGreMetal07,morDroHO12, morCanTU09,morIapQR16,morSmeO17,morTon11}. In \cite{morGreMetal07,morSmeO17,morTon11}, the authors note that when the number of basis elements of the EIM approximation is fixed at some small value, an increase in the number of RB vectors does not result in an improvement in the overall error. This is a plateauing due to \emph{large errors in the EIM approximation}.
However, in \cite{morDroHO12}, it is observed that EIM plateaus occur when the error contribution is dominated by the RB approximation error and a further enrichment of the EIM is useless. This is a plateauing due to \emph{large errors in the RB approximation}.
These observations suggest that simultaneous enrichment of the RB and the (D)EIM basis is critical to avoid plateaus in general. A solution proposed in \cite{morDroHO12} is to monitor an error estimator over the training set, i.e., $\max \limits_{\mu \in \Xi} \bar{\Delta}(\mu)$, between two successive iterations. If a newly added RB basis vector leads to an increase in the error, then it is dropped and only the EIM basis is updated. In \cite{morTon11}, the author considers different tolerances for the RB, EIM approximations. However, the proposed algorithm involves some user-defined constants which makes it less straightforward to implement. From our experience in the numerical tests, setting different tolerances for the RB and EIM approximation, with $\texttt{tol}_{\text{EI}} < \texttt{tol}_{\text{RB}}$ proves to give the best approximation. A similar observation is also noted in \cite{morSmeO17}, however no simultaneous enrichment is considered there. Still, it is not entirely clear how small the (D)EIM approximation tolerance has to be when compared to the RB tolerance. In our numerical experiments, we use $\texttt{tol}_{\text{EI}} = 0.01 \times \texttt{tol}_{\text{RB}}$.
\newline

In the following, we first extend the one-way adaptive algorithm in~\cite{morFenMB17} for non-parametric systems to a two-way algorithm, then we propose the adaptive POD-Greedy-(D)EIM algorithm for parametric systems.
\subsubsection{Adaptive algorithm for non-parametric systems}
In this section, we extend the one-way adaptive scheme in~\cite{morFenMB17} to a two-way process. This means we have the flexibility of starting our algorithm given any possible number of initial basis vectors. The method is able to suggest the proper number of basis vectors to be added to or removed from the current basis, and  yields a compact and stable ROM, for the given tolerance. We detail this in \cref{alg_twoway_adap}. The constants $p_{0}, d_{0}$ in \cref{alg_twoway_adap} act as $\pm 1$ in Equation \cref{adapt_eqn_second}. See \cref{rem:nonadapt_trivial_update}.
\renewcommand{\algorithmicrequire}{\textbf{Input:}}
\renewcommand{\algorithmicensure}{\textbf{Output:}}	
\begin{algorithm}[t!]
	\begin{algorithmic}[1]
		\caption{Adaptive POD-(D)EIM (two-way)}
		\label{alg_twoway_adap}
		\Require POD, (D)EIM projection basis and interpolation points from \cref{alg_pod,,alg_deim}, respectively:
		
		\noindent $V \in \mathbb{R}^{N \times \ell_{\text{RB}}}$, $U_{f} \in \mathbb{R}^{N \times \ell_{\text{EI}}}$, Index matrix $P = [e_{\wp_{1}}, e_{\wp_{2}},  \ldots, e_{\wp_{\ell}}]$, initial choice of basis $\ell_{\text{RB}}^{0}, \ell_{\text{EI}}^{0}$, Tolerance \texttt{tol}, \texttt{zoa}.
		
		\Ensure ROM with updated size ($\ell_{\text{RB}}^{*},\ell_{\text{EI}}^{*}$) .
		
		\State Form $V^{*} \coloneqq V(: \, ,\, 1:\ell_{\text{RB}}^{0}), \, U_{f}^{*} \coloneqq U_{f}(: \, ,\, 1:\ell_{\text{EI}}^{0}), \, P^{*} \coloneqq P(: \, ,\,1:\ell_{\text{EI}}^{0})$.
		
		\State Determine the dual system solution $\hat{x}_{\text{du}}$, using \texttt{GMRES}.
		
		\State Simulate the ROM constructed using $V^{*}, U_{f}^{*}, P^{*}$ and compute the error $\bar{\Delta} \coloneqq \bar{\Delta}_{\text{RB}} + \bar{\Delta}_{\mathcal{I}}$.
		
		\While{$\bar{\Delta} \notin$ \texttt{zoa}}
		\State Determine p = $\Bigl \lfloor \log \big ( \frac{\bar{\Delta}_{\text{RB}}}{\texttt{tol}} \big) \Bigr \rfloor$,
		d = $\Bigl \lfloor \log \big ( \frac{\bar{\Delta}_{\mathcal{I}}}{\texttt{tol}} \big) \Bigr \rfloor$.
		\State Trivial updates $p_{0}, d_{0} = \pm1$, in case $p$ or $d$ = 0, to ensure at least one basis vector is added/removed.
		\State $\text{PODinc} = p_{0} + p$	 ,  $\text{(D)EIMinc} = d_{0} + d$.
		\State $\ell_{\text{RB}}^{*} = \ell_{\text{RB}}^{*} + \text{PODinc}$.
		\State $\ell_{\text{EI}}^{*} = \ell_{\text{EI}}^{*} + \text{(D)EIMinc}$.
		\State Ensure $\ell_{\text{EI}}^{*} > \ell_{\text{RB}}^{*}$, for stability reasons.
		\State Update projection matrices,  $V^{*} \coloneqq V(: \, ,\, 1 : \ell_{\text{RB}}^{*}), \, U_{f}^{*} \coloneqq U_{f}(: \, ,\, 1 : \ell_{\text{EI}}^{*}), \, P^{*} \coloneqq P(: \, ,\,1 : \ell_{\text{EI}}^{*})$.	 
		\State Simulate the updated ROM constructed using $V^{*}, U_{f}^{*}, P^{*}$ and compute the updated error $\bar{\Delta} = \bar{\Delta}_{\text{RB}} + \bar{\Delta}_{\mathcal{I}}$.
		\EndWhile
	\end{algorithmic}
\end{algorithm}

\begin{remark}
	In our numerical experiments, we noticed that in some cases, when $\ell_{\text{EI}}^{*} < \ell_{\text{RB}}^{*}$, the ROM is no longer stable.  Therefore, in Step 10 of \cref{alg_twoway_adap} we set the number of EI basis vectors ($\ell_{\text{EI}}^{*}$) to be larger than that of the RB basis vectors ($\ell_{\text{RB}}^{*}$).
\end{remark}
\begin{remark} 
	Note that the inputs of \cref{alg_twoway_adap} are the POD and DEIM basis computed by standard algorithms, where both bases are conservatively computed to guarantee accuracy of the ROM. 
	The adaptive algorithm chooses proper basis vectors from the POD and DEIM basis, respectively, in order to construct a more compact and stable ROM. One starts from an initial choice for the POD and DEIM basis combination. Both basis vectors are then iteratively updated, based on the output error indicator.
\end{remark}
\subsubsection{Adaptive algorithm for parametric systems}
\label{sec:adappodgreedy}
For standard implementation of the POD-Greedy algorithm combined with interpolation of the nonlinear part, the (D)EIM basis and interpolation indices are usually pre-computed outside of the greedy loop. This separate basis generation often leads to a less compact ROM. In addition, (D)EIM needs FOM simulations at all samples in a training set, which is time consuming especially for problems needing many time steps for one simulation, e.g., the batch chromatographic model presented in~\cref{sec:results}. In~\cite{morDroHO12}, POD-Greedy and EIM are implemented such that at each iteration of the greedy algorithm, 
a single basis vector is added to the current reduced bases and the interpolation bases, respectively. The most recent work that addresses this issue can be found in~\cite{morBenEEetal18, morDavP15}. As has been discussed in the introduction, there are major differences between those existing algorithms and our proposed algorithm.\newline
\indent
\cref{alg_adap_podgreedy} is the proposed adaptive POD-Greedy-(D)EIM algorithm. \cref{alg_update_ei,,alg_update_vdu,,alg_adap_basis_upd} are the supporting functions needed. The basic idea of \cref{alg_adap_podgreedy} is that starting from the first selected parameter $\mu^*$, the RB basis and the (D)EIM basis are updated simultaneously but not trivially. The number of the RB and (D)EIM basis vectors $\ell_{\text{RB}}$ and $\ell_{\text{EI}}$ are determined by \cref{adapt_eqn_second}, and updated in the subroutine \texttt{Adapt\_basis\_update}. The reduced basis matrix $V$ is updated in Step 5 (or Step 9) of \cref{alg_adap_podgreedy}, where instead of using $\ell_{\text{RB}}$ = 1 as in the standard case, $\ell_{\text{RB}}$ is adaptively computed by \texttt{Adapt\_basis\_update}. The interpolation basis matrix $U_{\text{EI}}$ is updated by the subroutine \texttt{Update\_EI}, which essentially implements the interpolation algorithm DEIM or EIM. The dual system is also reduced by the reduced basis method implemented in the subroutine~\texttt{Update\_$V_{\text{du}}$}. It is clear that FOM simulations are only needed for those selected parameters $\mu^*$, not for all parameters in the training set $\Xi$. \cref{alg_adap_basis_upd} determines the number of RB and (D)EIM basis vectors $(\ell_{\text{RB}}, \ell_{\text{EI}})$, that will be added/removed at each iteration. 
\renewcommand{\algorithmicrequire}{\textbf{Input:}}
\renewcommand{\algorithmicensure}{\textbf{Output:}}	
\begin{algorithm}[t!]
	\begin{algorithmic}[1]
		\caption{Adaptive POD-Greedy-(D)EIM}
		\label{alg_adap_podgreedy}		
		\Require Parameter training set $\Xi \subset \mathscr{P}$, Tolerance $\texttt{tol}$, $\epsilon_{\text{POD}}$ (or $\epsilon_{\text{EI}}$).
		\Ensure RB Basis $V \in \mathbb{R}^{N \times \ell_{\text{RB}}}$, (D)EIM basis and interpolation points $U_{\text{EI}}, P_{\text{EI}} \in \mathbb{R}^{N \times \ell_{\text{EI}}}$.		
		\State In case of non-parametric dual system, precompute the approximate solution to the dual system ($\hat{x}_{\text{du}}$), using \texttt{GMRES}.
		\State Initialize, $V = [\,]$, $V_{\text{du}} = [\,]$, $U_{\text{EI}} = [\,]$, $P_{\text{EI}} = [\,]$, $\ell_{\text{RB}}$ = 1, $\ell_{\text{EI}}$ = 1. Initial $\mu^{*}$: a random parameter from $\Xi$.
		\While {$\bar{\Delta}(\mu^{*})$ $\notin$ \texttt{zoa}}
		\If{$\ell_{\text{RB}} < 0$}
			\State Remove the last $\ell_{\text{RB}}$ columns from $V$.
		\Else
			\State Compute FOM at $\mu^{*}$ and obtain snapshots, X = $[x(t_{1}, \mu^{*}), x(t_{2}, \mu^{*}), \ldots, x(t_{K}, \mu^{*})]$.
			\State Update the projection matrix. $\bar{\text{X}} \coloneqq$ X - $\text{Proj}_{\mathcal{V}}$(X) $\xrightarrow{\text{SVD}}$ $U \Sigma W^{T}$, $\mathcal{V}$ is the subspace spanned by the columns of $V$. 
			\State $V \leftarrow{} \texttt{orth}\{ V, \, U(: \, , 1:\ell_{\text{RB}} )\}$.
		\EndIf
		\State \texttt{Update\_EI}.	
		\State \texttt{Update\_$V_{\text{du}}$}, in case of parametric dual system.
		\State $\mu^{*} \coloneqq$ arg $\max\limits_{\mu \in \Xi} \bar{\Delta}(\mu)$.
		\State \texttt{Adapt\_basis\_update}.
		\EndWhile
	\end{algorithmic}
\end{algorithm}
%
\begin{algorithm}[t!]
	\begin{algorithmic}[1]
		\caption{\texttt{Update\_EI}}
		\label{alg_update_ei}
		\Require Snapshots of the nonlinear vector at all parameters so far selected by the greedy algorithm, 
		
		\noindent $F = \big[ f(x(t_{1},\mu), \, \mu), \, f(x(t_{2},\mu), \, \mu), \, \ldots \,,
		f(x(t_{K},\mu), \, \mu) \big]$ for all selected $\mu$, $\ell_{\text{EI}}$, $\epsilon_{\text{POD}}$ (or $\epsilon_{\text{EI}}$).
		\Ensure (D)EIM basis and interpolation points $U_{\text{EI}}$, $P_{\text{EI}}$.		
		\State \texttt{max\_iter} = $\ell_{\text{EI}}$ for \cref{alg_eim} or set $\ell = \ell_{\text{EI}}$ in Step 6 of \cref{alg_deim} and ignore Step 3.	
		\State Call [$U_{\text{EI}}$, $P_{\text{EI}}$] = \texttt{EIM}($F$,  \texttt{max\_iter}, $\epsilon_{\text{EI}}$) or 
		[$U_{\text{EI}}$, $P_{\text{EI}}$] = \texttt{DEIM}($F$, $\epsilon_{\text{POD}}$).
	\end{algorithmic}    
\end{algorithm}
\begin{algorithm}[t!]
	\begin{algorithmic}[1]
		\caption{\texttt{Update\_$V_{\text{du}}$}}
		\label{alg_update_vdu} 
		\Require $V_{\text{du}}$, $\mu^{*}_{\text{du}}$, \texttt{tol}.
		\Ensure Updated dual projection matrix $V_{\text{du}}$.
		\If{$\bar{\Delta}_{\text{du}}(\mu_{\text{du}}^{*}) >$  \texttt{tol}}
		\State Solve full order dual system for chosen parameter $\mu^{*}_{\text{du}}$ and obtain $x_{\text{du}}(\mu^{*}_\text{du})$.
		\State Update $V_{\text{du}}$, $V_{\text{du}} \coloneqq \texttt{orth}\{ V_{\text{du}} \,, x_{\text{du}}(\mu^{*}_\text{du})\}$.
		\State
		$\mu^{*}_{\text{du}} \coloneqq$ $\arg \max\limits_{\mu \in \Xi} \bar{\Delta}_{\text{du}}(\mu)$.
		\EndIf{}       
	\end{algorithmic}   
\end{algorithm}
\begin{algorithm}[t!]
	\begin{algorithmic}[1]
		\caption{\texttt{Adapt\_basis\_update}}
		\label{alg_adap_basis_upd}
		\Require $l_{\text{RB}}$, $l_{\text{EI}}$, $\bar{\Delta}_{\text{RB}}(\mu^{*})$, $\bar{\Delta}_{\mathcal{I}}(\mu^{*})$, $\texttt{tol}$.
		\Ensure Updated $\ell_{\text{RB}}$, $\ell_{\text{EI}}$.
		\State $p = \Bigl \lfloor \log \big ( \frac{\bar{\Delta}_{\text{RB}}(\mu^{*})}{\texttt{tol}} \big) \Bigr \rfloor $, $d = \Bigl \lfloor \log \big ( \frac{\bar{\Delta}_{\mathcal{I}}(\mu^{*})}{\texttt{tol}} \big) \Bigr \rfloor$.
		\State In case $p=0$ or $d=0$, enforce trivial update, $p_{0} = \pm1$ or $d_{0} = \pm1$.
		\State $\ell_{\text{RB}} = p_{0} + p$.
		\State $\ell_{\text{EI}} = \ell_{\text{EI}} + d_{0} + d$.
		\State Ensure $\ell_{\text{EI}} > (\text{rank}(V) + \ell_{\text{RB}})$, for stability reasons.
	\end{algorithmic}  
\end{algorithm}
\begin{remark}
	$\texttt{orth}\{ V, \, U(: \, , 1:\ell_{\text{RB}} )\}$ in Step 9 of \cref{alg_adap_podgreedy} is an orthogonalization step for the updated RB matrix $V$. Using a modified Gram-Schmidt (MGS) procedure is recommended, where the last $p$ columns in $V$ represent the most recently added basis vectors in the previous iteration. In case of $l_{\text{RB}}<0$, it allows direct removal of $p$ columns from $V$, without 
	influencing the previous basis vectors. The most recent computed basis vectors should be removed, if the error indicator indicates basis shrinking. This is easy to understand, since the error indicator monitors the error of the current ROM when new basis vectors are added to the basis space. If it is already smaller than the tolerance, it means those newly added basis vectors are not necessary and can be removed.
\end{remark}
\begin{remark}
	In practice, we can make use of separate tolerances for the RB, (D)EIM error in \cref{alg_adap_basis_upd}, i.e. Step 1 can be modified as, $p = \Bigl \lfloor \log \big ( \frac{\bar{\Delta}_{\text{RB}}(\mu^{*})}{\textnormal{\texttt{tol}}_{\texttt{RB}}} \big) \Bigr \rfloor $, $d = \Bigl \lfloor \log \big ( \frac{\bar{\Delta}_{\mathcal{I}}(\mu^{*})}{\textnormal{\texttt{tol}}_{\texttt{EI}}} \big) \Bigr \rfloor$. The tolerance for (D)EIM approximation is usually set a little lower than that for the RB approximation. This follows from the observation that the nonlinear vector needs to be sufficiently well-approximated to enable a good state approximation.
\end{remark}

\section{Numerical Results}
\label{sec:results}
In this section we test the proposed adaptive algorithm on three examples. All the examples we consider can be represented in the general form of \cref{fom}. We use the same non-parametric model as in~\cite{morFenMB17} to test the extended two-way adaptive \cref{alg_twoway_adap}. \cref{alg_adap_podgreedy} is tested 
for the other two parametric models. 
\subsection{A non-parametric example}
The model Fluidized Bed Crystallizer (FBC) is from the field of chemical engineering. Enantiomers are molecules that have the same physical, chemical properties but occur as mirror images of one another. Due to their similar properties, separation of the two components is not easily achieved using simple techniques, but requires sophisticated methods such as adsorption, crystallization, etc. For a more in-depth treatment, the reader is referred to \cite{morManFKetal15}. 
\begin{figure}[t]
	\centering
	\includegraphics[scale=0.3]{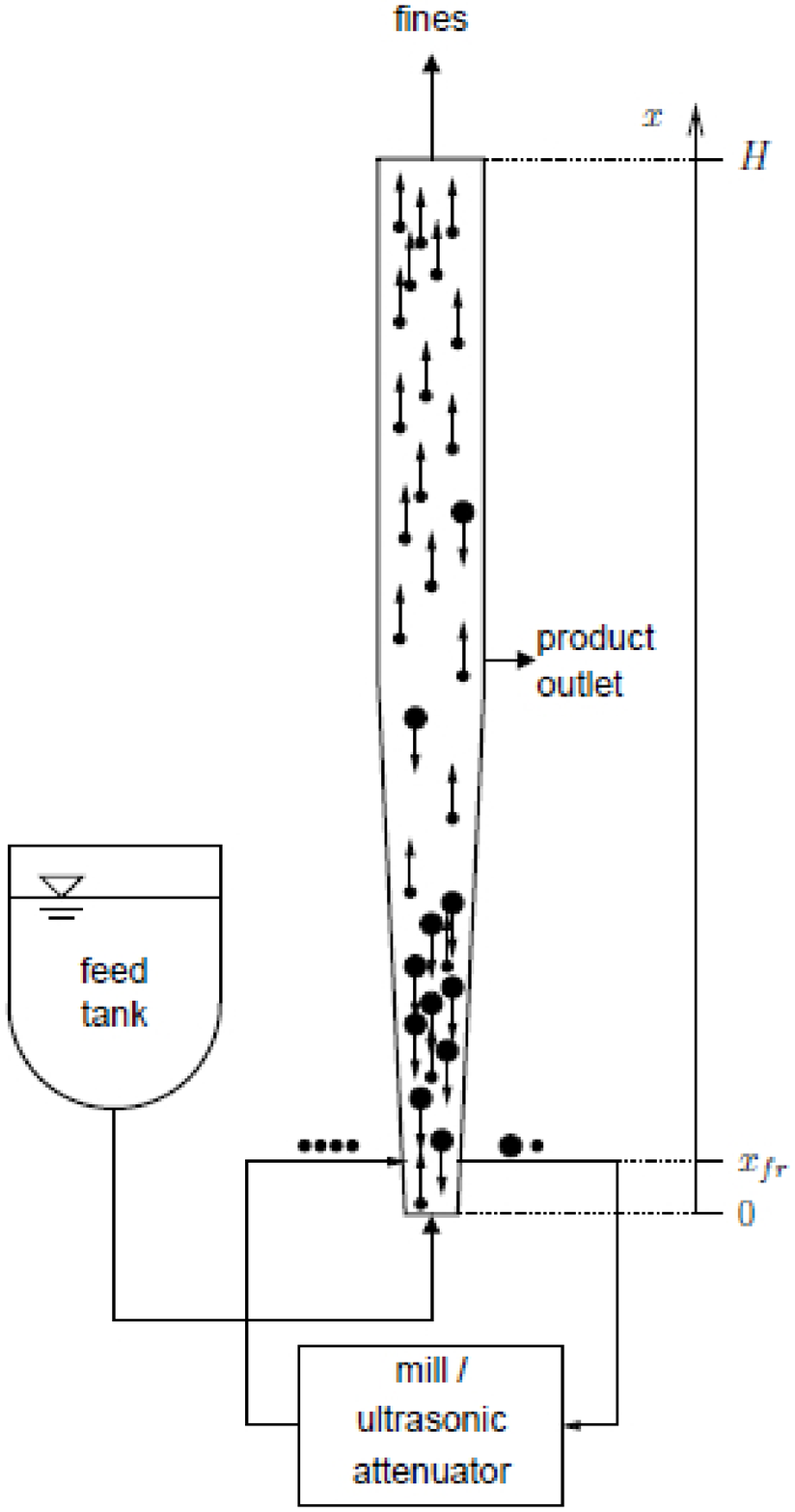}
	\caption{Model of the Fluidized bed crystallizer \cite{morFenMB17}.}
	\label{fig:crystallizer}
\end{figure}
\cref{fig:crystallizer} shows a FBC, which is a long cylindrical column, with the walls tapering inwards as one approaches the bottom. The chemical mixture that has the two enantiomers dissolved in it (called racemate) is injected from the bottom. Some seed crystals need to be introduced into the crystallizer before it begins operation. Seed crystals are essentially the pure crystals of the enantiomer we want to isolate. The seeds are necessary for triggering the precipitation of the crystals in the racemate. During its operation, the smaller crystals move to the top of the crystallizer along with the fluid flow. Bigger crystals sink to the bottom from where they are collected and sent to a crushing device (such as an ultrasonic attentuator) to be crushed to an appropriate size and reintroduced as seed crystals. The crystallization process is governed by a set of conservation formulas, called the population balance equations which are PDEs. The PDE governing the FBC is given as:
\begin{align}
\begin{autobreak}
A_{c}(x) \frac{\partial n}{\partial t} = 
-\frac{\partial}{\partial x} \big( A_{c}(x) v_{p} (x, L, t) n(x,L,t) \big) 
+ D \frac{\partial}{\partial x} \Big( A_{c}(x) \frac{\partial n}{\partial x} \Big) 					
-A_{c}(x) G \frac{\partial n}{\partial L} 												
+ \dot{V}_{\text{us}} \Bigg( n_{\text{us}}(L) \frac{\int_{0}^{\infty} nl^{3} dl}{\int_{0}^{\infty} n_{us}l^{3}dl} - n(x, L, t) \Bigg) \hat{\delta} (x - x_{\text{us}}),
\end{autobreak}
\label{crystallizer}
\end{align}
where
\begin{itemize}
	\item $x$ denotes the spatial coordinate, $L$ denotes the particle size coordinate,
	\item $A_{c}(\cdot)$ is the area of cross-section, $n$ is the number size density, i.e., the number of particles per volume of size $L$ at coordinate $x$ at time $t$,
	\item $v_{p}$ is the plug flow velocity,
	\item $D$ is the dispersion constant, $G$ is the crystal growth factor,
	\item $\dot{V}_{\text{us}}$ is the volume flow to/from the attenuator,
	\item $n_{\text{us}}(L)$ is a constant describing the distribution of the crystals coming from the attenuator.
\end{itemize}
\begin{table}[t!]
	\caption{Simulation data for the FBC.}
	\vspace{-5mm}	
	\begin{center}
		\begin{tabular}{|c|c|c|c|c|}
			\hline
			Interpolation & $N$ & $[0, T]$ (s)  & tolerance & $\epsilon_{\text{POD}}$\\
			\hline
			DEIM & 18400 & [0, 500] & $10^{-4}$ & $10^{-10}$\\
			\hline
		\end{tabular}				
	\end{center}
	\label{tab:cystallizer_param}
\end{table}
Equation \cref{crystallizer} can be discretized and rewritten into a form as in \cref{disc_primalfom}. \cref{tab:cystallizer_param} gives the model parameters that we consider for the full order simulation. After discretization in space using the finite volume method, the original system is of size $N=18400$. For the time variable, we consider a semi-implicit Euler discretization. We use \cref{alg_twoway_adap}, where DEIM is used to compute the interpolation basis. The tolerance for the ROM is set as $\texttt{tol}=10^{-4}$. The model of the crystallizer involves a very long time to reach a cyclic state, usually 5000 seconds. However, for snapshot generation, we need only the transient portion and the first cycle of the steady state since the latter cycles behave very similarly. As a result we only need to simulate the FOM till 500 seconds for snapshot generation. \cref{alg_twoway_adap} has been used to implement the two-way adaptivity scheme. We make use of \texttt{GMRES} to solve the associated dual system. It is implemented via the MATLAB\textsuperscript{$\circledR$} function \texttt{gmres}. Moreover, we use the Incomplete LU (ILU) factorization with a drop tolerance of $10^{-3}$ as a preconditioner. The \texttt{GMRES} tolerance is set to be $10^{-6}$.

To test \cref{alg_twoway_adap}, we consider two cases. We denote the first case as \textsf{\sc{Increase}}. It involves starting from a small initial choice of RB and DEIM basis dimension, and iteratively adding new basis vectors to both. In the second case, denoted as \textsf{\sc{Decrease}}, we initialize \cref{alg_twoway_adap} with a larger number of initial basis vectors and adaptively remove basis vectors, till the ROM reaches the defined tolerance band $\texttt{zoa}$. \cref{fig:fbc_convergence} shows the adaptive generation of POD and DEIM basis starting from small initial numbers of $(\ell_{\text{RB}}, \ell_{\text{EI}}) = (3,8)$. The error indicator is below the tolerance after 9 iterations, showing that \cref{alg_twoway_adap} terminates. 
The adaptive process results in a final ROM of $(\ell_{\text{RB}}, \ell_{\text{EI}}) = (16, 20)$ basis vectors. For the ROM without adaptivity, we set the SVD tolerance $\epsilon_{\text{POD}} = 10^{-10}$ in \cref{alg_pod,alg_deim} in order to obtain a ROM below the error tolerance $\texttt{tol} = 10^{-4}$. This results in a ROM with $(\ell_{\text{RB}}, \ell_{\text{EI}})=(60,61)$.
Moreover, simply using the Standard POD algorithm and setting $\epsilon_{\text{POD}}=10^{-4}$ in \cref{alg_pod} does not guarantee that the output error of the ROM is also below the tolerance $\texttt{tol} = 10^{-4}$.

In \cref{fig:inc-landscape}, we show the error landscape obtained by plotting the logarithm of the mean error ($\bar{\Delta}$) corresponding to different combinations of $(\ell_{\text{RB}}, \ell_{\text{EI}})$. On the landscape, we mark the trajectory of $\bar{\Delta}$ adaptively selected by \cref{alg_twoway_adap}. We can see that, for several combinations of $(\ell_{\text{RB}}, \ell_{\text{EI}})$ not present in the adaptive trajectory, the resulting ROMs were unstable. For ease of visualization, the mean error of those $(\ell_{\text{RB}}, \ell_{\text{EI}})$ combinations resulting in unstable ROMs were set to be $1$ in the log-scale. The figure clearly illustrates how the algorithm converges to the minimum in the landscape, while avoiding the combinations resulting in instabilities. Further, one can identify the plateaus in the error, whenever the RB approximation is too poor. An additional observation deserving attention is that the instabilities mainly occur at $(\ell_{\text{RB}}, \ell_{\text{EI}})$ combinations with $\ell_{\text{EI}} < \ell_{\text{RB}}$.

For a pair of big initial values: $(\ell_{\text{RB}}, \ell_{\text{EI}})=(31, 39)$, the iterations of \cref{alg_twoway_adap} are shown in~\cref{fig:fbc_convergence_decrease}. In the beginning, the error indicator is below $10^{-5}$, indicating that the ROM is very accurate and there is possibility to further reduce the size of the ROM. After 7 iterations, the reduced basis vectors from POD as well as the DEIM basis vectors are adaptively adjusted to
$(\ell_{\text{RB}}, \ell_{\text{EI}})=(17,28)$. These results have been summarised in \cref{tab:cystallizer_results}. In \cref{fig:fbc_convergence,,fig:fbc_convergence_decrease}, the true error is the mean error defined by the left hand side of \cref{error_mean} and the corresponding error indicator is defined by the right hand side of the same inequality.
%
\begin{table}[t!]
	\caption{Simulation results for the FBC example.}
	\vspace{-5mm}	
	\begin{center}
		\begin{tabular}{|l|cccc|c|}
			\hline
			\multirow{2}{*}{Process}             & \multicolumn{2}{c|}{Initial}                                    & \multicolumn{2}{c|}{Final}                                      & \multirow{2}{*}{Iterations}                \\ \cline{2-5}
			& \multicolumn{1}{l}{$\ell_{\text{RB}}$} & \multicolumn{1}{l|}{$\ell_{\text{EI}}$} & \multicolumn{1}{l}{$\ell_{\text{RB}}$} & \multicolumn{1}{l|}{$\ell_{\text{EI}}$} &                                                                        \\ \hline
			\multicolumn{1}{|c|}{\textsf{\sc{Increase}}} & 3                              & \multicolumn{1}{c|}{8}                              & 16                             & 20                             & 9                            
			\\ \cline{1-6}
			\multicolumn{1}{|c|}{\textsf{\sc{Decrease}}} & 31                             & \multicolumn{1}{c|}{39}                             & 17                             & 28                             & 7                                                                        \\ \hline
		\end{tabular}
	\end{center}
	\label{tab:cystallizer_results}
\end{table}	
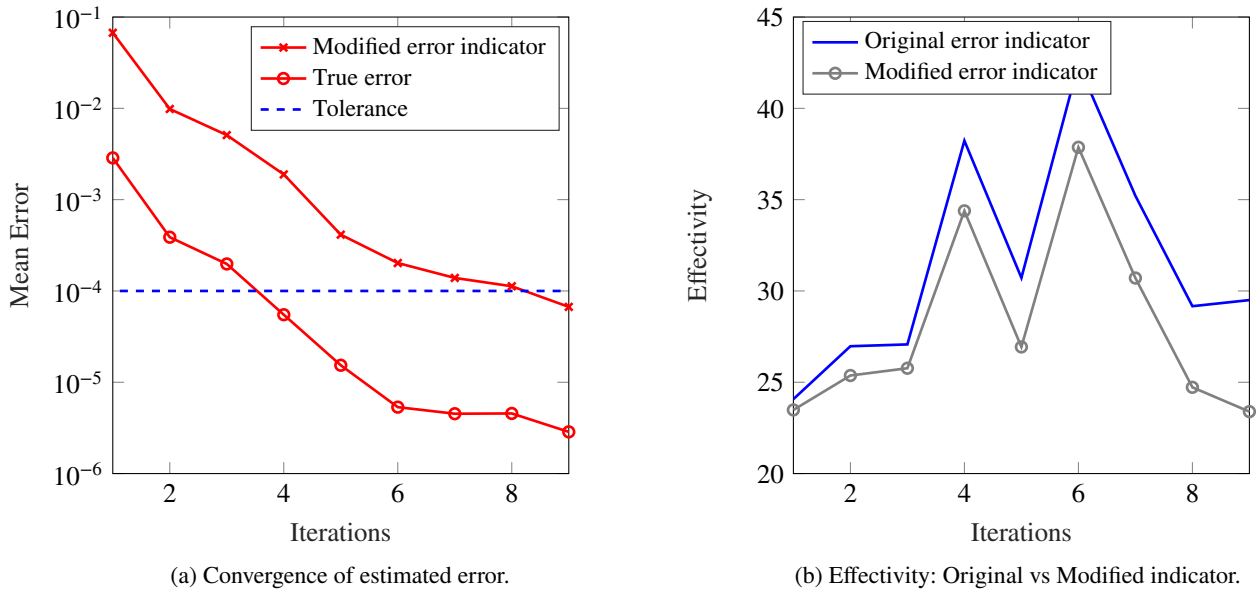
\begin{figure}[t!]
	\captionsetup[subfigure]{font=scriptsize,labelfont=scriptsize}
	\centering
	\newlength\fheight
	\newlength\fwidth
	\setlength\fheight{6cm}
	\setlength\fwidth{6cm}	
	\begin{subfigure}[b]{0.5\textwidth}
		\input{figures/fbc_convergence.tex}
		\caption{Convergence of estimated error.}
		\label{fig:fbc_convergence}
	\end{subfigure}%
	\begin{subfigure}[b]{0.5\textwidth}		
		\input{figures/fbc_effectivity.tex}
		\caption{Effectivity: Original vs Modified indicator.}
		\label{fig:fbc_effectivity}
	\end{subfigure}%
	\caption{FBC adaptive increment.  }
\end{figure}
%
\setlength\fheight{8cm}
\setlength\fwidth{8cm}	
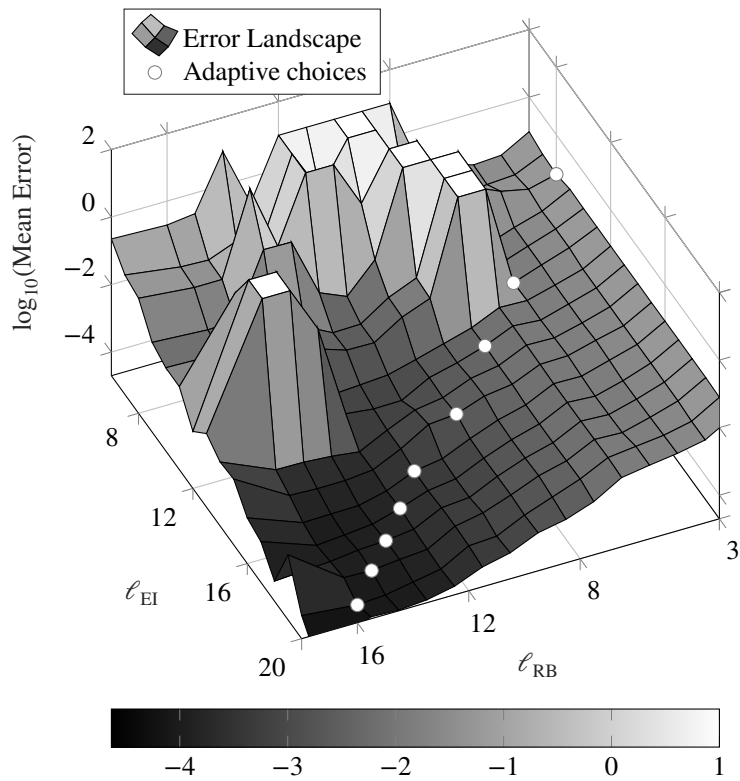
\begin{figure}[h!]
	\centering
	\input{figures/error_landscape_plot_3d_corrected_birdseye.tex}
	\caption{Error landscape for the \textsf{\sc{increase}} procedure for the FBC example.}
	\label{fig:inc-landscape}
\end{figure}
In \cref{fig:fbc_effectivity,,fig:fbc_effectivity_decrease}, we compare the effectivities of the original and the modified error indicators. On the one hand, both error indicators show good effectivities and are relatively sharp. On the other hand, the modified error indicator clearly outperforms the original indicator, especially in the final step of the algorithm. \cref{fig:dec-landscape} shows the error landscape for the \textsf{\sc{decrease}} case. Once again, one can see how \cref{alg_twoway_adap} avoids unstable RB, DEIM basis vector combinations and converges to a compact ROM. Finally, \cref{fig:fbc_error_compare_increase,,fig:fbc_error_compare_decrease} compare the modified error indicator for the final ROM over all time steps $t_{k}$, with the true error, in the increasing and decreasing cases, respectively. \cref{fig:fbc_error_compare} not only shows the sharpness of the modified error indicator, especially for the cyclic state in the time interval $[200,500]$s, but also verifies the reliability of the error indicator.
	\setlength\fheight{6cm}
	\setlength\fwidth{6cm}	
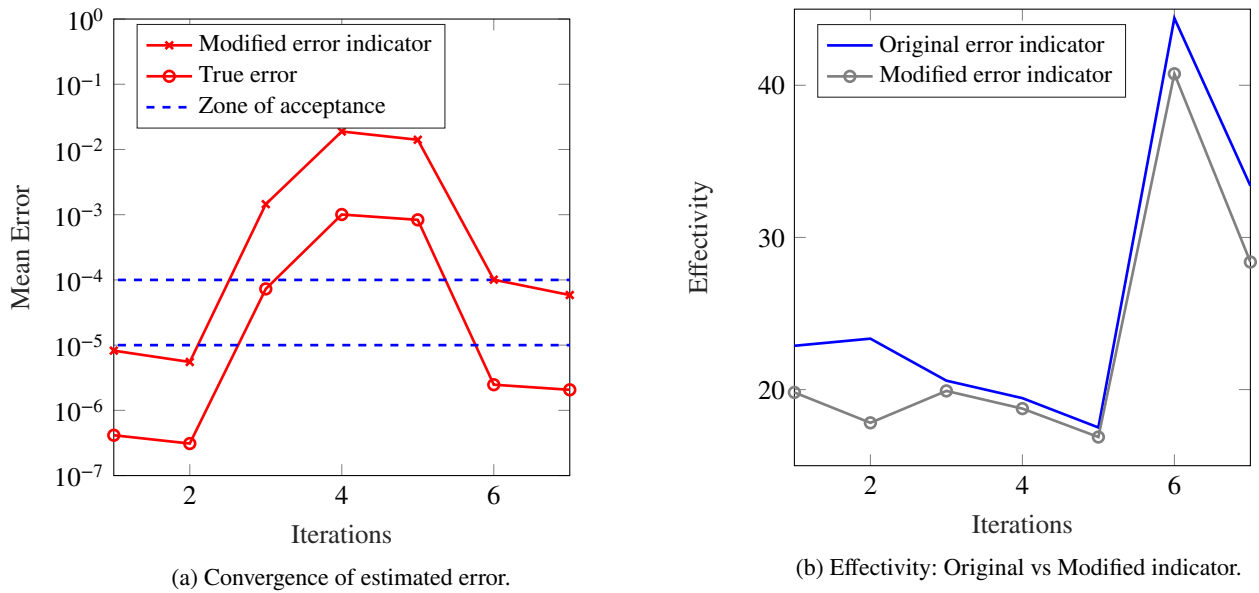
\begin{figure}[h!]
	\centering
	\begin{subfigure}[h]{0.5\textwidth}	
		\input{figures/fbc_convergence_decrease.tex}
		\caption{Convergence of estimated error.}
		\label{fig:fbc_convergence_decrease}
	\end{subfigure}%
	\begin{subfigure}[h]{0.5\textwidth}	
		\input{figures/fbc_effectivity_decrease.tex}
		\caption{Effectivity: Original vs Modified indicator.}
		\label{fig:fbc_effectivity_decrease}
	\end{subfigure}%
	\caption{FBC adaptive decrement.}
\end{figure}
\setlength\fheight{8cm}
\setlength\fwidth{8cm}
\begin{figure}[t!]
	\centering
	\input{figures/error_landscape_plot_3d_corrected_birdseye_decrease.tex}
	\caption{Error landscape for the \textsf{\sc{decrease}} procedure for the FBC example.}
	\label{fig:dec-landscape}
\end{figure}
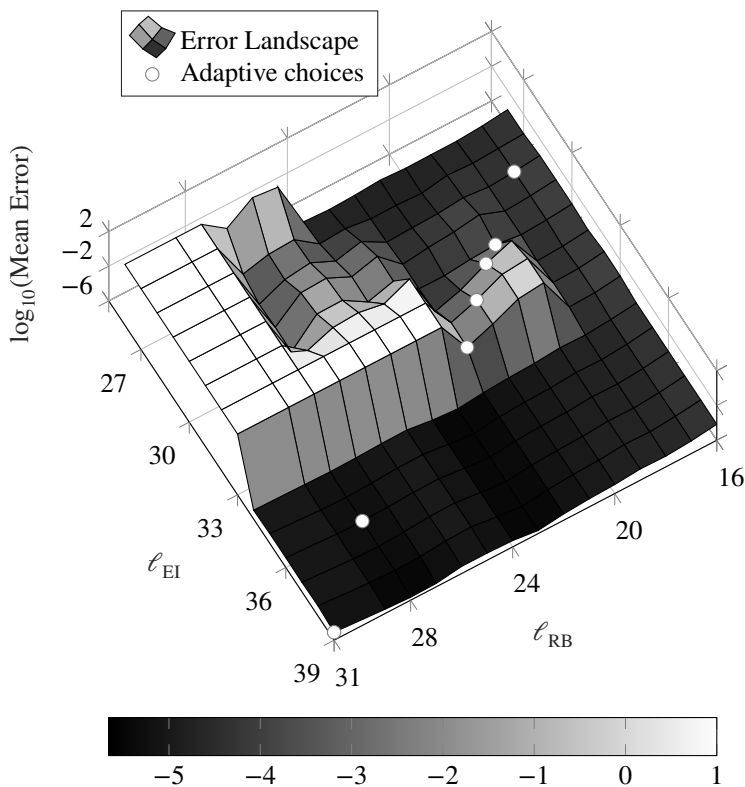
	\setlength\fheight{6cm}
	\setlength\fwidth{6cm}
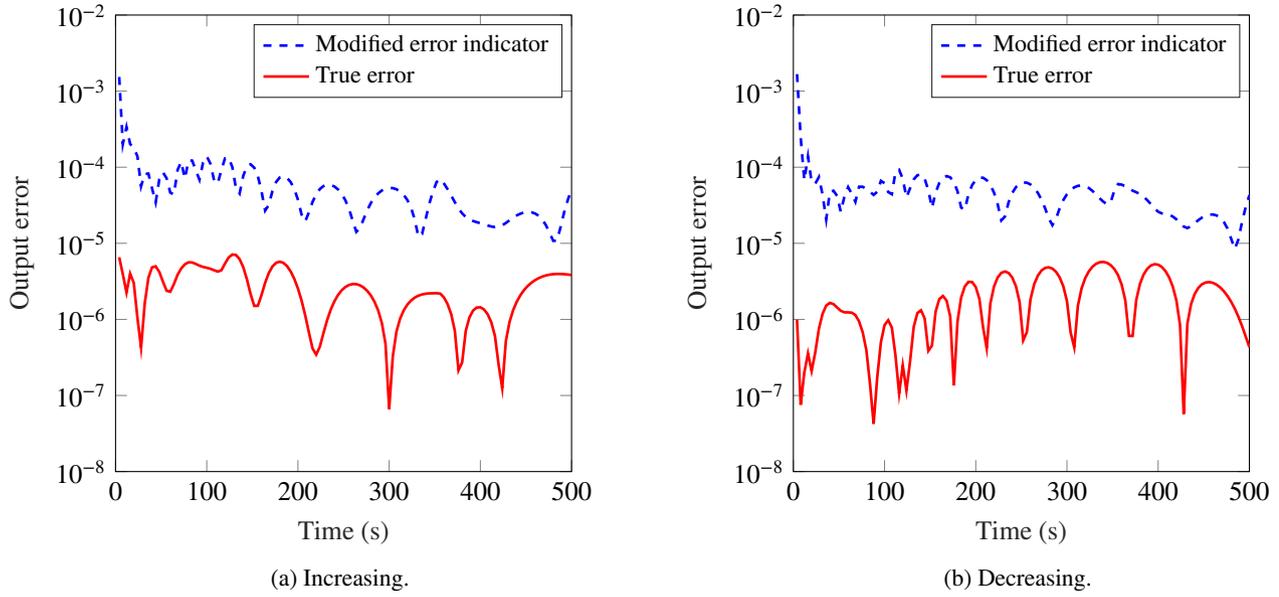
\begin{figure}[t!]
	\begin{subfigure}[b]{0.5\textwidth}
		\input{figures/fbc_final_ROM_error_compare.tex}
		\caption{Increasing.}
		\label{fig:fbc_error_compare_increase}
	\end{subfigure}%
	\begin{subfigure}[b]{0.5\textwidth}	
		\input{figures/fbc_final_ROM_error_compare_decrease.tex}
		\caption{Decreasing.}
		\label{fig:fbc_error_compare_decrease}
	\end{subfigure}%
	\caption{Final ROM error over all time for the FBC. }
	\label{fig:fbc_error_compare}
\end{figure}
\subsection{Parametric example}
We consider two examples of parametric systems. The first is the viscous Burgers' equation and the second is a chemical process called batch chromatography. The former is an example with a parametric dual system while the latter has a non-parametric dual system.
\subsubsection*{Burgers' equation}
We test the proposed Adaptive POD-(D)EIM-Greedy algorithm on the one-dimensional viscous Burgers' equation defined in the domain $\sigma \in \Omega \coloneqq [0 , 1]$. The parameter that varies is the viscosity. The equation and initial boundary conditions are given as
\begin{equation}
\begin{aligned}
\frac{\partial w}{\partial t} + w\frac{\partial w}{\partial \sigma}  &= \mu \frac{\partial ^2 w}{\partial \sigma^2} + s(\sigma,t),\\
w(\sigma , 0) &= 0, \\  
\frac{\partial w(1,t)}{\partial \sigma} &=  0,
\end{aligned}
\end{equation}
where $w \coloneqq w(\sigma,t) \in \mathbb{R}^{N} $ is the state variable. $s(\sigma,t)$ is the source/input term, $\mu$ is the viscosity. The output is taken at the last spatial point in the domain:  $y = w(1,t)$. We consider $s(\sigma,t) \equiv 1$. The initial condition is defined as $w(\sigma,0) := 0$.
\begin{table}[t!]
	\caption{Simulation parameters for the Burgers' equation.}
	\vspace{-5mm}	
	\begin{center}
		\begin{tabular}{|c|c|c|c|c|c|}
			\hline
			Interpolation & $N$ & $[0,T]$ (s) & Parameter training set ($\Xi$) & tolerance &  $\epsilon_{\text{EI}}$\\
			\hline
			\multirow{2}{*}{EIM}  & \multirow{2}{*}{500} & \multirow{2}{*}{[0,2]} & 100 $\log$-uniformly distributed samples& \multirow{2}{*}{$10^{-3}$} & \multirow{2}{*}{$10^{-10}$}\\
			&&& in [0.0005,1] && \\
			\hline
		\end{tabular}			
	\end{center}
	\label{tab:burgers_param}
\end{table}
\newline
	\setlength\fheight{6cm}
	\setlength\fwidth{6cm}
\begin{figure}[t!]			
	\centering
	\input{figures/burgers-low-viscosity-EIM-Adapt-greedy-convergence.tex}
	\caption{Burgers' equation: Convergence of the greedy algorithms: \cref{alg_podgreedy_ei} vs \cref{alg_adap_podgreedy}.}
	\label{fig:burgers_convergence}
\end{figure}
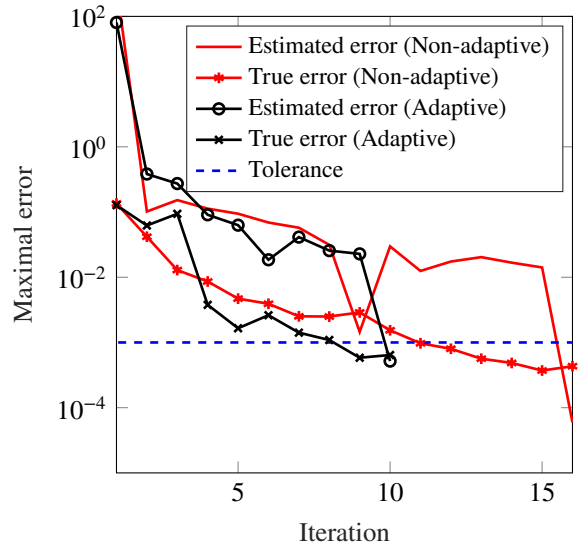
	\setlength\fheight{5.5cm}
	\setlength\fwidth{5.5cm}
\begin{figure}[h!]
	\begin{subfigure}[b]{0.5\textwidth}	
		\input{figures/burgers-low-viscosity-EIM-Adapt-basis.tex}
		\caption{}
		\label{fig:burgers_basis}
	\end{subfigure}%
	\begin{subfigure}[b]{0.5\textwidth}	
		\input{figures/burgers-low-viscosity-EIM-Adapt-effectivity.tex}
		\caption{}
		\label{fig:burgers_compare}
	\end{subfigure}%
	\caption{\cref{alg_adap_podgreedy} for the Burgers' equation. (a) Adaptive increment of RB vs EIM basis vectors. (b) Effectivity: Original vs modified indicator.}
\end{figure}
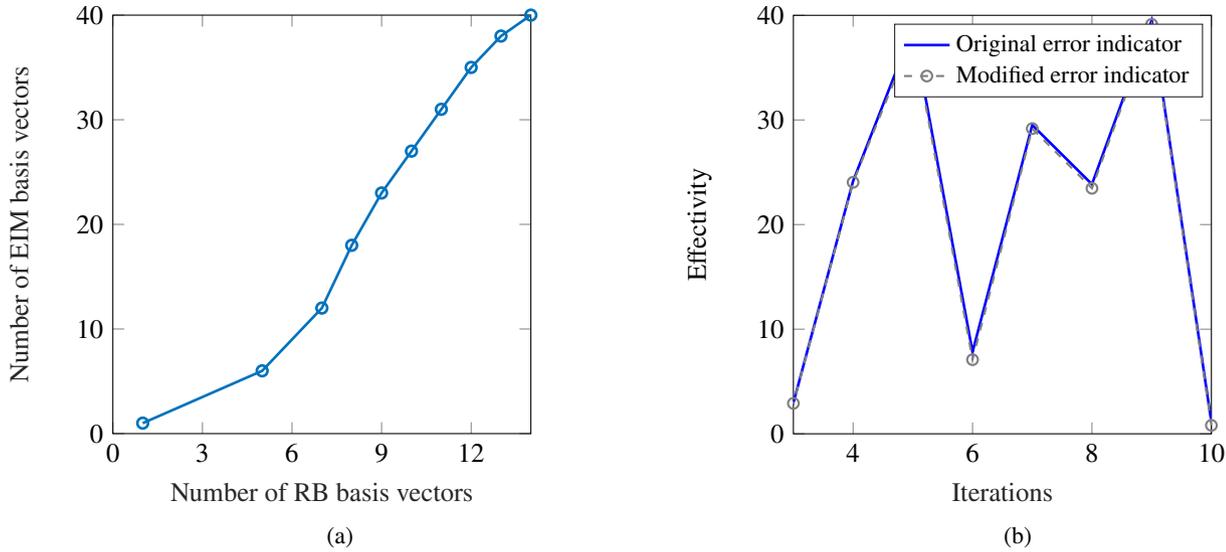
	\setlength\fheight{5.5cm}
	\setlength\fwidth{5.5cm}
\begin{figure}[t!]		
	\begin{subfigure}[b]{0.5\textwidth}
		\input{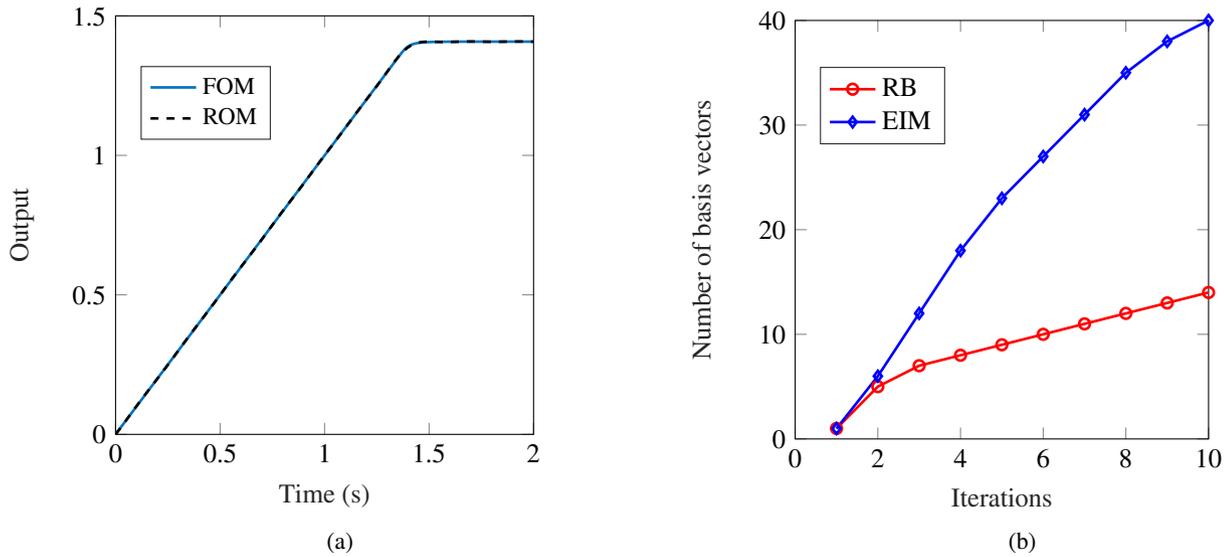}
		\caption{}
		\label{fig:burgers_compare_output}
	\end{subfigure}
	\begin{subfigure}[b]{0.5\textwidth}
		\input{figures/rb_burgers_basis_vs_iteration.tex}
		\caption{}
		\label{fig:burgers_basis_iteration}
	\end{subfigure}
	\caption{Burgers' Equation (a) Output at $\mu = 5.\,10^{-4}$. (b) RB, EIM basis vs iteration number.}
\end{figure}
\begin{table}[t]
	\caption{Runtime comparison for the Burgers' equation.}	
	\centering
	\begin{subtable}[t]{.4\linewidth}
		\centering
		\begin{tabular}{|c|c|}
			\hline
			Method    & runtime (s) \\ \hline
			Adaptive & 606         \\ \hline
			Non-adaptive       & 933        \\ \hline
		\end{tabular}
		\caption{Adaptive vs non-adaptive.}
		\label{tab:burgers_times}		
	\end{subtable}%
	\centering
	\begin{subtable}[t]{.4\linewidth}	
		\centering
		\begin{tabular}{|c|c|}
			\hline
			Method    & runtime (s) \\ \hline
			SVD & 3.7         \\ \hline
			RBF       & 0.4         \\ \hline
		\end{tabular}
		\caption{Inf-sup constant computed over training set.}
		\label{tab:burgers_infsup_RBF}		
	\end{subtable} 
\end{table}
The simulation parameters are listed in \cref{tab:burgers_param}. A training set, $\Xi$ is formed by 100 $\log$-uniformly distributed samples in the parameter domain $\mathscr{P} \coloneqq [0.0005\,,1]$. The model has $N = 500$ equations after discretization in space. We employ the central difference scheme for both the diffusion and convection terms. A semi-implicit Euler method is used to discretize the time variable. We make use of EIM to treat the nonlinear term. We set $\epsilon_{\text{EI}}$ to be $10^{-10}$ in \cref{alg_adap_podgreedy}. A time step of $\Delta t = 4\,\cdot 10^{-4}$ was used, with the snapshots collected every $10^{\text{th}}$ time step.
In \cref{fig:burgers_convergence}, we compare the Standard POD-Greedy-(D)EIM with the proposed Adaptive POD-Greedy-(D)EIM algorithm. It can be seen that using the latter leads to a much quicker convergence of the greedy loop: 10 iterations as compared to 16 iterations that the standard greedy algorithm needs. We show the convergence of the modified error indicator and the true error for both algorithms. In \cref{fig:burgers_convergence} and \cref{fig:batchchrom_convergence} the maximal errors are defined over all the parameters in the training set $\Xi$. In particular, $\bar{\Delta}_{\text{max}} \coloneqq \max\limits_{\mu \in \Xi} \bar{\Delta}(\mu)$, where $\bar{\Delta}(\mu)$ is defined in \cref{error_mean}. The maximal true error is defined as: $\max\limits_{\mu \in \Xi} 	\dfrac{1}{K} \sum_{i=1}^{K} \| y^{t_{i}} - \bar{y}^{t_{i}}_{r} \|$, where $t_{i}$, with $i = 1, 2, \ldots, K$, are the time instances where the snapshots are taken.  The improved convergence of \cref{alg_adap_podgreedy} is a direct consequence of enriching the basis in an adaptive manner.
\newline \indent
In \cref{fig:burgers_basis}, we plot the successive increments of the RB, EIM basis vectors. Starting from a value of $1$ for each, we can see that the biggest jumps are at the first few steps when the output error is estimated to be large. Subsequent steps moderate the number of basis vectors to be added, as the algorithm converges. We end up with a final value of $(\ell_{\text{RB}}, \ell_{\text{EI}}) =  (14, 40 )$ for the RB, EIM basis respectively. As for the standard implementation, where the EIM basis is precomputed outside the greedy loop, the resulting ROM has dimension $(\ell_{\text{RB}}, \ell_{\text{EI}}) = (16, 154)$. Thus, our proposed algorithm not only produces a ROM that meets a certain tolerance, but also leads to a more compact ROM. 
\newline
In \cref{tab:burgers_times}, we show the runtime taken for the Adaptive and Non-adaptive greedy algorithms till convergence. The adaptive algorithm needs much less time. The reduced runtime of the adaptive approach is mainly contributed by the reduced number of FOM simulations.
For the inf-sup constant (the smallest singular value of the system matrix) we apply radial basis function interpolation. From \cref{tab:burgers_infsup_RBF}, it is clear that the RBF approach is much faster as compared to using SVD to determine the smallest singular value of the system matrix. One can imagine, the savings in time would be much more significant for large-scale systems with $N \gg 500$.
\cref{fig:burgers_compare_output} shows the output $y(t, \mu)$ of the FOM and $y_{r}(t, \mu)$ of the ROM at $\mu = 5.\, 10^{-4}$. The ROM solution is nearly indistinguishable from the FOM solution. Finally, in \cref{fig:burgers_basis_iteration} we plot the number of RB, EIM basis vectors as a function of the iteration number. Note that, both RB, EIM basis start with just one basis vector at the first iteration.
\subsubsection*{Batch chromatography model}
The last example is from batch chromatography, a process used to separate components of a mixture. We give the schematic of the process in \cref{fig:chromatography}. The mixture containing two components that need to be separated is periodically injected at one end of the column. In the column, a static bed of a substance called the stationary phase is present. The injected mixture has to pass through this stationary phase. Batch chromatography relies on the phenomenon of adsorption. The components that need to be separated have different adsorption affinities towards the stationary phase and hence tend to move through the column with varying velocities. The separated components are then collected at the end of the column. The time of collecting the two components of the original mixture is determined based on the required purity specifications. The dynamic equations for the process can be given as,
\\
\begin{figure}[t]
	\centering
	\includegraphics[scale=0.6]{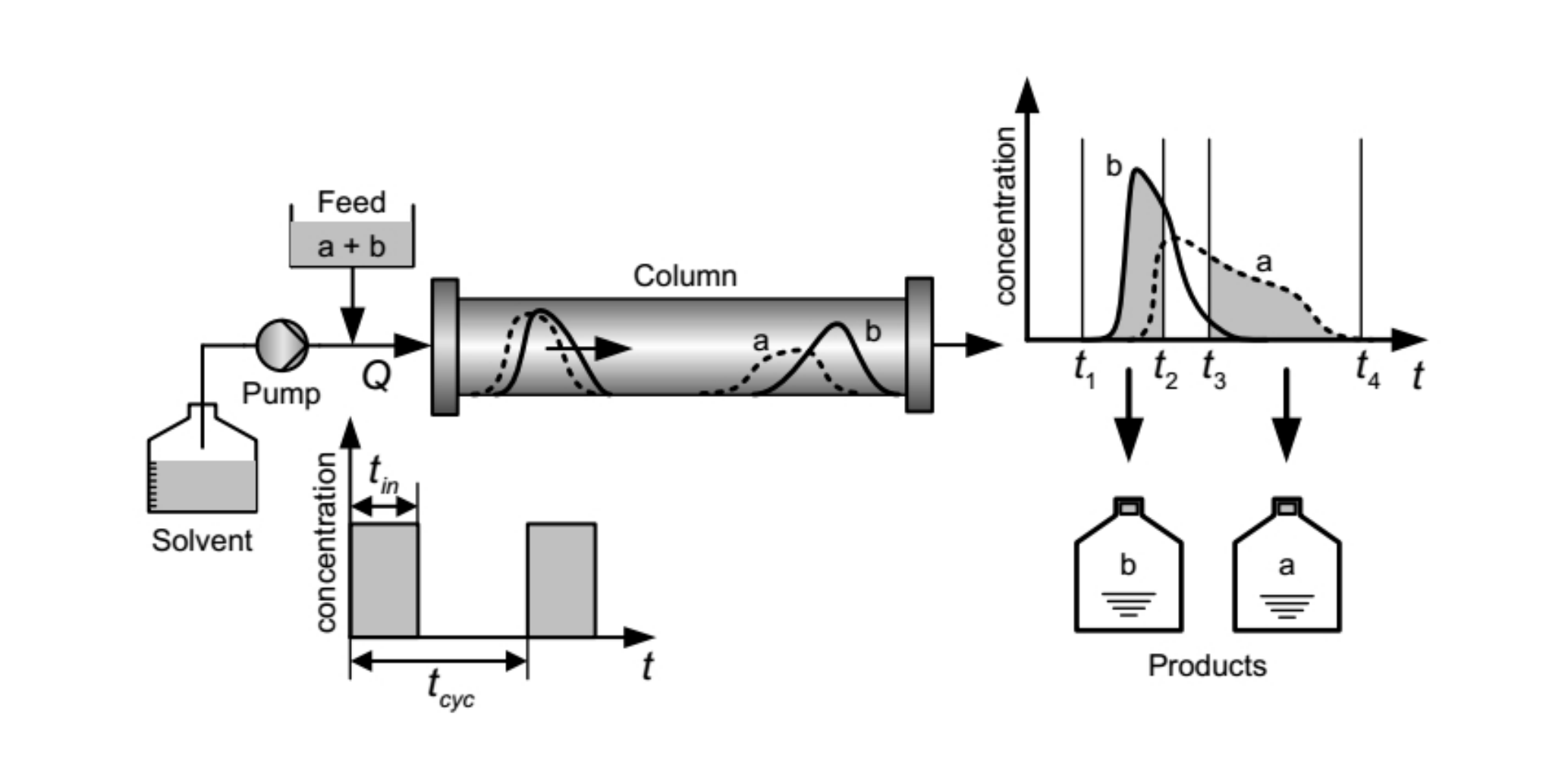}
	\caption{The schematic of a batch chromatographic process \cite{morZhaFLetal15}.}		
	\label{fig:chromatography}
\end{figure}
\begin{equation}
\begin{aligned}
\frac{\partial c_{z} }{\partial t} + \frac{1 - \epsilon}{\epsilon} \frac{\partial q_{z}}{\partial t} &= - \frac{\partial c_{z}}{\partial x} + \frac{1}{Pe} \frac{\partial^{2} c_{z}}{\partial x^{2}},\\ 
\frac{\partial q_{z}}{\partial t} &= \frac{L}{Q/\epsilon A_{c}} \kappa_{z} (q_{z}^{\text{eq}} - q_{z}),
\end{aligned}
\label{eqn:crystallizer_model}
\end{equation}
where
\begin{itemize}
	\item $c_{z}, q_{z}$ are the concentrations of the two components $(z = a,b)$ to be separated, in the solid and the liquid phase, respectively, 
	\item $\epsilon$ is the column porosity, $Pe$ is the Peclet number, $L$ is the length of the column, $Q$ is the volumetric feed flow rate, $A_{c}$ is the cross-section area and $\kappa_{z}$ the mass transfer coefficient,
	\item $q_{z}^{\text{eq}}$ is the adsorption equilibrium and it is the nonlinear term in the equation. It can be given as,
	\begin{align*}
	q_{z}^{\text{eq}} = \frac{H_{z1} c_{z} }{1 + K_{a1} c_{a}^{f} c_{a} + K_{b1} c_{b}^{f} c_{b} } + \frac{H_{z2} c_{z} }{1 + K_{a2} c_{a}^{f} c_{a} + K_{b2} c_{b}^{f} c_{b} },
	\end{align*}	
	\item $H_{z1}, H_{z2}$ are Henry's constants, $K_{z1}, K_{z2}$ are thermodynamic coefficients and $c_{z}^{f}$ is the feed concentration of each component $(z = a,b)$.
\end{itemize}
The PDE defining the batch chromatographic process is discretized using the finite volume method in space and a Crank-Nicolson scheme in time. Thus we have
\begin{equation}
\begin{aligned}
Ec_{z}^{k+1} &= Ac_{z}^{k} + d_{z}^{k} - \frac{1 - \epsilon}{\epsilon} \Delta t h_{z}^{k} ,\\
q_{z}^{k+1} &= q_{z}^{k} + \Delta t h_{z}^{k}.
\end{aligned}
\end{equation}
Here, $E, A$ are both non-parametric tri-diagonal matrices. The injection time $t_{in} \in [0.5, 2.0]$ and the volumetric feed flow rate $Q \in [0.0667, 0.1667]$ are the two parameters of interest. $d_{z}^{k}$ depends on $t_{in}$.
In particular,
\[  d_{z}^{k} = d_{0}^{k} (1, 0 , 0 , \ldots, 0)^{T}, \]
where $d_{0}^{k} \coloneqq \Delta x Pe \bigg( \frac{\Delta t}{2 \Delta x} + \frac{\Delta t}{Pe \Delta x^{2}}   \bigg) \chi (t^{k})$. The term $\chi(t^{k})$ can be given as
\[ \chi(t^{k}) = 1 \,  \text{if} \, \,  t^{k} \in [0 \,, t_{in}], \text{else} \, \,  0. \]
$h_{z}^{k}$ just corresponds to the right hand side of the second equation in \cref{eqn:crystallizer_model}, showing its dependency on $Q$.
Additionally, we note that for this example, we use the Adaptive Snapshot Selection (AdSS) technique \cite{morBenFetal15} to reduce the computational cost of SVD inside the DEIM algorithm and the greedy loop. 
\paragraph*{Adaptive Snapshot Selection (AdSS).}
For models where the number of time steps is large, it is often cumbersome to perform the SVD on such a large snapshot matrix. AdSS serves as a pre-treatment step to avoid this difficulty. It is essentially an algorithm to determine the linear dependency of successive vectors in the snapshot matrix. The angle between a new vector and the last selected vector is evaluated. If it falls below a tolerance it means that the new vector is almost linearly dependent on the last selected vector and thus can be discarded. For more details on this approach the reader is referred to \cite{morZhaFLetal14}. Finally, a much thinner snapshot matrix is obtained, reducing the costs of SVD in either the DEIM or the POD-Greedy algorithm.
\begin{table}[t]
	\caption{Simulation parameters for the batch chromatography  equation.}
	\vspace{-5mm}	
	\begin{center}
		\begin{tabular}{|c|c|c|c|c|}
			\hline
			Interpolation & $N$ & Parameter training set $\Xi$ & tolerance & $\epsilon_{\text{POD}}$\\
			\hline
			\multirow{2}{*}{DEIM} & \multirow{2}{*}{1000} & 60 samples uniformly distributed & \multirow{2}{*}{$10^{-4}$} & \multirow{2}{*}{$10^{-10}$}\\
			&& in $[0.0667, 0.1667 ] \times [0.5, 2.0]$ && \\
			\hline
		\end{tabular}			
	\end{center}
	\label{tab:batchchrom}
\end{table}	 
\newline
\indent
The parameters used for the simulations are shown in \cref{tab:batchchrom}. 
The system dimension is $N=1000$, where the finite volume method is used for space discretization. The Lax-Friedrichs flux is used for the convection term, while a central difference scheme is applied for the diffusion term. For details we refer to \cite{morZhaFLetal14}. A semi-implicit discretization is considered here for the time variable. We take 60 uniformly distributed parameters $\mu \coloneqq (Q, t_{in})$ from the parameter domain. The simulation time varies depending on the choice of $Q$. This is due to the fact that the volumetric feed flow ($Q$) determines the speed of flow of the components through the column. This speed in turn determines the time for which the model needs to be simulated per switching cycle.
The tolerance for the ROM error is set as $10^{-4}$. Since EIM is used for the Burger's equation, for this example, we use DEIM as the interpolation method, to show the flexibility of the adaptive algorithm. We set $\epsilon_{\text{POD}}$ to be $10^{-10}$. 

The dual system is a linear algebraic system without parameters, therefore, we 
propose to use the Krylov method \texttt{GMRES} (with the same configurations as done for the FBC example in Section 5.1) to compute the approximate solution to the dual system. The results are to be compared with the one obtained using the primal reduced basis to reduce the dual system. 

As for the Burgers' equation, for the batch chromatography example, we show the convergence of the modified error indicator and the corresponding true error, taken as the maximum over all parameters in $\Xi$, for each iteration. As can be seen from \cref{fig:batchchrom_convergence}, the proposed Adaptive POD-Greedy-(D)EIM algorithm results in a quicker convergence as compared to the Standard POD-Greedy. The Standard POD-Greedy results in a ROM of size $(\ell_{\text{RB}}, \ell_{\text{EI}}) = (47, 109)$ in 47 iterations, whereas the proposed Adaptive POD-Greedy method leads to a ROM of size $(\ell_{\text{RB}}, \ell_{\text{EI}}) = (46,50)$ in 29 iterations. \cref{fig:batchchrom_basis_iteration} shows the number of RB, DEIM basis vectors as a function of the iteration number, while \cref{fig:batchchrom_basis} shows the adaptive increase of the RB, DEIM basis vectors. Similar to the Burgers' equation, the largest jumps are in the first few steps when the error is large.
In \cref{batchchrom:effectivity_orig_mod}, we compare the effectivities of the original and the modified error indicators. The modified indicator offers more efficient results. This is due to the combination of two facts,
\begin{itemize}
	\item the use of \texttt{GMRES} for solving the dual system leads to a faster decay of the dual residual norm. 
	\item the second term of the modified error indicator is multiplied by the term $1 - \bar{\rho}$. From \cref{fig:batchchrom_rho} we can see that $\bar{\rho}$ tends to one as the iteration proceeds.
\end{itemize}
	\setlength\fheight{6cm}
	\setlength\fwidth{6cm}
\begin{figure}[t!]
	\centering	
	\input{figures/rb_batchchrom_convergence.tex}
	\caption{Batch chromatography: Convergence of the greedy algorithms: \cref{alg_podgreedy_ei,,alg_adap_podgreedy}. }
	\label{fig:batchchrom_convergence}
\end{figure}
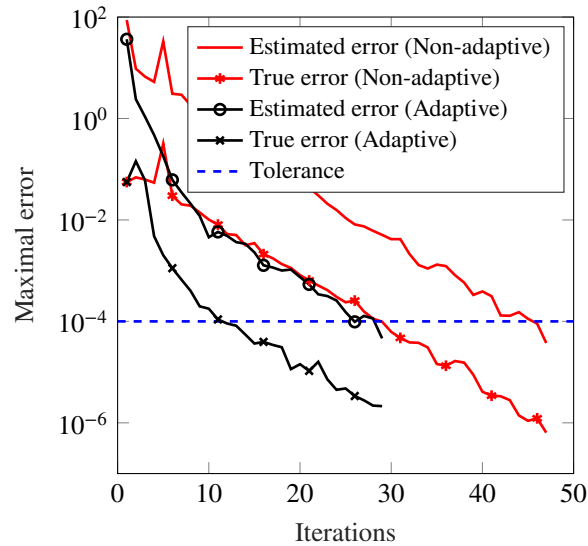
\setlength\fheight{6cm}
\setlength\fwidth{6cm}
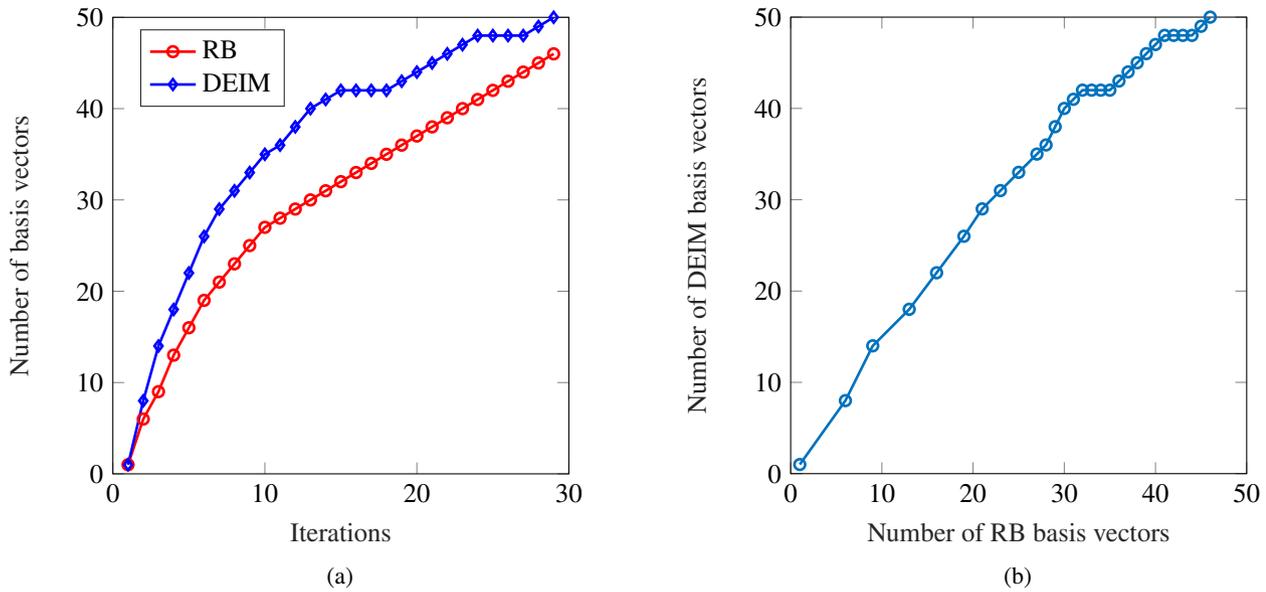
\begin{figure}[t!]
	\begin{subfigure}[b]{0.5\textwidth}		
		\input{figures/rb_batchchrom_basis_vs_iteration.tex}
		\caption{}
		\label{fig:batchchrom_basis_iteration}
	\end{subfigure}%
	\begin{subfigure}[b]{0.5\textwidth}	
		\input{figures/rb_batchchrom_basis.tex}
		\caption{}
		\label{fig:batchchrom_basis}
	\end{subfigure}%
	\caption{Batch chromatography example. (a) RB, DEIM basis vs iteration number. (b) Adaptive increment of RB vs DEIM basis vectors.}
\end{figure}
	\setlength\fheight{6cm}
\setlength\fwidth{6cm}
\begin{figure}[t!]
	\centering
	\begin{subfigure}[b]{0.5\textwidth}
		\captionsetup{justification=centering}		
		\input{figures/rb_batchchrom_effectivity.tex}
		\caption{}
		\label{batchchrom:effectivity_orig_mod}
	\end{subfigure}%
	\begin{subfigure}[b]{0.5\textwidth}
		\captionsetup{justification=centering}		
		\input{figures/rb_batchchrom_effectivity_modified_dual.tex}
		\caption{}
		\label{batchchrom:effectivity_dual}
	\end{subfigure}%
	\caption{Effectivity comparison inside \cref{alg_adap_podgreedy} for the batch chromatography example. (a) Effectivity: original vs modified, (b) Modified error indicator using the primal reduced basis vs \texttt{GMRES} for the dual system.}
	\label{batchchrom:effectivities}
\end{figure}
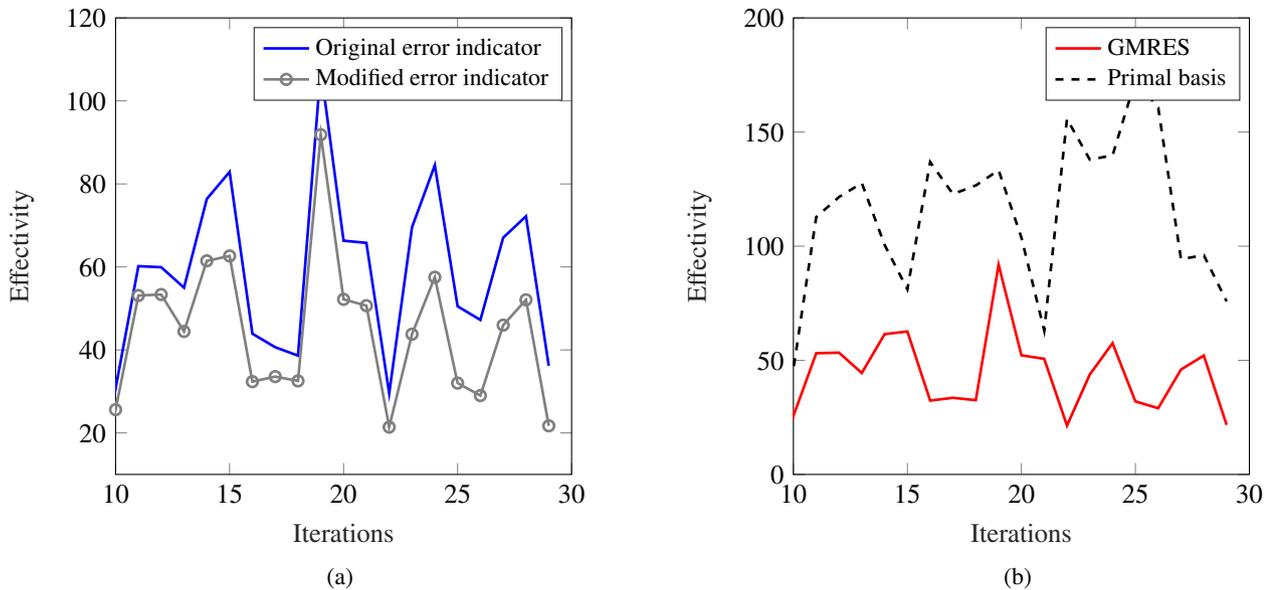
In \cref{batchchrom:effectivity_dual}, we compare the effectivities of the modified error indicator in two cases: in one case, we use the primal reduced basis to reduce the dual system and get the approximate dual solution $(\hat{x}_{\text{du}})$ from the reduced dual system;  while in the other case, we use the \texttt{GMRES} method to solve the dual system. It is clear that the case of using \texttt{GMRES} results in a smaller effectivity. This is due to the fact that the residual of the dual system is very small when using a Krylov space method to solve the dual system.
In \cref{fig:batchchrom_error_final}, we compare the effectivity of the original error indicator with the modified error indicator, for estimating the final ROM at the final iteration of \cref{alg_adap_podgreedy}. This shows that the modified error indicator is much sharper for the final ROM. Finally, it can be seen from \cref{tab:batchchrom_times} that the proposed adaptive approach is able to find a more compact ROM in a much shorter time. 
\begin{figure}[t!]
	\begin{subfigure}[b]{0.5\textwidth}
		\captionsetup{justification=centering}		
		\input{figures/rb_batchchrom_rho.tex}
		\caption{}
		\label{fig:batchchrom_rho}
	\end{subfigure}%
	\begin{subfigure}[b]{0.5\textwidth}
		\captionsetup{justification=centering}	
		\input{figures/rb_batchchrom_finalstep_org_vs_modified_effectivity.tex}
		\caption{}
		\label{fig:batchchrom_error_final}
	\end{subfigure}%
	\caption{Results of the adaptive process for batch chromatography. (a) The value of $\bar{\rho}(\mu^{*})$ at selected $\mu^{*}$ at each iteration of \cref{alg_adap_podgreedy}. (b) Effectivities of the two error indicators in \cref{error_split_phi} and \cref{error_split_psi}, respectively.}
\end{figure}
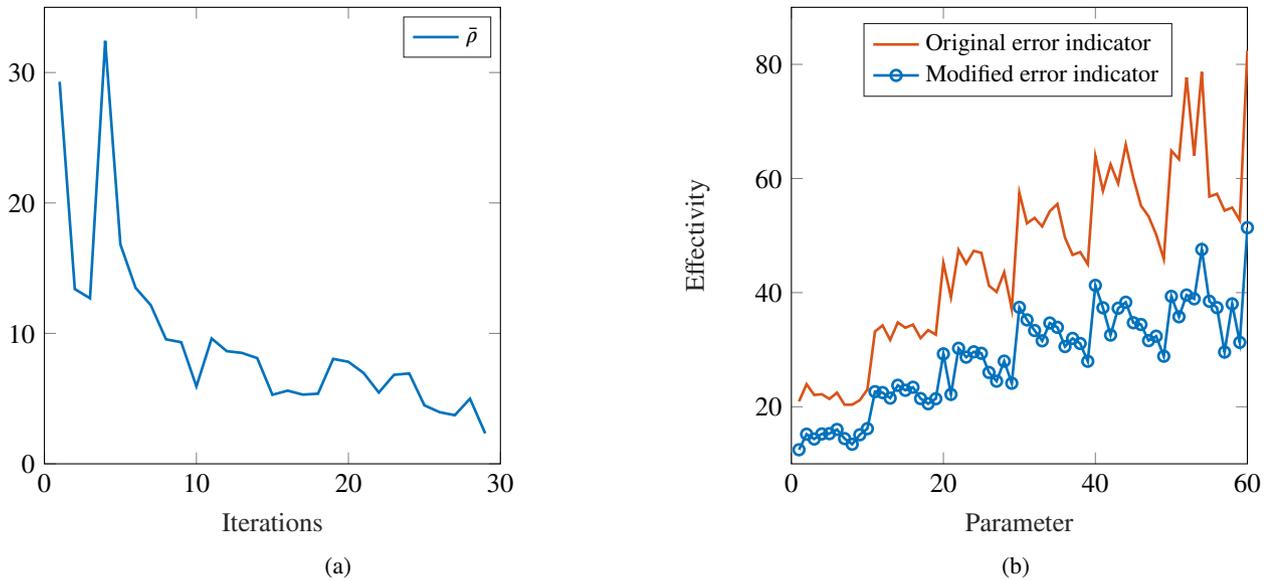
\begin{table}[t!]
	\caption{Runtime comparison for the batch chromatography example.}
	\vspace{-5mm}	
	\begin{center}	
		\begin{tabular}{|c|c|}
			\hline
			Method    & runtime (s) \\ \hline
			Adaptive & 7140         \\ \hline
			Non-adaptive       & 11260        \\ \hline
		\end{tabular}		
	\end{center}
	\label{tab:batchchrom_times}
\end{table}	 
\section{Conclusions}
In this work, we propose an adaptive scheme to generate the RB and (D)EIM basis. A good balance between the approximations of the state and the nonlinear term is achieved, while obtaining a ROM of desired tolerance. Whenever necessary, the proposed scheme adaptively adds new basis vectors to, or removes redundant basis vectors from the existing RB, (D)EIM projection matrices, in order to obtain a compact ROM. Expensive FOM simulations at all samples of the training set is avoided. The adaptive scheme is driven by a suitable \textit{a posteriori} error indicator, which makes use of an appropriate dual system solver. We have tested the adaptive approach on several examples from applications. The method is shown to work successfully for both non-parametric and parametric systems. The Adaptive POD-Greedy-(D)EIM scheme is able to deliver a ROM with much fewer iterations when compared to the standard approach. Also, we demonstrated that the proposed modified error indicator offers a better effectivity when compared to a similar indicator. The current error indicator is based on a semi-implicit time discretization. Future work would be to derive an error indicator for implicit time integration schemes involving nonlinear solvers, such as the Newton method.

\bibliography{refs}%

\end{document}

%% file: figures/fbc_convergence.tex
%
%
\definecolor{mycolor1}{rgb}{0.00000,0.44700,0.74100}%
\definecolor{mycolor2}{rgb}{0.85000,0.32500,0.09800}%
\begin{tikzpicture}

\begin{axis}[%
width=\fwidth,
height=\fheight,
at={(0\fwidth,0\fheight)},
scale only axis,
xmin=1,
xmax=9,
xlabel style={font=\color{white!15!black}},
xlabel={Iterations},
ymode=log,
ymin=1e-06,
ymax=0.1,
yminorticks=false,
ylabel style={font=\color{white!15!black}},
ylabel={Mean Error},
axis background/.style={fill=white},
legend style={legend cell align=left, align=left, draw=white!15!black, font = \scriptsize}
]
\addplot [color=red, line width=1.0pt,mark=x]
  table[row sep=crcr]{%
1	0.0672673706292057\\
2	0.00984153527766395\\
3	0.00509244215053889\\
4	0.00188866010877315\\
5	0.000413481247617018\\
6	0.000202197804525145\\
7	0.000138918137168987\\
8	0.000112586954503399\\
9	6.69313580318424e-05\\
};
\addlegendentry{Modified error indicator}

\addplot [color=red, line width=1.0pt,mark=o, mark options={solid, red}]
  table[row sep=crcr]{%
1	0.00286443108421125\\
2	0.000387933986807301\\
3	0.000197623447045129\\
4	5.49249487649395e-05\\
5	1.53518088931506e-05\\
6	5.33976479576669e-06\\
7	4.52387570457589e-06\\
8	4.55383111437439e-06\\
9	2.86136699108319e-06\\
};
\addlegendentry{True error}
\addplot [color=blue, dashed, line width=1.0pt]
table[row sep=crcr]{%
	0	1e-04\\
	1	1e-04\\
	2	1e-04\\
	3	1e-04\\
	4	1e-04\\
	5	1e-04\\
	6	1e-04\\
	7	1e-04\\
	8	1e-04\\
	9	1e-04\\
};
\addlegendentry{Tolerance}
\end{axis}
\end{tikzpicture}%

%% file: figures/fbc_effectivity.tex
%
%
\definecolor{mycolor1}{rgb}{0.00000,0.44700,0.74100}%
\definecolor{mycolor2}{rgb}{0.85000,0.32500,0.09800}%
\begin{tikzpicture}

\begin{axis}[%
width=\fwidth,
height=\fheight,
at={(0\fwidth,0\fheight)},
scale only axis,
xmin=1,
xmax=9,
xlabel style={font=\color{white!15!black}},
xlabel={Iterations},
ymin=20,
ymax=45,
ylabel style={font=\color{white!15!black}},
ylabel={Effectivity},
axis background/.style={fill=white},
legend style={at={(0.02,0.83)}, anchor=south west, legend cell align=left, align=left, draw=white!15!black, font = \scriptsize}
]
\addplot [color=blue, line width=1.0pt]
  table[row sep=crcr]{%
1	24.0751554791002\\
2	26.9726754149636\\
3	27.0716543308674\\
4	38.2282462906597\\
5	30.7198303918254\\
6	42.4243246878178\\
7	35.2115210204819\\
8	29.1648835811223\\
9	29.5000336512586\\
};
\addlegendentry{Original error indicator}

\addplot [color=gray, mark=o, mark options={solid, gray}, line width=1.0pt]
  table[row sep=crcr]{%
1	23.4836756939218\\
2	25.3690978680673\\
3	25.7684107158397\\
4	34.3861970059542\\
5	26.9337151403374\\
6	37.8664252563045\\
7	30.707770557991\\
8	24.7235682825419\\
9	23.3913923800823\\
};
\addlegendentry{Modified error indicator}

\end{axis}
\end{tikzpicture}%

%% file: figures/error_landscape_plot_3d_corrected_birdseye.tex
%
%
\begin{tikzpicture}

\begin{axis}[%
width=\fwidth,
height=\fheight,
at={(0\fwidth,0\fheight)},
scale only axis,
point meta min=-4.63406679142332,
point meta max=1,
xmin=3,
xmax=18,
xtick = {3,8,12,16},
tick align=outside,
xlabel style={font=\color{white!15!black}},
xlabel={$\ell_{\text{RB}}$},
ymin=6,
ymax=20,
ytick = {8,12,16,20},
ylabel style={font=\color{white!15!black}},
ylabel={$\ell_{\text{EI}}$},
zmin=-4.7,
zmax=2,
ztick = {-4,-2,0,2},
zlabel style={font=\color{white!15!black}},
zlabel={$\log_{10}$(Mean Error)},
view={155.5}{52},
axis background/.style={fill=white},
xmajorgrids,
ymajorgrids,
zmajorgrids,
legend style={at={(0.02,0.893)}, anchor=south west, legend cell align=left, align=left, draw=white!15!black},
colormap/blackwhite,
colorbar horizontal
]

\addplot3[%
surf,
shader=flat mean,z buffer = auto, draw=black, colormap/blackwhite, mesh/rows=15]
table[row sep=crcr, point meta=\thisrow{c}] {%
x	y	z	c\\
3	6	-1.02325374248032	-1.02325374248032\\
3	7	-1.13904084991373	-1.13904084991373\\
3	8	-1.17219554784187	-1.17219554784187\\
3	9	-1.05838917722755	-1.05838917722755\\
3	10	-1.07030651423085	-1.07030651423085\\
3	11	-1.0704228378883	-1.0704228378883\\
3	12	-1.07469409144385	-1.07469409144385\\
3	13	-1.07548258896785	-1.07548258896785\\
3	14	-1.07237872399647	-1.07237872399647\\
3	15	-1.07767556887875	-1.07767556887875\\
3	16	-1.07228137447373	-1.07228137447373\\
3	17	-1.06915524128146	-1.06915524128146\\
3	18	-1.07297312145958	-1.07297312145958\\
3	19	-1.07635956308971	-1.07635956308971\\
3	20	-1.10381068795302	-1.10381068795302\\
3	21	-1.10558732560946	-1.10558732560946\\
4	6	-1.48276303876073	-1.48276303876073\\
4	7	-1.50701219047835	-1.50701219047835\\
4	8	-1.37947264925339	-1.37947264925339\\
4	9	-1.47561285783382	-1.47561285783382\\
4	10	-1.44406812367556	-1.44406812367556\\
4	11	-1.44396332075932	-1.44396332075932\\
4	12	-1.5189661996133	-1.5189661996133\\
4	13	-1.51986008915055	-1.51986008915055\\
4	14	-1.52353946094491	-1.52353946094491\\
4	15	-1.52740234414916	-1.52740234414916\\
4	16	-1.52521142882864	-1.52521142882864\\
4	17	-1.52231447226521	-1.52231447226521\\
4	18	-1.53029018498049	-1.53029018498049\\
4	19	-1.53141844078932	-1.53141844078932\\
4	20	-1.57826618310747	-1.57826618310747\\
4	21	-1.57340650279687	-1.57340650279687\\
5	6	-1.5334934776995	-1.5334934776995\\
5	7	-1.56014105364615	-1.56014105364615\\
5	8	-1.20451608007776	-1.20451608007776\\
5	9	-1.7197367048589	-1.7197367048589\\
5	10	-1.71900415876413	-1.71900415876413\\
5	11	-1.71906041768758	-1.71906041768758\\
5	12	-1.75623113905359	-1.75623113905359\\
5	13	-1.76158438601305	-1.76158438601305\\
5	14	-1.76341544403656	-1.76341544403656\\
5	15	-1.72370975907316	-1.72370975907316\\
5	16	-1.75848451329327	-1.75848451329327\\
5	17	-1.80860593048752	-1.80860593048752\\
5	18	-1.79802842901632	-1.79802842901632\\
5	19	-1.77934240713453	-1.77934240713453\\
5	20	-1.75503408455579	-1.75503408455579\\
5	21	-1.73535176394914	-1.73535176394914\\
6	6	-1.26409195080101	-1.26409195080101\\
6	7	-1.66853167591151	-1.66853167591151\\
6	8	-0.843929616253993	-0.843929616253993\\
6	9	-2.00922409594802	-2.00922409594802\\
6	10	-2.0068657143923	-2.0068657143923\\
6	11	-2.00693714642562	-2.00693714642562\\
6	12	-1.97702644336335	-1.97702644336335\\
6	13	-2.00080371760487	-2.00080371760487\\
6	14	-1.96969771755367	-1.96969771755367\\
6	15	-1.94305889001593	-1.94305889001593\\
6	16	-1.99372521536313	-1.99372521536313\\
6	17	-2.0291640877767	-2.0291640877767\\
6	18	-2.01814334402383	-2.01814334402383\\
6	19	-1.99498050072071	-1.99498050072071\\
6	20	-1.84683747108251	-1.84683747108251\\
6	21	-1.84752696171104	-1.84752696171104\\
7	6	-0.962454136820727	-0.962454136820727\\
7	7	-1.81999168094451	-1.81999168094451\\
7	8	-1.82931468443761	-1.82931468443761\\
7	9	1	1\\
7	10	1	1\\
7	11	1	1\\
7	12	-1.94847372552383	-1.94847372552383\\
7	13	-1.93829942266626	-1.93829942266626\\
7	14	-1.94889313335659	-1.94889313335659\\
7	15	-1.96265307143258	-1.96265307143258\\
7	16	-1.97853690341629	-1.97853690341629\\
7	17	-1.976935176452	-1.976935176452\\
7	18	-1.97432718235884	-1.97432718235884\\
7	19	-1.96054493152223	-1.96054493152223\\
7	20	-1.94296774899282	-1.94296774899282\\
7	21	-1.97811284814906	-1.97811284814906\\
8	6	1	1\\
8	7	-0.672532936983513	-0.672532936983513\\
8	8	1	1\\
8	9	1	1\\
8	10	1	1\\
8	11	1	1\\
8	12	-2.31531910132617	-2.31531910132617\\
8	13	-2.29307389582695	-2.29307389582695\\
8	14	-2.31986129479256	-2.31986129479256\\
8	15	-2.31578734062746	-2.31578734062746\\
8	16	-2.36932305814837	-2.36932305814837\\
8	17	-2.37764509714851	-2.37764509714851\\
8	18	-2.37620982970481	-2.37620982970481\\
8	19	-2.35766396432713	-2.35766396432713\\
8	20	-2.34038925444058	-2.34038925444058\\
8	21	-2.4538564789457	-2.4538564789457\\
9	6	1	1\\
9	7	1	1\\
9	8	1	1\\
9	9	1	1\\
9	10	-1.71065373524181	-1.71065373524181\\
9	11	-1.76077701803929	-1.76077701803929\\
9	12	-2.54814893869168	-2.54814893869168\\
9	13	-2.55430159333187	-2.55430159333187\\
9	14	-2.46178111767969	-2.46178111767969\\
9	15	-2.42844560107741	-2.42844560107741\\
9	16	-2.5819779653918	-2.5819779653918\\
9	17	-2.60918510829888	-2.60918510829888\\
9	18	-2.59611159946021	-2.59611159946021\\
9	19	-2.56467806559353	-2.56467806559353\\
9	20	-2.63919661759721	-2.63919661759721\\
9	21	-2.7424312948642	-2.7424312948642\\
10	6	1	1\\
10	7	1	1\\
10	8	-0.385044889658935	-0.385044889658935\\
10	9	-1.4937911391457	-1.4937911391457\\
10	10	-2.05517243507897	-2.05517243507897\\
10	11	-2.0780965582289	-2.0780965582289\\
10	12	-2.78126232439427	-2.78126232439427\\
10	13	-2.75774030260016	-2.75774030260016\\
10	14	-2.79480886948818	-2.79480886948818\\
10	15	-2.72384619251014	-2.72384619251014\\
10	16	-2.89486978610154	-2.89486978610154\\
10	17	-2.91950268051771	-2.91950268051771\\
10	18	-2.90327826602499	-2.90327826602499\\
10	19	-2.8970652800689	-2.8970652800689\\
10	20	-2.93663128630852	-2.93663128630852\\
10	21	-2.90239573351753	-2.90239573351753\\
11	6	1	1\\
11	7	-0.305431040465166	-0.305431040465166\\
11	8	1	1\\
11	9	-2.3423244905407	-2.3423244905407\\
11	10	-2.60375922707712	-2.60375922707712\\
11	11	-2.59911305305924	-2.59911305305924\\
11	12	-2.99715095969123	-2.99715095969123\\
11	13	-3.19865010256304	-3.19865010256304\\
11	14	-2.92264550982494	-2.92264550982494\\
11	15	-2.79519857789668	-2.79519857789668\\
11	16	-3.25554297810827	-3.25554297810827\\
11	17	-3.29320397669878	-3.29320397669878\\
11	18	-3.28765409949458	-3.28765409949458\\
11	19	-3.18819633249092	-3.18819633249092\\
11	20	-3.32441816244652	-3.32441816244652\\
11	21	-3.25812989122427	-3.25812989122427\\
12	6	1	1\\
12	7	1	1\\
12	8	1	1\\
12	9	-1.99547692115529	-1.99547692115529\\
12	10	-2.62344603284101	-2.62344603284101\\
12	11	-2.58381688037769	-2.58381688037769\\
12	12	-3.0574099204861	-3.0574099204861\\
12	13	-3.21177093831632	-3.21177093831632\\
12	14	-3.2826037622586	-3.2826037622586\\
12	15	-3.28076462841461	-3.28076462841461\\
12	16	-3.38354418197873	-3.38354418197873\\
12	17	-3.4054114697936	-3.4054114697936\\
12	18	-3.35747281833425	-3.35747281833425\\
12	19	-3.34642318466013	-3.34642318466013\\
12	20	-3.37247622480232	-3.37247622480232\\
12	21	-3.44332698522827	-3.44332698522827\\
13	6	-0.771695595704834	-0.771695595704834\\
13	7	-1.17752727947135	-1.17752727947135\\
13	8	-0.939580054199879	-0.939580054199879\\
13	9	-0.260218006427799	-0.260218006427799\\
13	10	-2.20720758776763	-2.20720758776763\\
13	11	-2.30905400231592	-2.30905400231592\\
13	12	-3.10907297159984	-3.10907297159984\\
13	13	-3.1163875322109	-3.1163875322109\\
13	14	-3.68818609381294	-3.68818609381294\\
13	15	-3.66462692113446	-3.66462692113446\\
13	16	-3.73339316313135	-3.73339316313135\\
13	17	-3.69422356431273	-3.69422356431273\\
13	18	-3.75516897176106	-3.75516897176106\\
13	19	-3.79664532658652	-3.79664532658652\\
13	20	-3.82391531245896	-3.82391531245896\\
13	21	-3.77556176485437	-3.77556176485437\\
14	6	1	1\\
14	7	-1.11003313718227	-1.11003313718227\\
14	8	1	1\\
14	9	-0.244477621471531	-0.244477621471531\\
14	10	-2.29249390565912	-2.29249390565912\\
14	11	-2.34906889959915	-2.34906889959915\\
14	12	-2.51692284179656	-2.51692284179656\\
14	13	-0.362909708378746	-0.362909708378746\\
14	14	-3.58423931308678	-3.58423931308678\\
14	15	-3.74719954522132	-3.74719954522132\\
14	16	-3.93117981056399	-3.93117981056399\\
14	17	-3.92569068509217	-3.92569068509217\\
14	18	-3.85724104901989	-3.85724104901989\\
14	19	-3.86060390351251	-3.86060390351251\\
14	20	-3.97165865665588	-3.97165865665588\\
14	21	-3.94431509434138	-3.94431509434138\\
15	6	-0.610403239208101	-0.610403239208101\\
15	7	-1.43018739983995	-1.43018739983995\\
15	8	-1.44705163440893	-1.44705163440893\\
15	9	-2.12031717908751	-2.12031717908751\\
15	10	-2.22429183093576	-2.22429183093576\\
15	11	0.110271512914167	0.110271512914167\\
15	12	1	1\\
15	13	1	1\\
15	14	-3.52261655116	-3.52261655116\\
15	15	-3.75088365598895	-3.75088365598895\\
15	16	-3.94522738110903	-3.94522738110903\\
15	17	-4.00970915954228	-4.00970915954228\\
15	18	-3.90369714841057	-3.90369714841057\\
15	19	-3.94851192844925	-3.94851192844925\\
15	20	-4.12004428154868	-4.12004428154868\\
15	21	-4.0067037234617	-4.0067037234617\\
16	6	-0.709901568825534	-0.709901568825534\\
16	7	-1.39505800145019	-1.39505800145019\\
16	8	-1.4496993599417	-1.4496993599417\\
16	9	-2.14571619899733	-2.14571619899733\\
16	10	-2.28924998486995	-2.28924998486995\\
16	11	0.054001040619904	0.054001040619904\\
16	12	1	1\\
16	13	1	1\\
16	14	-3.50950668925724	-3.50950668925724\\
16	15	-3.71464420867652	-3.71464420867652\\
16	16	-3.96305511569718	-3.96305511569718\\
16	17	-4.08008051551925	-4.08008051551925\\
16	18	-4.00311329787612	-4.00311329787612\\
16	19	-3.92344150018831	-3.92344150018831\\
16	20	-4.17437036310718	-4.17437036310718\\
16	21	-3.97241757650121	-3.97241757650121\\
18	6	-0.648632332576295	-0.648632332576295\\
18	7	-1.16153135155904	-1.16153135155904\\
18	8	-1.2932973141065	-1.2932973141065\\
18	9	-2.04419019385916	-2.04419019385916\\
18	10	-2.23771565430504	-2.23771565430504\\
18	11	-2.24343860810196	-2.24343860810196\\
18	12	-3.01297861228774	-3.01297861228774\\
18	13	-2.52027730060783	-2.52027730060783\\
18	14	-2.80565062790064	-2.80565062790064\\
18	15	-2.75468422630697	-2.75468422630697\\
18	16	-3.34924921241088	-3.34924921241088\\
18	17	-3.49506151631604	-3.49506151631604\\
18	18	-4.10720398174065	-4.10720398174065\\
18	19	-2.77574205162151	-2.77574205162151\\
18	20	-4.01562361762525	-4.01562361762525\\
18	21	-4.63406679142332	-4.63406679142332\\
};
\addlegendentry{Error Landscape}

\addplot3[only marks, mark=*, mark options={}, mark size=2.5000pt, color=gray, fill=white] table[row sep=crcr]{%
x	y	z\\
3	8	-1.17219554784187\\
6	11	-2.00693714642562\\
8	13	-2.29307389582695\\
10	15	-2.72384619251014\\
12	16	-3.38354418197873\\
13	17	-3.69422356431273\\
14	18	-3.85724104901989\\
15	19	-3.94851192844925\\
16	20	-4.17437036310718\\
};
\addlegendentry{Adaptive choices}

\end{axis}
\end{tikzpicture}%

%% file: figures/fbc_convergence_decrease.tex
%
%
\definecolor{mycolor1}{rgb}{0.00000,0.44700,0.74100}%
\definecolor{mycolor2}{rgb}{0.85000,0.32500,0.09800}%
\begin{tikzpicture}

\begin{axis}[%
width=\fwidth,
height=\fheight,
at={(0\fwidth,0\fheight)},
scale only axis,
xmin=1,
xmax=7,
xlabel style={font=\color{white!15!black}},
xlabel={Iterations},
ymode=log,
ymin=1e-07,
ymax=1,
yminorticks=false,
ylabel style={font=\color{white!15!black}},
ylabel={Mean Error},
axis background/.style={fill=white},
legend style={at={(0.05,0.76)}, anchor=south west, legend cell align=left, align=left, draw=white!15!black, font = \scriptsize}
]
\addplot [color=red, line width=1.0pt,mark=x]
  table[row sep=crcr]{%
1	8.23093879889685e-06\\
2	5.52036981549891e-06\\
3	0.00144993550820055\\
4	0.018877529185802\\
5	0.0140814379728716\\
6	0.000100823923592323\\
7	5.86415432868762e-05\\
};
\addlegendentry{Modified error indicator}

\addplot [color=red, line width=1.0pt,mark=o, mark options={solid, red}]
  table[row sep=crcr]{%
1	4.15397283218613e-07\\
2	3.09831981291572e-07\\
3	7.28107062652784e-05\\
4	0.00100675952982042\\
5	0.000834027601854489\\
6	2.47396073931039e-06\\
7	2.06594050523347e-06\\
};
\addlegendentry{True error}
\addplot [color=blue, dashed, line width=1.0pt]
table[row sep=crcr]{%
	0	1e-04\\
	1	1e-04\\
	2	1e-04\\
	3	1e-04\\
	4	1e-04\\
	5	1e-04\\
	6	1e-04\\
	7	1e-04\\
	8	1e-04\\
	9	1e-04\\
};
\addplot [color=blue, dashed, line width=1.0pt]
table[row sep=crcr]{%
	0	1e-05\\
	1	1e-05\\
	2	1e-05\\
	3	1e-05\\
	4	1e-05\\
	5	1e-05\\
	6	1e-05\\
	7	1e-05\\
	8	1e-05\\
	9	1e-05\\
};
\addlegendentry{Zone of acceptance}

\end{axis}
\end{tikzpicture}%

%% file: figures/fbc_effectivity_decrease.tex
%
%
\definecolor{mycolor1}{rgb}{0.00000,0.44700,0.74100}%
\definecolor{mycolor2}{rgb}{0.85000,0.32500,0.09800}%
\begin{tikzpicture}

\begin{axis}[%
width=\fwidth,
height=\fheight,
at={(0\fwidth,0\fheight)},
scale only axis,
xmin=1,
xmax=7,
xlabel style={font=\color{white!15!black}},
xlabel={Iterations},
ymin=15,
ymax=45,
ylabel style={font=\color{white!15!black}},
ylabel={Effectivity},
axis background/.style={fill=white},
legend style={at={(0.05,0.806)}, anchor=south west, legend cell align=left, align=left, draw=white!15!black, font = \scriptsize}
]
\addplot [color=blue, line width=1.0pt]
  table[row sep=crcr]{%
1	22.8791926283667\\
2	23.3482227016157\\
3	20.587715797078\\
4	19.4381436258951\\
5	17.5096554052539\\
6	44.4089890621363\\
7	33.3952158162367\\
};
\addlegendentry{Original error indicator}

\addplot [color=gray, mark=o, mark options={solid, gray}, line width=1.0pt]
  table[row sep=crcr]{%
1	19.8146187551379\\
2	17.8173014692885\\
3	19.9137679411851\\
4	18.7507827109115\\
5	16.883659415541\\
6	40.7540515862945\\
7	28.3849138628749\\
};
\addlegendentry{Modified error indicator}

\end{axis}
\end{tikzpicture}%

%% file: figures/error_landscape_plot_3d_corrected_birdseye_decrease.tex
%
%
\begin{tikzpicture}

\begin{axis}[%
width=\fwidth,
height=\fheight,
at={(0\fwidth,0\fheight)},
scale only axis,
point meta min=-5.65946761698962,
point meta max=1,
xmin=16,
xmax=31,
xtick = {16,20,24,28,31},
tick align=outside,
xlabel style={font=\color{white!15!black}},
xlabel={$\ell_{\text{RB}}$},
ymin=25,
ymax=39,
ytick = {27,30,33,36,39},
ylabel style={font=\color{white!15!black}},
ylabel={$\ell_{\text{EI}}$},
zmin=-6,
zmax=2,
ztick = {-6, -2, 2},
zlabel style={font=\color{white!15!black}},
zlabel={$\log_{10}$(Mean Error)},
x dir=reverse,
y dir = reverse,
view={-30.5}{80},
%
axis background/.style={fill=white},
xmajorgrids,
ymajorgrids,
zmajorgrids,
legend style={at={(0.02,0.893)}, anchor=south west, legend cell align=left, align=left, draw=white!15!black},
colormap/blackwhite,
colorbar horizontal,
colorbar style = {xtick={-6,-5,...,1}}
]

\addplot3[%
surf,
shader=flat mean, z buffer = auto, draw=black, colormap/blackwhite, mesh/rows=15]
table[row sep=crcr, point meta=\thisrow{c}] {%
x	y	z	c\\
31	39	-5.08455062749391	-5.08455062749391\\
31	38	-4.82889626702629	-4.82889626702629\\
31	37	-4.90419403952846	-4.90419403952846\\
31	36	-4.87221828567251	-4.87221828567251\\
31	35	-4.880572427442	-4.880572427442\\
31	34	-4.98357227732457	-4.98357227732457\\
31	33	1	1\\
31	32	1	1\\
31	31	1	1\\
31	30	1	1\\
31	29	1	1\\
31	28	1	1\\
31	27	1	1\\
31	26	1	1\\
29	39	-5.50478596068582	-5.50478596068582\\
29	38	-5.25569837017008	-5.25569837017008\\
29	37	-5.3004791901123	-5.3004791901123\\
29	36	-5.19038511913399	-5.19038511913399\\
29	35	-5.22266948912669	-5.22266948912669\\
29	34	-5.10080821349788	-5.10080821349788\\
29	33	1	1\\
29	32	1	1\\
29	31	1	1\\
29	30	1	1\\
29	29	1	1\\
29	28	1	1\\
29	27	1	1\\
29	26	1	1\\
28	39	-5.65946761698962	-5.65946761698962\\
28	38	-5.36680086342549	-5.36680086342549\\
28	37	-5.38410595981967	-5.38410595981967\\
28	36	-5.25803182744278	-5.25803182744278\\
28	35	-5.30477560659481	-5.30477560659481\\
28	34	-5.15820480116745	-5.15820480116745\\
28	33	1	1\\
28	32	1	1\\
28	31	1	1\\
28	30	1	1\\
28	29	1	1\\
28	28	1	1\\
28	27	1	1\\
28	26	1	1\\
27	39	-5.60468810970501	-5.60468810970501\\
27	38	-5.2807180020778	-5.2807180020778\\
27	37	-5.23698138688929	-5.23698138688929\\
27	36	-5.16127122237391	-5.16127122237391\\
27	35	-5.20312767754494	-5.20312767754494\\
27	34	-5.11061793385156	-5.11061793385156\\
27	33	1	1\\
27	32	1	1\\
27	31	-1.11008209908739	-1.11008209908739\\
27	30	-4.06286272720499	-4.06286272720499\\
27	29	-4.00357933182886	-4.00357933182886\\
27	28	-4.16193866179409	-4.16193866179409\\
27	27	-3.45961406002951	-3.45961406002951\\
27	26	-1.00574446201055	-1.00574446201055\\
26	39	-5.18007859280303	-5.18007859280303\\
26	38	-5.01771656435891	-5.01771656435891\\
26	37	-5.01580665394912	-5.01580665394912\\
26	36	-5.0110290881288	-5.0110290881288\\
26	35	-5.01185941516824	-5.01185941516824\\
26	34	-5.00529226459444	-5.00529226459444\\
26	33	1	1\\
26	32	1	1\\
26	31	0.00824672753790531	0.00824672753790531\\
26	30	-1.22684032339669	-1.22684032339669\\
26	29	-1.57303389106139	-1.57303389106139\\
26	28	-2.84983518836612	-2.84983518836612\\
26	27	-2.72581307317531	-2.72581307317531\\
26	26	1	1\\
25	39	-5.17893995787353	-5.17893995787353\\
25	38	-4.95570748495411	-4.95570748495411\\
25	37	-4.96441397309437	-4.96441397309437\\
25	36	-4.96803482599375	-4.96803482599375\\
25	35	-4.97449009977386	-4.97449009977386\\
25	34	-4.96472351933639	-4.96472351933639\\
25	33	1	1\\
25	32	1	1\\
25	31	0.192716464503581	0.192716464503581\\
25	30	-1.18136469198744	-1.18136469198744\\
25	29	-1.4270324704453	-1.4270324704453\\
25	28	-2.44924148248948	-2.44924148248948\\
25	27	-2.56653744061087	-2.56653744061087\\
25	26	0.677851483400506	0.677851483400506\\
24	39	-5.27578506209704	-5.27578506209704\\
24	38	-5.31903544306211	-5.31903544306211\\
24	37	-5.30574681102731	-5.30574681102731\\
24	36	-5.31070549065126	-5.31070549065126\\
24	35	-5.34322837373808	-5.34322837373808\\
24	34	-5.36393795143917	-5.36393795143917\\
24	33	1	1\\
24	32	1	1\\
24	31	-0.34187408992718	-0.34187408992718\\
24	30	-2.90580152255869	-2.90580152255869\\
24	29	-2.79080327200499	-2.79080327200499\\
24	28	-3.77141977003426	-3.77141977003426\\
24	27	-4.58187558822111	-4.58187558822111\\
24	26	-4.84895334599937	-4.84895334599937\\
23	39	-5.64276201421677	-5.64276201421677\\
23	38	-5.63510972474738	-5.63510972474738\\
23	37	-5.59072767992205	-5.59072767992205\\
23	36	-5.57849680974033	-5.57849680974033\\
23	35	-5.65254210871531	-5.65254210871531\\
23	34	-5.58698237040999	-5.58698237040999\\
23	33	1	1\\
23	32	1	1\\
23	31	1	1\\
23	30	-2.0476109671476	-2.0476109671476\\
23	29	-1.97288766772728	-1.97288766772728\\
23	28	-3.06438060926766	-3.06438060926766\\
23	27	-4.37661063042559	-4.37661063042559\\
23	26	-4.98637787845988	-4.98637787845988\\
22	39	-5.28054950599896	-5.28054950599896\\
22	38	-5.26279390692263	-5.26279390692263\\
22	37	-5.23721160111308	-5.23721160111308\\
22	36	-5.235450932635	-5.235450932635\\
22	35	-5.25974112870759	-5.25974112870759\\
22	34	-5.27842517659807	-5.27842517659807\\
22	33	-2.83865131435504	-2.83865131435504\\
22	32	-2.77262976615216	-2.77262976615216\\
22	31	-4.29687857675017	-4.29687857675017\\
22	30	-4.46269415726942	-4.46269415726942\\
22	29	-4.41985856084573	-4.41985856084573\\
22	28	-4.67047614032346	-4.67047614032346\\
22	27	-4.77669900278356	-4.77669900278356\\
22	26	-4.87290841055576	-4.87290841055576\\
21	39	-5.07203713908596	-5.07203713908596\\
21	38	-5.04821371190611	-5.04821371190611\\
21	37	-5.07674396872475	-5.07674396872475\\
21	36	-5.07748027446688	-5.07748027446688\\
21	35	-5.08007142284546	-5.08007142284546\\
21	34	-5.07871318081333	-5.07871318081333\\
21	33	-1.69613257645681	-1.69613257645681\\
21	32	-1.72405484961388	-1.72405484961388\\
21	31	-3.3612581784276	-3.3612581784276\\
21	30	-4.48686673195954	-4.48686673195954\\
21	29	-4.03129868225715	-4.03129868225715\\
21	28	-4.5459530640674	-4.5459530640674\\
21	27	-4.75552132553262	-4.75552132553262\\
21	26	-4.65938061572383	-4.65938061572383\\
20	39	-4.96218706524699	-4.96218706524699\\
20	38	-4.9439241703844	-4.9439241703844\\
20	37	-4.98432515091041	-4.98432515091041\\
20	36	-4.98844647113806	-4.98844647113806\\
20	35	-4.99119662960758	-4.99119662960758\\
20	34	-4.96387038385813	-4.96387038385813\\
20	33	-0.324308454723164	-0.324308454723164\\
20	32	-0.287908722404269	-0.287908722404269\\
20	31	-1.85135299350388	-1.85135299350388\\
20	30	-4.16276298978767	-4.16276298978767\\
20	29	-3.77643859664947	-3.77643859664947\\
20	28	-4.45048935521091	-4.45048935521091\\
20	27	-4.89929735797578	-4.89929735797578\\
20	26	-4.81989174463823	-4.81989174463823\\
19	39	-4.7066645025039	-4.7066645025039\\
19	38	-4.6813945010954	-4.6813945010954\\
19	37	-4.75550740923455	-4.75550740923455\\
19	36	-4.76220166191998	-4.76220166191998\\
19	35	-4.76485312469578	-4.76485312469578\\
19	34	-4.70139988441165	-4.70139988441165\\
19	33	0.00110169621638641	0.00110169621638641\\
19	32	0.00448648434436201	0.00448648434436201\\
19	31	-0.724983432447409	-0.724983432447409\\
19	30	-3.99643640587288	-3.99643640587288\\
19	29	-3.32075037469409	-3.32075037469409\\
19	28	-4.1504670668939	-4.1504670668939\\
19	27	-4.36373979482579	-4.36373979482579\\
19	26	-4.78984527422265	-4.78984527422265\\
18	39	-4.76738448496646	-4.76738448496646\\
18	38	-4.75081350232317	-4.75081350232317\\
18	37	-4.80474427887399	-4.80474427887399\\
18	36	-4.81001525080273	-4.81001525080273\\
18	35	-4.81023737510622	-4.81023737510622\\
18	34	-4.75674522923272	-4.75674522923272\\
18	33	-4.31089434196027	-4.31089434196027\\
18	32	-4.28840893274317	-4.28840893274317\\
18	31	-4.59590590320183	-4.59590590320183\\
18	30	-4.37267065737302	-4.37267065737302\\
18	29	-4.42845756599719	-4.42845756599719\\
18	28	-4.44292247064906	-4.44292247064906\\
18	27	-4.51283040102885	-4.51283040102885\\
18	26	-4.77637246288058	-4.77637246288058\\
17	39	-4.54524431769954	-4.54524431769954\\
17	38	-4.48957396641608	-4.48957396641608\\
17	37	-4.55962882770378	-4.55962882770378\\
17	36	-4.56111141758306	-4.56111141758306\\
17	35	-4.56262785688936	-4.56262785688936\\
17	34	-4.51060872019199	-4.51060872019199\\
17	33	-4.20808971852902	-4.20808971852902\\
17	32	-4.19254723982814	-4.19254723982814\\
17	31	-4.44470033273242	-4.44470033273242\\
17	30	-4.19597423747007	-4.19597423747007\\
17	29	-4.2160739219361	-4.2160739219361\\
17	28	-4.23179460875992	-4.23179460875992\\
17	27	-4.3526980345572	-4.3526980345572\\
17	26	-4.56317852114664	-4.56317852114664\\
16	39	-4.15959871382949	-4.15959871382949\\
16	38	-4.14013301443113	-4.14013301443113\\
16	37	-4.15287343088787	-4.15287343088787\\
16	36	-4.14204705830454	-4.14204705830454\\
16	35	-4.13901774846651	-4.13901774846651\\
16	34	-4.14401738624762	-4.14401738624762\\
16	33	-3.83231234891717	-3.83231234891717\\
16	32	-3.78958133326175	-3.78958133326175\\
16	31	-4.10331672367995	-4.10331672367995\\
16	30	-3.9041113953889	-3.9041113953889\\
16	29	-3.88583158698749	-3.88583158698749\\
16	28	-3.90348941073285	-3.90348941073285\\
16	27	-4.1858001850623	-4.1858001850623\\
16	26	-4.20984378955893	-4.20984378955893\\
};
\addlegendentry{Error Landscape}

\addplot3[only marks, mark=*, mark options={}, mark size=2.5000pt, color=gray, fill=white] table[row sep=crcr]{%
x	y	z\\
31	39	-5.08455062749391\\
28	36	-5.25803182744278\\
22	33	-2.83865131435504\\
21	32	-1.72405484961388\\
20	31	-1.85135299350388\\
19	30	-3.99643640587288\\
17	28	-4.23179460875992\\
};
\addlegendentry{Adaptive choices}

\end{axis}
\end{tikzpicture}%

%% file: figures/fbc_final_ROM_error_compare.tex
%
%
\definecolor{mycolor1}{rgb}{0.00000,0.44700,0.74100}%
\definecolor{mycolor2}{rgb}{0.85000,0.32500,0.09800}%
\begin{tikzpicture}

\begin{axis}[%
width=\fwidth,
height=\fheight,
at={(0\fwidth,0\fheight)},
scale only axis,
xmin=0,
xmax=500,
xlabel style={font=\color{white!15!black}},
xlabel={Time (s)},
ymode=log,
ymin=1e-08,
ymax=0.01,
yminorticks=false,
ylabel style={font=\color{white!15!black}},
ylabel={Output error},
axis background/.style={fill=white},
legend style={legend cell align=left, align=left, draw=white!15!black, font = \scriptsize}
]
\addplot [color=blue, dashed,line width=1.0pt]
table[row sep=crcr]{%
	4	0.00153725233380292\\
	8	0.000201362706774701\\
	12	0.000338026236558367\\
	16	0.000201755252831545\\
	20	0.000179547120814056\\
	24	0.000139321354440808\\
	28	5.62881926452451e-05\\
	32	8.12538359530341e-05\\
	36	8.24971410954377e-05\\
	40	5.28538067211552e-05\\
	44	3.45215059349246e-05\\
	48	6.96056769393852e-05\\
	52	8.54678245542356e-05\\
	56	7.15825847694358e-05\\
	60	4.6306275364232e-05\\
	64	4.39358981933259e-05\\
	68	9.20902031858024e-05\\
	72	0.000117822862885997\\
	76	7.09536095973851e-05\\
	80	0.000122698155099615\\
	84	0.000122235597047863\\
	88	9.05964954555181e-05\\
	92	6.85323944405797e-05\\
	96	0.000116003751951246\\
	100	0.000138915072158684\\
	104	0.000111287878719382\\
	108	6.72769938577619e-05\\
	112	6.08533372326546e-05\\
	116	9.23817551267712e-05\\
	120	0.000130921473005998\\
	124	0.000133629124322808\\
	128	0.000100507577397858\\
	132	5.28075377884449e-05\\
	136	4.50809308030336e-05\\
	140	7.71752013598859e-05\\
	144	0.000102639545945801\\
	148	0.00010929199119191\\
	152	9.88285989793423e-05\\
	156	7.63781669803579e-05\\
	160	4.83210876697281e-05\\
	164	2.69066320826005e-05\\
	168	3.19923647916677e-05\\
	172	4.86232053655929e-05\\
	176	6.39237641264861e-05\\
	180	7.32603742306391e-05\\
	184	7.59201219711367e-05\\
	188	7.2411286011787e-05\\
	192	6.38248194211673e-05\\
	196	5.15856835146114e-05\\
	200	3.73884733432534e-05\\
	204	2.38114398446774e-05\\
	208	1.94032597412307e-05\\
	212	2.55217746227978e-05\\
	216	3.49358767333857e-05\\
	220	4.41998838921286e-05\\
	224	5.16159320835853e-05\\
	228	5.65446346902959e-05\\
	232	5.88127461584711e-05\\
	236	5.84872312718559e-05\\
	240	5.57815668773586e-05\\
	244	5.10059003167357e-05\\
	248	4.45371071888002e-05\\
	252	3.68057628808544e-05\\
	256	2.83167141564284e-05\\
	260	1.98281978986935e-05\\
	264	1.41478596173972e-05\\
	268	1.58690147491732e-05\\
	272	2.08091616691291e-05\\
	276	2.72308128332074e-05\\
	280	3.3942787484215e-05\\
	284	4.01417034126809e-05\\
	288	4.53901383723616e-05\\
	292	4.94382522506872e-05\\
	296	5.21325596062195e-05\\
	300	5.33774624640148e-05\\
	304	5.31176523696396e-05\\
	308	5.13289063284797e-05\\
	312	4.801293154014e-05\\
	316	4.31952530884245e-05\\
	320	3.69278509354016e-05\\
	324	2.93058508680006e-05\\
	328	2.0557483709708e-05\\
	332	1.22083113298555e-05\\
	336	1.31539190643454e-05\\
	340	2.08239403478774e-05\\
	344	3.3217095162539e-05\\
	348	4.78942843149311e-05\\
	352	6.39335955027946e-05\\
	356	6.96781465537402e-05\\
	360	6.11860285314198e-05\\
	364	5.15611238642773e-05\\
	368	4.24094668783452e-05\\
	372	3.4693279233484e-05\\
	376	2.88138134284369e-05\\
	380	2.47434442577107e-05\\
	384	2.21622009887929e-05\\
	388	2.061170054796e-05\\
	392	1.96513593037354e-05\\
	396	1.89587325655354e-05\\
	400	1.83480198630104e-05\\
	404	1.77420691189849e-05\\
	408	1.71431763867742e-05\\
	412	1.66232144630652e-05\\
	416	1.63310245950113e-05\\
	420	1.64790292132713e-05\\
	424	1.72229945468266e-05\\
	428	1.84996907110702e-05\\
	432	2.00804052723776e-05\\
	436	2.17308434606708e-05\\
	440	2.32654887193356e-05\\
	444	2.4536678446473e-05\\
	448	2.54186570779641e-05\\
	452	2.57987848636967e-05\\
	456	2.55743955139327e-05\\
	460	2.46538896019412e-05\\
	464	2.29633811653278e-05\\
	468	2.04663451674138e-05\\
	472	1.7224341409879e-05\\
	476	1.36018189841249e-05\\
	480	1.07694898457294e-05\\
	484	1.08245458606437e-05\\
	488	1.63966408556436e-05\\
	492	2.58602203069465e-05\\
	496	3.75389416610344e-05\\
	500	5.11205529696606e-05\\
};
\addlegendentry{Modified error indicator}
\addplot [color=red,line width=1.0pt]
  table[row sep=crcr]{%
4	6.56130085013862e-06\\
8	3.91757715417773e-06\\
12	2.25530792674622e-06\\
16	3.99193624561756e-06\\
20	3.01581971468012e-06\\
24	1.0501799102558e-06\\
28	4.25781719048945e-07\\
32	1.61863939285301e-06\\
36	3.51708331036704e-06\\
40	4.82827608377923e-06\\
44	4.98095349671379e-06\\
48	4.20049165156255e-06\\
52	3.12411587133354e-06\\
56	2.37679749914843e-06\\
60	2.30934205824873e-06\\
64	2.91315528844205e-06\\
68	3.86165618182499e-06\\
72	4.67679738924698e-06\\
76	5.29579172381389e-06\\
80	5.62458657826781e-06\\
84	5.59889219464704e-06\\
88	5.35365731957338e-06\\
92	5.08101127372207e-06\\
96	4.89243587020738e-06\\
100	4.76777370328474e-06\\
104	4.62326380812339e-06\\
108	4.42941337785996e-06\\
112	4.24464560799809e-06\\
116	4.42155273872036e-06\\
120	5.43943046826811e-06\\
124	6.5172099176225e-06\\
128	7.12833785776468e-06\\
132	7.02257780316184e-06\\
136	6.18945007191396e-06\\
140	4.84721756432993e-06\\
144	3.37095048241753e-06\\
148	2.1601833168905e-06\\
152	1.50536462306672e-06\\
156	1.51034699213959e-06\\
160	2.0908964213362e-06\\
164	3.03204674423352e-06\\
168	4.06957047160983e-06\\
172	4.96354570844737e-06\\
176	5.54564477339703e-06\\
180	5.73639680412352e-06\\
184	5.53850562712288e-06\\
188	5.01612941439245e-06\\
192	4.26940686892863e-06\\
196	3.41075782406275e-06\\
200	2.54642326047438e-06\\
204	1.76430652087944e-06\\
208	1.12768711857392e-06\\
212	6.73661928907165e-07\\
216	4.14965882655771e-07\\
220	3.43903774346899e-07\\
224	4.37331047487e-07\\
228	6.61866750184537e-07\\
232	9.78762715098114e-07\\
236	1.34806816765032e-06\\
240	1.73190831420644e-06\\
244	2.09683437479846e-06\\
248	2.41530094191589e-06\\
252	2.66638773194217e-06\\
256	2.83591286609486e-06\\
260	2.91609101832346e-06\\
264	2.90487978704679e-06\\
268	2.80513801498294e-06\\
272	2.62369595827039e-06\\
276	2.37041292794338e-06\\
280	2.0572758442361e-06\\
284	1.69757322709962e-06\\
288	1.3051641166717e-06\\
292	8.93850112371197e-07\\
296	4.76850724684397e-07\\
300	6.63770585163093e-08\\
304	3.26704180486637e-07\\
308	6.9312580042169e-07\\
312	1.02540410062968e-06\\
316	1.31796111468852e-06\\
320	1.56719373189951e-06\\
324	1.77149739122306e-06\\
328	1.93125110681258e-06\\
332	2.04876971099388e-06\\
336	2.12822843825844e-06\\
340	2.17556435311295e-06\\
344	2.19835861581252e-06\\
348	2.2057031801026e-06\\
352	2.20805518136302e-06\\
356	2.18352947201073e-06\\
360	1.94773623951594e-06\\
364	1.59103155095153e-06\\
368	1.17274572120074e-06\\
372	7.07283337120224e-07\\
376	2.14086850358619e-07\\
380	2.71998018885888e-07\\
384	7.09455984404528e-07\\
388	1.06195760796179e-06\\
392	1.30583828839015e-06\\
396	1.43187243817344e-06\\
400	1.44299631099631e-06\\
404	1.35029367875905e-06\\
408	1.16903249236522e-06\\
412	9.1568449056556e-07\\
416	6.06141843717545e-07\\
420	2.54923115949879e-07\\
424	1.24989399497011e-07\\
428	5.22013688941669e-07\\
432	9.25763627979137e-07\\
436	1.3268956141621e-06\\
440	1.7170416781731e-06\\
444	2.08881606467148e-06\\
448	2.43586385584305e-06\\
452	2.75292058815602e-06\\
456	3.03586094130281e-06\\
460	3.28172511632197e-06\\
464	3.48871995514699e-06\\
468	3.65619697006991e-06\\
472	3.78461174588285e-06\\
476	3.87546952118623e-06\\
480	3.93126116304821e-06\\
484	3.95539284436808e-06\\
488	3.95211193493505e-06\\
492	3.92643103175327e-06\\
496	3.8840516348726e-06\\
500	3.83128877134986e-06\\
};
\addlegendentry{True error}

\end{axis}
\end{tikzpicture}%

%% file: figures/fbc_final_ROM_error_compare_decrease.tex
%
%
\definecolor{mycolor1}{rgb}{0.00000,0.44700,0.74100}%
\definecolor{mycolor2}{rgb}{0.85000,0.32500,0.09800}%
\begin{tikzpicture}

\begin{axis}[%
width=\fwidth,
height=\fheight,
at={(0\fwidth,0\fheight)},
scale only axis,
xmin=0,
xmax=500,
xlabel style={font=\color{white!15!black}},
xlabel={Time (s)},
ymode=log,
ymin=1e-08,
ymax=0.01,
yminorticks=false,
ylabel style={font=\color{white!15!black}},
ylabel={Output error},
axis background/.style={fill=white},
legend style={legend cell align=left, align=left, draw=white!15!black, font = \scriptsize}
]
\addplot [color=blue, dashed, line width=1.0pt]
table[row sep=crcr]{%
	4	0.00167329541571962\\
	8	0.000247221508597389\\
	12	6.75413927257942e-05\\
	16	0.000138541364732635\\
	20	6.37781044407788e-05\\
	24	6.07427061128773e-05\\
	28	7.43340203407188e-05\\
	32	4.36009894766757e-05\\
	36	2.06066328608414e-05\\
	40	4.73675430577192e-05\\
	44	4.90618480845337e-05\\
	48	4.18615406170484e-05\\
	52	2.58250166425056e-05\\
	56	4.1500518312697e-05\\
	60	5.47629244838498e-05\\
	64	4.48835093727245e-05\\
	68	3.51301458053421e-05\\
	72	5.53633990887703e-05\\
	76	5.53235627337607e-05\\
	80	5.18856868087464e-05\\
	84	4.76179870483209e-05\\
	88	4.32526891575322e-05\\
	92	4.75604416630253e-05\\
	96	6.61908987922622e-05\\
	100	6.1665235162423e-05\\
	104	4.71755480831734e-05\\
	108	4.41947806151756e-05\\
	112	8.70661990083921e-05\\
	116	9.23622463543115e-05\\
	120	7.0075414068759e-05\\
	124	4.01051376445428e-05\\
	128	5.2942867303755e-05\\
	132	6.92377401323039e-05\\
	136	7.8466878336569e-05\\
	140	8.01077698989322e-05\\
	144	7.03182501655679e-05\\
	148	5.0093849867366e-05\\
	152	2.91539980955401e-05\\
	156	4.6291680485286e-05\\
	160	6.19719318140613e-05\\
	164	7.22198578031734e-05\\
	168	7.64469349769432e-05\\
	172	7.37210900635137e-05\\
	176	6.38856443851815e-05\\
	180	4.8244560131634e-05\\
	184	2.94518390751319e-05\\
	188	2.76341186504719e-05\\
	192	4.41830552570822e-05\\
	196	5.81686212527913e-05\\
	200	6.80003210802281e-05\\
	204	7.29304649300798e-05\\
	208	7.27147732341492e-05\\
	212	6.76290868747728e-05\\
	216	5.84346630113427e-05\\
	220	4.62817827924638e-05\\
	224	3.26280738769603e-05\\
	228	1.99854568114508e-05\\
	232	2.25068023320801e-05\\
	236	3.47332754151671e-05\\
	240	4.59029958800319e-05\\
	244	5.46551700383319e-05\\
	248	6.04466260739322e-05\\
	252	6.30757529701669e-05\\
	256	6.26014165052711e-05\\
	260	5.92939072627954e-05\\
	264	5.35963675878571e-05\\
	268	4.61023453038656e-05\\
	272	3.75597092717142e-05\\
	276	2.88986609197904e-05\\
	280	2.13198460112405e-05\\
	284	1.73179364234266e-05\\
	288	2.10480072621155e-05\\
	292	2.8894521119648e-05\\
	296	3.69964918991179e-05\\
	300	4.42203207429366e-05\\
	304	5.01015407330715e-05\\
	308	5.43775117151443e-05\\
	312	5.6884295233601e-05\\
	316	5.75290415958556e-05\\
	320	5.62936318456805e-05\\
	324	5.32753066329101e-05\\
	328	4.88021610863317e-05\\
	332	4.36423650513794e-05\\
	336	3.89043303522856e-05\\
	340	3.51346949562705e-05\\
	344	3.25905508260989e-05\\
	348	3.4959720133269e-05\\
	352	4.92797530354077e-05\\
	356	6.000960596221e-05\\
	360	5.9005212908758e-05\\
	364	5.61158130445941e-05\\
	368	5.23576907912899e-05\\
	372	4.92273327209505e-05\\
	376	4.69128778671654e-05\\
	380	4.43320919027903e-05\\
	384	4.09796612340236e-05\\
	388	3.69868860455074e-05\\
	392	3.27476870536385e-05\\
	396	2.88218276328237e-05\\
	400	2.59843525829878e-05\\
	404	2.47428984282178e-05\\
	408	2.43338775496791e-05\\
	412	2.37156587738331e-05\\
	416	2.2486833080611e-05\\
	420	2.06662723485731e-05\\
	424	1.85130726776302e-05\\
	428	1.65888302928716e-05\\
	432	1.58723765153806e-05\\
	436	1.67330563518392e-05\\
	440	1.83654234382312e-05\\
	444	2.01419872836585e-05\\
	448	2.17632067385058e-05\\
	452	2.303811379971e-05\\
	456	2.38059671598379e-05\\
	460	2.3919264643068e-05\\
	464	2.32446420095443e-05\\
	468	2.16703297791353e-05\\
	472	1.91232302378673e-05\\
	476	1.56203229908694e-05\\
	480	1.14802831368027e-05\\
	484	8.4605186451213e-06\\
	488	1.1263094550705e-05\\
	492	1.92595274998968e-05\\
	496	3.00298170910948e-05\\
	500	4.29315087137304e-05\\
};

\addlegendentry{Modified error indicator}
\addplot [color=red, line width=1.0pt]
  table[row sep=crcr]{%
4	9.91630219937534e-07\\
8	7.55023177867464e-08\\
12	2.00129468053145e-07\\
16	3.57283278795695e-07\\
20	2.09844299992046e-07\\
24	3.68840435749362e-07\\
28	7.520467777411e-07\\
32	1.1859923036428e-06\\
36	1.51923439717372e-06\\
40	1.65076150454579e-06\\
44	1.58767638114288e-06\\
48	1.43185583068028e-06\\
52	1.29739940246676e-06\\
56	1.23908635085712e-06\\
60	1.23707405452755e-06\\
64	1.22614395625664e-06\\
68	1.13969954484983e-06\\
72	9.43320252355662e-07\\
76	6.77921512703961e-07\\
80	3.9548503150133e-07\\
84	1.43517663975246e-07\\
88	4.2364383223692e-08\\
92	1.75467435359344e-07\\
96	5.01688136589351e-07\\
100	8.42958101210378e-07\\
104	9.75767589839194e-07\\
108	7.79793846716892e-07\\
112	3.33826974396878e-07\\
116	1.05901561076038e-07\\
120	2.48515939160754e-07\\
124	1.1593149562561e-07\\
128	2.81484244690766e-07\\
132	7.99141479546073e-07\\
136	1.21102389893935e-06\\
140	1.31852644680297e-06\\
144	1.02925658540531e-06\\
148	3.8604609631232e-07\\
152	4.52935555728118e-07\\
156	1.26845747261406e-06\\
160	1.84625339194611e-06\\
164	2.03317949054593e-06\\
168	1.76991608558108e-06\\
172	1.09656690605586e-06\\
176	1.35390302746785e-07\\
180	9.40843390306867e-07\\
184	1.94548447085552e-06\\
188	2.71166427778802e-06\\
192	3.11759578108806e-06\\
196	3.10093133104594e-06\\
200	2.66203215859573e-06\\
204	1.85777491668038e-06\\
208	7.88549373464953e-07\\
212	4.18539505941951e-07\\
216	1.6274118839199e-06\\
220	2.70934490820274e-06\\
224	3.55622284853307e-06\\
228	4.08974799170903e-06\\
232	4.26630982375276e-06\\
236	4.07774361166346e-06\\
240	3.54859343476566e-06\\
244	2.73072627243742e-06\\
248	1.696237190707e-06\\
252	5.29565990836645e-07\\
256	6.80357893356742e-07\\
260	1.84728284635671e-06\\
264	2.89363954997235e-06\\
268	3.75509478711678e-06\\
272	4.38354047105349e-06\\
276	4.74855109455774e-06\\
280	4.8374569179277e-06\\
284	4.65425867657299e-06\\
288	4.21765368874283e-06\\
292	3.55845692201484e-06\\
296	2.71668936435709e-06\\
300	1.73857641905251e-06\\
304	6.73657397309846e-07\\
308	4.27840550543124e-07\\
312	1.51725944130821e-06\\
316	2.54933963039772e-06\\
320	3.4836024409568e-06\\
324	4.28530016916007e-06\\
328	4.92597186707577e-06\\
332	5.38365955926601e-06\\
336	5.64284844495333e-06\\
340	5.69419696294649e-06\\
344	5.53411991455199e-06\\
348	5.16428131580593e-06\\
352	4.59104474948013e-06\\
356	3.82668431153643e-06\\
360	2.86627153578412e-06\\
364	1.77895517106119e-06\\
368	6.07400570395988e-07\\
372	6.05372350026556e-07\\
376	1.80163204033246e-06\\
380	2.9141700130797e-06\\
384	3.87443113281094e-06\\
388	4.62136361034649e-06\\
392	5.10948350407325e-06\\
396	5.31437170003812e-06\\
400	5.23463404422664e-06\\
404	4.89046168972163e-06\\
408	4.31966951530693e-06\\
412	3.57233921122191e-06\\
416	2.70509553979004e-06\\
420	1.77578492521757e-06\\
424	8.39036684241457e-07\\
428	5.70630592777732e-08\\
432	8.731432475928e-07\\
436	1.57965084546241e-06\\
440	2.15704550832196e-06\\
444	2.59528232515738e-06\\
448	2.8927864071937e-06\\
452	3.05511122145052e-06\\
456	3.09344862958749e-06\\
460	3.02312805922078e-06\\
464	2.86220917355973e-06\\
468	2.63023992830291e-06\\
472	2.34722213343641e-06\\
476	2.03280094457359e-06\\
480	1.70567400492594e-06\\
484	1.38320050524499e-06\\
488	1.08118007391411e-06\\
492	8.13765836826796e-07\\
496	5.93474431376251e-07\\
500	4.31257486055436e-07\\
};
\addlegendentry{True error}

\end{axis}
\end{tikzpicture}%

%% file: figures/burgers-low-viscosity-EIM-Adapt-greedy-convergence.tex
%
%
\begin{tikzpicture}
\begin{axis}[%
width=\fwidth,
height=\fheight,
at={(0\fwidth,0\fheight)},
scale only axis,
xmin=1,
xmax=16,
xlabel style={font=\color{white!15!black}},
xlabel={Iteration},
ymode=log,
ymin=0.00001,
ymax=100,
yminorticks=false,
ytick = {0.000001, 0.0001, 0.01, 1, 100},
ylabel style={font=\color{white!15!black}},
ylabel={Maximal error},
axis background/.style={fill=white},
legend style={legend cell align=left, align=left, draw=white!15!black, font = \scriptsize}
]
\addplot [color=red, line width=1.0pt]
  table[row sep=crcr]{%
1	303.77690470456\\
2	0.101779505732425\\
3	0.152342321275383\\
4	0.112731246617921\\
5	0.0942679364857335\\
6	0.0690171319035167\\
7	0.0574836063084119\\
8	0.0314269814140766\\
9	0.00145321643601173\\
10	0.0297306625555421\\
11	0.0124909092511437\\
12	0.0174187472957422\\
13	0.0203397108899014\\
14	0.0167471331506638\\
15	0.014141059220332\\
16	5.95741107320864e-05\\
};
\addlegendentry{Estimated error (Non-adaptive)}
\addplot [color=red, line width=1.0pt, mark=asterisk, mark options={solid, red}]
  table[row sep=crcr]{%
1	0.135169592265782\\
2	0.0417938169480999\\
3	0.01289076881716\\
4	0.00860297230957934\\
5	0.00470209626276879\\
6	0.0039205404206631\\
7	0.00251746747546189\\
8	0.00249851071662513\\
9	0.00288994364817099\\
10	0.00153863873883233\\
11	0.00097116707857313\\
12	0.000800306943065473\\
13	0.000563789731533654\\
14	0.000485185163644454\\
15	0.000371485764034632\\
16	0.000432534593848287\\
};
\addlegendentry{True error (Non-adaptive)}
\addplot [color=black, line width=1.0pt,mark=o, mark options={solid, black}]
  table[row sep=crcr]{%
1	80.5010090970679\\
2	0.382634486074283\\
3	0.273165704427768\\
4	0.0907667697149393\\
5	0.0629015862544381\\
6	0.0185406103144083\\
7	0.0412695614585471\\
8	0.0256121744356466\\
9	0.0228296640679302\\
10	0.000518181830883642\\
};
\addlegendentry{Estimated error (Adaptive)}
\addplot [color=black, line width=1.0pt, mark=x, mark options={solid, black}]
  table[row sep=crcr]{%
1	0.127368032832044\\
2	0.0621625171759098\\
3	0.0939483856361606\\
4	0.0037758669265859\\
5	0.00165161544821124\\
6	0.0026173198602807\\
7	0.00141465791569257\\
8	0.00109149013702179\\
9	0.000583264250107386\\
10	0.000642566025445847\\
};
\addlegendentry{True error (Adaptive)}
\addplot [color=blue, dashed, line width=1.0pt]
  table[row sep=crcr]{%
0	1e-03\\
1	1e-03\\
2	1e-03\\
3	1e-03\\
4	1e-03\\
5	1e-03\\
6	1e-03\\
7	1e-03\\
8	1e-03\\
9	1e-03\\
10	1e-03\\
11	1e-03\\
12	1e-03\\
13	1e-03\\
14	1e-03\\
15	1e-03\\
16	1e-03\\
};
\addlegendentry{Tolerance}
\end{axis}
\end{tikzpicture}%

%% file: figures/burgers-low-viscosity-EIM-Adapt-basis.tex
%
%
%
\definecolor{mycolor1}{rgb}{0.00000,0.44700,0.74100}%
\begin{tikzpicture}

\begin{axis}[%
width=\fwidth,
height=\fheight,
at={(0\fwidth,0\fheight)},
scale only axis,
xmin=0,
xmax=14,
xtick distance  = 3, 
xlabel style={font=\color{white!15!black}},
xlabel={Number of RB basis vectors},
ymin=0,
ymax=40,
ylabel style={font=\color{white!15!black}},
ylabel={Number of EIM basis vectors},
axis background/.style={fill=white}
]
\addplot [color=mycolor1, line width=1.0pt, mark=o, mark options={solid, mycolor1}]
  table[row sep=crcr]{%
1	1\\
5	6\\
7	12\\
8	18\\
9	23\\
10	27\\
11	31\\
12	35\\
13	38\\
14	40\\
};
\end{axis}
\end{tikzpicture}%

%% file: figures/burgers-low-viscosity-EIM-Adapt-effectivity.tex
%
%
\definecolor{mycolor1}{rgb}{0.00000,0.44700,0.74100}%
\definecolor{mycolor2}{rgb}{0.85000,0.32500,0.09800}%
\begin{tikzpicture}

\begin{axis}[%
width=\fwidth,
height=\fheight,
at={(0\fwidth,0\fheight)},
scale only axis,
xmin=3,
xmax=10,
xlabel style={font=\color{white!15!black}},
xlabel={Iterations},
ymin=0,
ymax=40,
ylabel={Effectivity},
axis background/.style={fill=white},
legend style={legend cell align=left, align=left, draw=white!15!black, font = \scriptsize}
]
\addplot [color=blue, line width=1.0pt]
  table[row sep=crcr]{%
1	633.408958336416\\
2	6.18981767994867\\
3	2.96565425846326\\
4	24.1625127742445\\
5	38.7147966234475\\
6	7.82805270868821\\
7	29.5153776266478\\
8	23.8520523856474\\
9	39.5030323482792\\
10	1.11505536178961\\
};
\addlegendentry{Original error indicator}
\addplot [color=gray,mark=o, mark options={solid, gray}, dashed, thick]
  table[row sep=crcr]{%
1	632.03464250109\\
2	6.1553891872089\\
3	2.90761467137576\\
4	24.0386569441444\\
5	38.0848861171423\\
6	7.08381524007536\\
7	29.1728205107049\\
8	23.4653283313501\\
9	39.1412024030737\\
10	0.806425815190119\\
};
\addlegendentry{Modified error indicator}
\end{axis}
\end{tikzpicture}%

%% file: figures/rb_burgers_basis_vs_iteration.tex
%
%
\begin{tikzpicture}

\begin{axis}[%
width=0.989\fwidth,
height=\fheight,
at={(0\fwidth,0\fheight)},
scale only axis,
xmin=0,
xmax=10,
xlabel style={font=\color{white!15!black}},
xlabel={Iterations},
ymin=0,
ymax=40,
ylabel style={font=\color{white!15!black}},
ylabel={Number of basis vectors},
axis background/.style={fill=white},
legend style={at={(0.06,0.704)}, anchor=south west, legend cell align=left, align=left, draw=white!15!black}
]
\addplot [color=red, mark=o, mark options={solid, red},line width=1.0pt]
  table[row sep=crcr]{%
1	1\\
2	5\\
3	7\\
4	8\\
5	9\\
6	10\\
7	11\\
8	12\\
9	13\\
10	14\\
};
\addlegendentry{RB}

\addplot [color=blue, mark=diamond, mark options={solid, blue},line width=1.0pt]
  table[row sep=crcr]{%
1	1\\
2	6\\
3	12\\
4	18\\
5	23\\
6	27\\
7	31\\
8	35\\
9	38\\
10	40\\
};
\addlegendentry{EIM}

\end{axis}
\end{tikzpicture}%

%% file: figures/rb_batchchrom_convergence.tex
%
%
\begin{tikzpicture}

\begin{axis}[%
width=\fwidth,
height=\fheight,
at={(0\fwidth,0\fheight)},
scale only axis,
xmin=0,
xmax=50,
ymode=log,
ymin=1e-07,
ymax=100,
yminorticks=true,
xlabel style={font=\color{white!15!black}},
xlabel={Iterations},
ylabel style={font=\color{white!15!black}},
ylabel={Maximal error},
axis background/.style={fill=white},
legend style={legend cell align=left, align=left, draw=white!15!black, font = \scriptsize}
]
\addplot [color=red, line width=1.0pt]
  table[row sep=crcr]{%
1	87.411428478925\\
2	9.5859757256063\\
3	6.69530290788166\\
4	5.27085419050464\\
5	33.024678480773\\
6	3.06945706802759\\
7	2.94141827599873\\
8	1.83475286364505\\
9	2.34197777697626\\
10	1.65959886233613\\
11	1.54850647123074\\
12	0.519809801410994\\
13	0.912206054660568\\
14	0.271874091124715\\
15	0.370221081599291\\
16	0.155650681681406\\
17	0.119906771067529\\
18	0.0895850116640843\\
19	0.0716792044969385\\
20	0.0537686767125179\\
21	0.0430601180885787\\
22	0.0304452268968945\\
23	0.0208408851407866\\
24	0.0157202451749827\\
25	0.0109309015218575\\
26	0.00815515060552503\\
27	0.00747120027008242\\
28	0.00610699801046521\\
29	0.00507856229281324\\
30	0.00417874816362928\\
31	0.00418809897744031\\
32	0.00213210226037804\\
33	0.00129116549004425\\
34	0.00110242678630164\\
35	0.00131625915470446\\
36	0.00123044607880877\\
37	0.0008139117774287\\
38	0.000581776968227853\\
39	0.000332419907639358\\
40	0.000389107628905937\\
41	0.000313668707412415\\
42	0.000130716394396128\\
43	0.000129022452845092\\
44	0.000151907744052459\\
45	0.000112370436555076\\
46	9.03731708687196e-05\\
47	3.74513370774222e-05\\
};
\addlegendentry{Estimated error (Non-adaptive)}

\addplot [color=red, line width=1.0pt, mark=asterisk, mark options={solid, red},mark repeat=5,mark phase=1]
  table[row sep=crcr]{%
1	0.0558648189483966\\
2	0.0690924671989514\\
3	0.0626298067674615\\
4	0.054258461773796\\
5	0.321323412924734\\
6	0.0300688337574755\\
7	0.0202551730098506\\
8	0.019105028260639\\
9	0.0140323543113437\\
10	0.0101626269426736\\
11	0.0081011918761833\\
12	0.00526174844728779\\
13	0.00503938218150131\\
14	0.00317274917024004\\
15	0.00348939446959\\
16	0.00211713106647054\\
17	0.00176061073244945\\
18	0.00134699221723103\\
19	0.00113157979236036\\
20	0.000815218186045026\\
21	0.000640768403952409\\
22	0.00051013074659473\\
23	0.000413211648852103\\
24	0.000304397208053587\\
25	0.000238094159147474\\
26	0.000257063565251324\\
27	0.00015721450410307\\
28	0.000113414044495436\\
29	9.84347477883903e-05\\
30	6.21846375569342e-05\\
31	4.72963822036443e-05\\
32	3.85359427923378e-05\\
33	3.77832742278706e-05\\
34	3.0698725148431e-05\\
35	1.44552389661545e-05\\
36	1.33505264565595e-05\\
37	1.65419679982889e-05\\
38	1.54003448533141e-05\\
39	8.88640042511611e-06\\
40	4.10005631145112e-06\\
41	3.43227588246279e-06\\
42	3.34143330554445e-06\\
43	2.80663155837909e-06\\
44	1.39793103071167e-06\\
45	1.09567483088763e-06\\
46	1.22854291965495e-06\\
47	6.41137036231179e-07\\
};
\addlegendentry{True error (Non-adaptive)}

\addplot [color=black, line width=1.0pt,mark=o, mark options={solid, black},mark repeat=5,mark phase=1]
  table[row sep=crcr]{%
1	36.4751084112045\\
2	2.39234317818396\\
3	1.08881744943141\\
4	0.472992961212231\\
5	0.180700050947384\\
6	0.0617150602934228\\
7	0.0350989008765347\\
8	0.0208668367697665\\
9	0.0120091137060169\\
10	0.00455827990054793\\
11	0.00580836812204089\\
12	0.00485114431395884\\
13	0.00365493588511888\\
14	0.00338701684685665\\
15	0.0022920098549094\\
16	0.00127456771410159\\
17	0.00115548195734827\\
18	0.00100617083236061\\
19	0.00105161427997767\\
20	0.000746500065866567\\
21	0.00053618695969019\\
22	0.000341818426497902\\
23	0.000315554245910868\\
24	0.000257899388693745\\
25	0.000152383397722482\\
26	9.79886691435597e-05\\
27	0.000128136974740462\\
28	0.000113930677553567\\
29	4.63441077213225e-05\\
};
\addlegendentry{Estimated error (Adaptive)}

\addplot [color=black, line width=1.0pt, mark=x, mark options={solid, black},mark repeat=5,mark phase=1]
  table[row sep=crcr]{%
1	0.0559249301523806\\
2	0.143761789100029\\
3	0.0605480000070894\\
4	0.00478166912681531\\
5	0.00202012145093904\\
6	0.00111873830706855\\
7	0.00066472117431702\\
8	0.000400242992442003\\
9	0.000197951747623708\\
10	0.000177982002915521\\
11	0.000109372848547047\\
12	9.09307187540545e-05\\
13	8.22205815856995e-05\\
14	5.51040949868249e-05\\
15	3.6590722304234e-05\\
16	3.93938262087215e-05\\
17	3.44148549079076e-05\\
18	3.09249215749041e-05\\
19	1.14481465709293e-05\\
20	1.43035481169895e-05\\
21	1.05827955971535e-05\\
22	1.59659458064439e-05\\
23	7.20591208130557e-06\\
24	4.48370946761534e-06\\
25	4.76003651562894e-06\\
26	3.37884928345059e-06\\
27	2.78743232208343e-06\\
28	2.18777296449301e-06\\
29	2.13374853600535e-06\\
};
\addlegendentry{True error (Adaptive)}

\addplot [color=blue, dashed, line width=1.0pt]
  table[row sep=crcr]{%
0	0.0001\\
1	0.0001\\
2	0.0001\\
3	0.0001\\
4	0.0001\\
5	0.0001\\
6	0.0001\\
7	0.0001\\
8	0.0001\\
9	0.0001\\
10	0.0001\\
11	0.0001\\
12	0.0001\\
13	0.0001\\
14	0.0001\\
15	0.0001\\
16	0.0001\\
17	0.0001\\
18	0.0001\\
19	0.0001\\
20	0.0001\\
21	0.0001\\
22	0.0001\\
23	0.0001\\
24	0.0001\\
25	0.0001\\
26	0.0001\\
27	0.0001\\
28	0.0001\\
29	0.0001\\
30	0.0001\\
31	0.0001\\
32	0.0001\\
33	0.0001\\
34	0.0001\\
35	0.0001\\
36	0.0001\\
37	0.0001\\
38	0.0001\\
39	0.0001\\
40	0.0001\\
41	0.0001\\
42	0.0001\\
43	0.0001\\
44	0.0001\\
45	0.0001\\
46	0.0001\\
47	0.0001\\
48	0.0001\\
49	0.0001\\
50	0.0001\\
};
\addlegendentry{Tolerance}

\end{axis}
\end{tikzpicture}%

%% file: figures/rb_batchchrom_basis_vs_iteration.tex
%
%
\begin{tikzpicture}

\begin{axis}[%
width=\fwidth,
height=\fheight,
at={(0\fwidth,0\fheight)},
scale only axis,
xmin=0,
xmax=30,
xlabel style={font=\color{white!15!black}},
xlabel={Iterations},
ymin=0,
ymax=50,
ylabel style={font=\color{white!15!black}},
ylabel={Number of basis vectors},
axis background/.style={fill=white},
legend style={at={(0.06,0.804)}, anchor=south west, legend cell align=left, align=left, draw=white!15!black}
]
\addplot [color=red, mark=o, mark options={solid, red},line width=1.0pt]
  table[row sep=crcr]{%
1	1\\
2	6\\
3	9\\
4	13\\
5	16\\
6	19\\
7	21\\
8	23\\
9	25\\
10	27\\
11	28\\
12	29\\
13	30\\
14	31\\
15	32\\
16	33\\
17	34\\
18	35\\
19	36\\
20	37\\
21	38\\
22	39\\
23	40\\
24	41\\
25	42\\
26	43\\
27	44\\
28	45\\
29	46\\
};
\addlegendentry{RB}

\addplot [color=blue, mark=diamond, mark options={solid, blue},line width=1.0pt]
  table[row sep=crcr]{%
1	1\\
2	8\\
3	14\\
4	18\\
5	22\\
6	26\\
7	29\\
8	31\\
9	33\\
10	35\\
11	36\\
12	38\\
13	40\\
14	41\\
15	42\\
16	42\\
17	42\\
18	42\\
19	43\\
20	44\\
21	45\\
22	46\\
23	47\\
24	48\\
25	48\\
26	48\\
27	48\\
28	49\\
29	50\\
};
\addlegendentry{DEIM}

\end{axis}
\end{tikzpicture}%

%% file: figures/rb_batchchrom_basis.tex
%
%
\definecolor{mycolor1}{rgb}{0.00000,0.44700,0.74100}%
\begin{tikzpicture}

\begin{axis}[%
width=\fwidth,
height=\fheight,
at={(0\fwidth,0\fheight)},
scale only axis,
xmin=0,
xmax=50,
xlabel style={font=\color{white!15!black}},
xlabel={Number of RB basis vectors},
ymin=0,
ymax=50,
ylabel style={font=\color{white!15!black}},
ylabel={Number of DEIM basis vectors},
axis background/.style={fill=white}
]
\addplot [color=mycolor1, line width=1.0pt, mark=o, mark options={solid, mycolor1}, forget plot]
  table[row sep=crcr]{%
1	1\\
6	8\\
9	14\\
13	18\\
16	22\\
19	26\\
21	29\\
23	31\\
25	33\\
27	35\\
28	36\\
29	38\\
30	40\\
31	41\\
32	42\\
33	42\\
34	42\\
35	42\\
36	43\\
37	44\\
38	45\\
39	46\\
40	47\\
41	48\\
42	48\\
43	48\\
44	48\\
45	49\\
46	50\\
};
\end{axis}
\end{tikzpicture}%

%% file: figures/rb_batchchrom_effectivity.tex
%
%
\begin{tikzpicture}

\begin{axis}[%
width=\fwidth,
height=\fheight,
at={(0\fwidth,0\fheight)},
scale only axis,
xmin=10,
xmax=30,
xlabel style={font=\color{white!15!black}},
xlabel={Iterations},
ymin=10,
ymax=120,
yminorticks=true,
ylabel style={font=\color{white!15!black}},
ylabel={Effectivity},
axis background/.style={fill=white},
legend style={legend cell align=left, align=left, draw=white!15!black, font = \scriptsize}
]
\addplot [color=blue, line width=1.0pt]
  table[row sep=crcr]{%
1	662.95133110984\\
2	18.5039209964369\\
3	19.6732659499281\\
4	102.921573823033\\
5	94.964673553759\\
6	61.3271188055916\\
7	57.4575414283611\\
8	59.6186652418763\\
9	69.8396643077187\\
10	30.7329462165295\\
11	60.1791587542626\\
12	59.9333585503026\\
13	54.9983301426585\\
14	76.3907588117508\\
15	82.903916874056\\
16	43.9120946596401\\
17	40.65871870267\\
18	38.6274345881148\\
19	107.750418532277\\
20	66.3186504441825\\
21	65.7952143156604\\
22	29.5371029035008\\
23	69.5633552458688\\
24	84.498611436885\\
25	50.5093199333684\\
26	47.2267555868434\\
27	67.0461007883404\\
28	72.1995937344703\\
29	36.1799693069412\\
};
\addlegendentry{Original error indicator}

\addplot [color=gray, mark=o, mark options={solid, gray}, line width=1.0pt]
  table[row sep=crcr]{%
1	652.215538076122\\
2	16.6410225774206\\
3	17.9827153548246\\
4	98.9179612114346\\
5	89.450092648334\\
6	55.1648762748956\\
7	52.8024414335795\\
8	52.1354206414751\\
9	60.666873872947\\
10	25.6109034951781\\
11	53.1061245930923\\
12	53.3498951776683\\
13	44.4528099245965\\
14	61.46579210976\\
15	62.6390983991095\\
16	32.3545041638887\\
17	33.5750930939642\\
18	32.5359218752921\\
19	91.8589112623767\\
20	52.1898524590473\\
21	50.6659090944184\\
22	21.4092187610924\\
23	43.7910208104698\\
24	57.5192015799606\\
25	32.0130732657515\\
26	29.0006037332007\\
27	45.9695375293229\\
28	52.0760972014155\\
29	21.7195732952134\\
};
\addlegendentry{Modified error indicator}

\end{axis}
\end{tikzpicture}%

%% file: figures/rb_batchchrom_effectivity_modified_dual.tex
%
%
\begin{tikzpicture}

\begin{axis}[%
width=\fwidth,
height=\fheight,
at={(0\fwidth,0\fheight)},
scale only axis,
xmin=10,
xmax=30,
xlabel style={font=\color{white!15!black}},
xlabel={Iterations},
ymin=0,
ymax=200,
ylabel style={font=\color{white!15!black}},
ylabel={Effectivity},
axis background/.style={fill=white},
legend style={legend cell align=left, align=left, draw=white!15!black, font = \scriptsize}
]
\addplot [color=red, line width=1.0pt]
  table[row sep=crcr]{%
1	652.215538076122\\
2	16.6410225774206\\
3	17.9827153548246\\
4	98.9179612114346\\
5	89.450092648334\\
6	55.1648762748956\\
7	52.8024414335795\\
8	52.1354206414751\\
9	60.666873872947\\
10	25.6109034951781\\
11	53.1061245930923\\
12	53.3498951776683\\
13	44.4528099245965\\
14	61.46579210976\\
15	62.6390983991095\\
16	32.3545041638887\\
17	33.5750930939642\\
18	32.5359218752921\\
19	91.8589112623767\\
20	52.1898524590473\\
21	50.6659090944184\\
22	21.4092187610924\\
23	43.7910208104698\\
24	57.5192015799606\\
25	32.0130732657515\\
26	29.0006037332007\\
27	45.9695375293229\\
28	52.0760972014155\\
29	21.7195732952134\\
};
\addlegendentry{GMRES}

\addplot [color=black, dashed, line width=1.0pt]
  table[row sep=crcr]{%
1	889.366330221234\\
2	25.9985512962407\\
3	8.44135357765642\\
4	116.243680597935\\
5	101.401988859615\\
6	76.653357084774\\
7	82.9244816421183\\
8	94.534981323532\\
9	77.9146064570617\\
10	46.1343091598055\\
11	112.721749868602\\
12	121.541043971549\\
13	127.644621075068\\
14	100.582534568214\\
15	81.1382955459658\\
16	136.843267924244\\
17	122.849526629856\\
18	126.608396038712\\
19	133.284366346205\\
20	103.770486285831\\
21	63.0874076624563\\
22	155.73501587315\\
23	137.922569512524\\
24	139.768352345566\\
25	172.670513604758\\
26	160.790448272321\\
27	94.3222651828708\\
28	96.0337280561118\\
29	75.8268755894993\\
};
\addlegendentry{Primal basis}

\end{axis}
\end{tikzpicture}%

%% file: figures/rb_batchchrom_rho.tex
%
%
\definecolor{mycolor1}{rgb}{0.00000,0.44700,0.74100}%
\begin{tikzpicture}
\begin{axis}[%
width=\fwidth,
height=\fheight,
at={(0\fwidth,0\fheight)},
scale only axis,
xmin=0,
xmax=30,
xlabel style={font=\color{white!15!black}},
xlabel={Iterations},
ymin=0,
ymax=35,
legend style={legend cell align=left, align=left, draw=white!15!black, font = \scriptsize},
axis background/.style={fill=white}
]
\addplot [color=mycolor1, line width=1.0pt]
  table[row sep=crcr]{%
1	29.2923369742646\\
2	13.3978542745742\\
3	12.6916308111929\\
4	32.4272176707767\\
5	16.8211283522849\\
6	13.4922796953572\\
7	12.1636517775328\\
8	9.54382330041538\\
9	9.32153646058191\\
10	5.93210224151474\\
11	9.6069587205629\\
12	8.63631633832155\\
13	8.50020617214473\\
14	8.10293710595995\\
15	5.2929832841918\\
16	5.60786402229974\\
17	5.30616181876286\\
18	5.3754621121365\\
19	8.04496119739764\\
20	7.82093545705042\\
21	6.96404274527278\\
22	5.46461972177572\\
23	6.82303779199179\\
24	6.92751801965331\\
25	4.4831237296741\\
26	3.95884769953258\\
27	3.72690733157676\\
28	4.99277129564942\\
29	2.3330489568211\\
};
\addlegendentry{$\bar{\rho}$}
\end{axis}
\end{tikzpicture}%

%% file: figures/rb_batchchrom_finalstep_org_vs_modified_effectivity.tex
%
%
\definecolor{mycolor1}{rgb}{0.00000,0.44700,0.74100}%
\definecolor{mycolor2}{rgb}{0.85000,0.32500,0.09800}%
\begin{tikzpicture}

\begin{axis}[%
width=\fwidth,
height=\fheight,
at={(0\fwidth,0\fheight)},
scale only axis,
xmin=0,
xmax=60,
xlabel style={font=\color{white!15!black}},
xlabel={Parameter},
ymin=10,
ymax=90,
ylabel style={font=\color{white!15!black}},
ylabel={Effectivity},
axis background/.style={fill=white},
legend style={at={(0.158,0.804)}, anchor=south west, legend cell align=left, align=left, draw=white!15!black,font = \scriptsize}
]
\addplot [color=mycolor2, line width=1.0pt]
  table[row sep=crcr]{%
1	20.9363837309275\\
2	23.9464932507964\\
3	22.0628481799135\\
4	22.195599053237\\
5	21.368561671288\\
6	22.4625966110759\\
7	20.353196746115\\
8	20.3639035744512\\
9	21.1473092379247\\
10	23.0085407476591\\
11	33.1972137485092\\
12	34.2522630180633\\
13	31.7089094728328\\
14	34.7778994442168\\
15	33.8382148372176\\
16	34.4002391502908\\
17	32.0357862290259\\
18	33.4269084003224\\
19	32.6339262519964\\
20	45.1760796909263\\
21	39.1899352355121\\
22	47.4685370921966\\
23	45.0807968130554\\
24	47.3178349698567\\
25	46.9567172908322\\
26	41.2168848728545\\
27	40.1176006495788\\
28	43.5819702140687\\
29	37.13286123778\\
30	57.4724022107915\\
31	52.1424393774862\\
32	53.094669007895\\
33	51.5910377570467\\
34	54.2683040014406\\
35	55.5336761425455\\
36	49.6281183815951\\
37	46.6047677197882\\
38	47.1085085443382\\
39	44.9781942353045\\
40	63.9131628201719\\
41	57.8344787955336\\
42	62.4735428789073\\
43	59.1283220177293\\
44	65.923080903364\\
45	60.0920386853799\\
46	55.2387987718571\\
47	53.3650898205037\\
48	50.1780864698514\\
49	45.9005041334575\\
50	64.8563059012102\\
51	63.3792027763765\\
52	77.712994472774\\
53	64.0021340649617\\
54	78.6723360103178\\
55	56.8283553728559\\
56	57.3153579143028\\
57	54.3743259961959\\
58	54.8998336990903\\
59	52.6608607250303\\
60	82.3844549617804\\
};
\addlegendentry{Original error indicator}
\addplot [color=mycolor1, mark=o,   mark options={solid, mycolor1}, line width=1.0pt]
table[row sep=crcr]{%
	1	12.4660090772373\\
	2	15.1904170917927\\
	3	14.3470412991402\\
	4	15.2337376426312\\
	5	15.286088337456\\
	6	16.0167604848782\\
	7	14.4175328043209\\
	8	13.4262366590574\\
	9	15.0905695217059\\
	10	16.1562920286648\\
	11	22.645222622271\\
	12	22.4797974834681\\
	13	21.5131218473015\\
	14	23.7582532669856\\
	15	22.885177893084\\
	16	23.4169700075061\\
	17	21.4363040818648\\
	18	20.5278108602351\\
	19	21.3979194210706\\
	20	29.2622192656095\\
	21	22.1649416514049\\
	22	30.2404974320502\\
	23	28.6930999335704\\
	24	29.6234081382187\\
	25	29.3680014267219\\
	26	26.0099869747425\\
	27	24.481974539916\\
	28	27.9913554243511\\
	29	24.1202836081011\\
	30	37.4142317973242\\
	31	35.2203579562149\\
	32	33.3626536146913\\
	33	31.5574385563981\\
	34	34.6741582567486\\
	35	33.9188196066133\\
	36	30.5514296108629\\
	37	31.9778797118479\\
	38	31.0875922635319\\
	39	27.9675967070199\\
	40	41.2729461741262\\
	41	37.3430719355332\\
	42	32.5657042324798\\
	43	37.2646210780018\\
	44	38.3148722569403\\
	45	34.7203551397222\\
	46	34.4119362405988\\
	47	31.5888588234414\\
	48	32.3650021095165\\
	49	28.8557024857029\\
	50	39.3441741987754\\
	51	35.7657381586819\\
	52	39.5794228378191\\
	53	38.9136466347051\\
	54	47.5582646428039\\
	55	38.4851889460809\\
	56	37.3635570152271\\
	57	29.586706201684\\
	58	38.0202526183584\\
	59	31.2825314873655\\
	60	51.3809611182213\\
};
\addlegendentry{Modified error indicator}
\end{axis}
\end{tikzpicture}%